\documentclass[a4paper,12pt,oneside,reqno]{amsart}

\usepackage{enumitem}
\usepackage{hyperref}
\usepackage[headinclude,DIV13]{typearea}
\areaset{15.1cm}{25.0cm}
\usepackage{amsfonts,amssymb,amsmath,amsthm,bbm}
\usepackage{xcolor}
\usepackage[latin1] {inputenc}
\usepackage{graphicx, psfrag}
\usepackage{color}
\usepackage{bold-extra}

\usepackage{subscript}

\newtheorem{theorem}{Theorem}[section]
\newtheorem{lemma}[theorem]{Lemma}
\newtheorem{proposition}[theorem]{Proposition}
\newtheorem{corollary}[theorem]{Corollary}

\theoremstyle{definition}

\theoremstyle{remark}
\newtheorem*{remark}{Remark}

\def\paragraph#1{\noindent \textbf{#1}}

\numberwithin{equation}{section}

\def\Var{\mathop{\rm Var}\nolimits}
\def\Cov{\mathop{\bf Cov}\nolimits}

\def\d{\mathrm{d}}

\def\<{\langle}
\def\>{\rangle}

\def\a{\alpha}
\def\b{\beta}
\def\e{\epsilon}

\def\g{\gamma}
\def\k{\kappa}
\def\l{\lambda}

\def\s{\sigma}

\def\z{\zeta}
\def\o{\omega}

\def\O{\Omega}
\def\S{\Sigma}

\def \D {{\mathbb D}}

\def\R{{\mathbb R}}  
\def\N{{\mathbb N}}  
\def\P{{\mathbb P}}   
\def\Z{{\mathbb Z}}  
  
\def\C{{\mathbb C}}  
\def\E{{\mathbb E}}

\let\cal=\mathcal
\def\AA{{\cal A}}
\def\BB{{\cal B}}
\def\CC{{\cal C}}
\def\DD{{\cal D}}
\def\EE{{\cal E}}
\def\FF{{\cal F}}

\def\HH{{\cal H}}
\def\II{{\cal I}}
\def\JJ{{\cal J}}

\def\LL{{\cal L}}
\def\MM{{\cal M}}

\def\OO{{\cal O}}
\def\PP{{\cal P}}

\def\SS{{\cal S}}

 \def \k {{\kappa}}
 \def \b {{\beta}}
\def \e {{\epsilon}}
 \def \s {{\sigma}}
 
 \def \z {{\zeta}}

 \def \g {{\gamma}}
 \def \l {{\lambda}}
 \def \d {{\delta}}
 \def \a {{\alpha}}
 \def \o {{\omega}}
 \def \O {{\Omega}}

\newcommand{\rand}{{\rm rand}}

\newcommand{\Prob}{\mathbb{P}}

 \def \ba {\begin{array}}
 \def \ea {\end{array}}

 \newcommand{\be}{\begin{equation}}
 \newcommand{\ee}{\end{equation}}

\newcommand{\bea}{\begin{eqnarray}}
 \newcommand{\eea}{\end{eqnarray}}
\def\TH(#1){\label{#1}}\def\thv(#1){\ref{#1}}
\def\Eq(#1){\label{#1}}\def\eqv(#1){(\ref{#1})}

\def\var{\hbox{\rm Var}}
\def\sfrac#1#2{{\textstyle{#1\over #2}}}

 \def \1{\mathbbm{1}}

\def\wh{\widehat}

\def\Leb{{\rm{Leb}}}

\begin{document}
\title{Integral Means Spectrum for the Random Riemann Zeta~Function}
\author[B. Duplantier]{Bertrand Duplantier}
\address{Bertrand Duplantier\\ Universit\'e Paris-Saclay, CEA, CNRS, Institut de Physique Th\'eorique, 91191 Gif-sur-Yvette Cedex, France}
\email{bertrand.duplantier@ipht.fr}
\author[V. Gayrard]{V\'eronique Gayrard}
 \address{
V\'eronique Gayrard\\ Aix Marseille Univ, CNRS, I2M, Marseille, France
}
\email{veronique.gayrard@math.cnrs.fr}
\author[E. Saksman]{Eero Saksman}
\address{Eero Saksman\\University of Helsinki, Department of Mathematics and Statistics, PO Box 68 (Pietari Kalmin katu 5), 00014 University of Helsinki, Finland}
\email{eero.saksman@helsinki.fi}

\date{\today}

\begin{abstract} 
We study the integral means spectrum associated with the analytic function whose derivative is the so-called  randomized Riemann zeta-function, introduced some time ago by Bagchi. The randomized $\z$-function,  $\z_{\mathrm{rand}}(\s+ih)$, is known to represent the asymptotic statistical behaviour  of the random vertical shifts of the actual $\z$-function in the critical strip, $1/2 <\s\leq 1, h\in \mathbb R$, and appears in a number of recent works on the asymptotic behavior of the moments and maxima of the $\z$-function on short intervals along the critical axis $\s=1/2$. 
Using probability and basic analytic number theory, we show that the complex integral means spectrum of the primitive of $\z_{\mathrm{rand}}$ is almost surely of the form conjectured 30 years ago by Kraetzer, for the so-called universal integral means spectrum of univalent functions in the disc. The Riemann $\z$-function and its random version have recently been rigorously related to the so-called Gaussian multiplicative chaos (GMC), initiated by Kahane 40 years ago.  In the case of the holomorphic multiplicative chaos on the unit disc -- an important stochastic object closely related to Liouville quantum gravity on the unit circle -- we prove that the integral means spectrum    of the primitive is almost surely also of the same Kraetzer form. However, we establish that neither the primitive of the random function  $\z_{\mathrm{rand}}$, nor that of the holomorphic GMC are injective. Building on earlier work by one of the authors and Webb on the convergence of Riemann $\z$-function on the critical line to a holomorphic GMC distribution, we finally provide an alternative derivation of the integral means spectrum for the random  Riemann $\z$-function.
\end{abstract}

\thanks{B.~Duplantier is grateful to  Michel Zinsmeister for his early collaboration on this subject, and we thank Christian Webb for valuable discussions. B.~Duplantier and V.~Gayrard would like to thank the Centre International de Rencontres Math\' ematiques (CIRM) at Luminy and the Institute for Applied Mathematics at the University of Bonn for their hospitality during the writing of this paper. B.~Duplantier and E. Saksman would like to thank the Hausdorff Institute for Mathematics for its hospitality during the 
Trimester Program ``Probabilistic methods in quantum field theory". Funding for all their stays was provided by the Deutsche Forschungsgemeinschaft (DFG, German Research Foundation) under Germany's Excellence Strategy -- EXC-2047/1 -- 390685813.  In addition, E. Saksman was  supported by the Finnish Academy Center of Excellence FiRST}
\maketitle


\section{Introduction}
    \TH(S1)
    
 \subsection{Integral means spectrum} \TH(S0.0)
Twenty years ago, Peter W. Jones published a nice essay in honor  of Lennart Carleson's 75th birthday, entitled ``{\it On Scaling Properties of Harmonic Measure}" \cite{10.1007/3-540-30434-7_7}, where he discussed in particular various conjectures  concerning the fine structure of harmonic measure in the plane. The first conjectures he discussed concerned conformal mappings from the unit disc, $h: \D\to \Omega$, where $\Omega$ is a bounded planar domain.  Let $\beta$ be a complex number, and consider the integral means of the growth of the modulus of the $\beta$th power of the derivative, $|h'(z)^\b|$. The \emph{integral means spectrum} associated with $h$ is then defined as \cite{MR1629379,Bin97}
\be \Eq(-1.0)
b_h(\b) := \underset{r\to 1^{-}}{\limsup} 
\frac{\log \int_{r\partial \D} |h'(z)^\b| |dz|}{\log\left(\frac{1}{1-r}\right)}.
\ee
When the limit exists, one has the asymptotic behavior,
\be \Eq(-1.1)
\int_{r\partial \D} |h'(z)^\b| |dz|\asymp \left(1-r\right)^{-b_h(\b)}, \quad r\to 1^{-},
\ee
in the sense of the equivalence of the logarithms. 

When the conformal mapping $h$ is \emph{random}, one is usually lead to first define an \emph{average} integral means spectrum, where one takes the expectation of the l.h.s. of \eqv(-1.1), and the question naturally arises of the comparison of the average spectrum and of the almost sure one.   For instance, the \emph{average} integral means 
spectrum of the celebrated Schramm-Loewner evolution \cite{OS} is given, for  the bounded version of whole-plane SLE$_\kappa$, by  the convex function for real $\b$ \cite{Duplantier00,BS,MR3638311}
\bea \Eq(-1.2)
b(\b,\kappa)-\b+1&=& 1+2\tau -2\sqrt{b\tau}, \quad \tau=d(\kappa)-\b,
\\ \label{bk}
d(\kappa) &=&\frac{(4+\kappa)^2}{8\kappa},\quad \quad \,\,\,\quad \b\in [\b_1,\b_2],
\eea
 with $\b_1=-1-\frac{3}{8}\kappa, \b_2=\frac{3}{4}d(\kappa)$. Outside of that interval, an average spectrum associated with the SLE \emph{tip}  exists  for $\b\leq \b_1$, 
whereas the  average spectrum becomes \emph{linear}  for $\b\geq \b_2$ \cite{Duplantier00,PhysRevLett.88.055506,BS,MR3638311}. The \emph{almost sure tip spectrum} was first obtained in  \cite{zbMATH06126650}, whereas the almost sure version of the SLE \emph{bulk spectrum} was finally established in   \cite{gwynne2018}, using the so-called \emph{imaginary geometry} of Miller and Sheffield. The \emph{a.s.} bulk spectrum was found to be identical to the average one   \eqv(-1.2), except that its transition to a linear spectrum happens before $\b_2$, exactly at the point where the intersection of its tangent with the vertical axis leaves the $[0,-1]$ interval, {\it i.e.}, Makarov's criterion for a spectrum to be that of an actual conformal map \cite{MR1629379}. An \emph{a.s. boundary spectrum} is also established in  \cite{zbMATH06571699}, further extended in \cite{zbMATH07250114}.

\subsection{Universal spectra}
   \TH(S0.2) 
The so-called universal spectra are obtained when one considers the supremum of all conformal maps over bounded (vs. unbounded), simply connected planar domains. In particular, the so-called \emph{pressure spectrum}  $\mathrm B(\b)$ is defined as the supremum over univalent holomorphic functions $h$ from $\mathbb D$ to \emph{bounded} domains $\Omega$ as 
\be\nonumber
\mathrm{B}(\b)=\underset{\Omega}{\sup} \,b(\b),\quad \Omega\, \,\textrm{bounded},
\ee
whereas a universal spectrum $\mathrm B_{\bullet}(\b)$ is similarly defined  for \emph{unbounded} domains. 

For real parameter $\b$, a well-known result from Makarov \cite{MR1629379} is the following relation between the bounded vs. unbounded spectra,
\be \label{Makbu}
\mathrm B_{\bullet}(\b)=\max\{\mathrm B(\b),3\b-1\},\quad \b\in \mathbb R.
\ee
A number of  exact results and outstanding conjectures are known for the universal integral means spectra.  
(See, {\it e.g.}, the detailed survey by Hedenmalm and Sola \cite{MR2419488}, and the treatise by Garnett and Marshall \cite{MR2450237}. See also Section \ref{App} for a more detailed historical perspective.)

\subsection{Kraetzer Conjecture} 
All known conjectures are encompassed by the well-known \emph{Kraetzer conjecture} \cite{MR1427159} for the universal spectrum, 
\be\label{K}
\mathrm B_{\mathcal K}(\b)=\frac{\b^2}{4}, \quad \b\in [-2,2],
\ee
which is supplemented  for $|\b|\geq 2$ by the  linear integral means spectrum, 
\begin{equation} \label{t-1}
\mathrm B_{\mathcal K}(\b)=|\b|-1,\quad |\b|\geq 2,
\end{equation}
a classical result for the universal spectrum in the case $\b\geq 2$   \cite{MR1217706}.

For complex values of $\b$, this conjecture was extended by I. Binder \cite{Bin09}, under the simple form,
\be\label{Kc}
{\mathrm B}_{{\mathcal K}}(\b)=
\begin{cases}\frac{1}{4}|\b|^2,  \,\,\,t\in \mathbb C, \,\,\,|\b| \leq 2,&\\ 
|\b|-1,\quad\quad \quad |\b| \geq 2.&
\end{cases}
\ee

\medskip
As shown in Section \ref{App}, the SLE  integral means spectrum \eqv(-1.2) satisfies Kraetzer's bounds. However, despite all the explicit multifractal properties exhibited by the conformally invariant SLE processes, this family did not allow for sharp bounds in the problem of universal multifractal spectra. This situation was 
especially reflected  upon by  P. Jones in 2005  \cite{10.1007/3-540-30434-7_7}, when he wrote the following lines: 

\emph{``It was understood in the 1990's that dynamical systems and Conformal Field Theory were intimately tied to these problems, at least in dimension two. It is a curious state of affairs that CFT as it is understood today, makes no notable predictions for these problems. This is perhaps even more surprising in light of the great success CFT has had with SLE. Exactly how the full story will unfold remains mysterious. My guess is that we are missing fundamental concepts, and that the future will bring some unsuspected, closer connections between analysis and statistical physics.''}

Since that time, it does not seem that much more knowledge was gained. However, we mention here the interesting work by Hedenmalm \cite{MR3649242}, where
for a holomorphic function $g: \mathbb D \mapsto \mathbb C$, he defined the function $\tilde b_g: \mathbb C \to [0,\infty]$, as 
\be \Eq(eden)
\tilde b_g(t) := \underset{r\to 1^{-}}{\limsup} 
\frac{\log \int_{r\partial \D} \big|e^{\b g(z)}\big| |dz|}{\log\left(\frac{1}{1-r}\right)},
\ee
which he called the ``exponential type spectrum'' of the (zero-free) function $e^g$.  In the case of  a function $g$ belonging to the Bloch space, and which is the Bergmann projection of a  bounded function $\mu$ on $\mathbb D$,  with norm $||L_\infty(\mu)||\leq 1$, he showed that its exponential spectrum \eqv(eden) is bounded above by the Binder-Kraetzer spectrum \eqv(Kc), with $\tilde  b_g(\b)\leq {\mathrm B}_{\mathcal K}(\b)$ for all $\b\in \mathbb C$. 

\subsection{Relation to the random Riemann $\z$-function and to the Random-Energy Model} In the present work, we shall try and address some of Peter Jones' queries. We present an explicit (random) function, whose integral means spectrum is (almost surely) exactly of the Kraetzer extended form in the complex setting \eqv(Kc). It is intimately related to the \emph{Riemann zeta-function}, more precisely to the \emph{primitive} of a \emph{random} version of the latter, introduced some time ago by B. Bagchi \cite{Bag81}.  

This randomized $\z$-function, $\z_{\mathrm{rand}}(\s+ih)$, is known to represent the asymptotic statistical behavior of the shifts of the  actual Riemann function $\z(\s+ih)$ in the critical strip, $1/2 <\s\leq 1, h\in \mathbb R$, for large values of $h$  \cite{MR1555301,MR1555343,Bag81,MR4041109}, and appears in a number of recent works \cite{HLS,H13,Harper19,ABH,AT19,MR3906393,SW,MR4441506,MR4718391}, in relation to the statistics of log-correlated random variables (see, {\it e.g.}, \cite{LPA17}).  Its natural truncations provided a test bed \cite{H13,ABH} for studying the asymptotic behavior of the moments and maxima of the $\z$-function on short intervals along the critical axis $\s=1/2$, a prominent question initiated by the Fyodorov-Hiary-Keating conjecture \cite{PhysRevLett.108.170601,MR3151088}, which has recently been the focus of a wealth of studies and witnessed spectacular progress \cite{MR3851835,MR3911893,MR4348686,2020arXiv200700988A,2023arXiv230700982A}. 
 
 Another domain which has seen important advances is that of the so-called \emph{Gaussian multiplicative chaos} (GMC), initiated by Kahane in 1985 \cite{K}, and recently reborn with the application of probabilistic methods to Liouville quantum gravity (LQG) \cite{DS,rhodes-2008}. Rigorous connections between probabilistic number theory and GMC appeared recently. It has been shown in \cite{SW} that for $\o$ uniformly distributed on $[0,1]$, the Riemann function on the critical line, $\z(\frac{1}{2}+i\o T+ih)$, converges as $T\to \infty$ to a random generalized function of $h$,  the non-smooth part of which is  precisely a holomorphic Gaussian multiplicative chaos distribution.  The  \emph{holomorphic multiplicative chaos} on the unit disc, first thoroughly studied  in \cite{MR4597318}, proved to be an important stochastic object intimately related to LQG on the unit circle.
 
 The moments  of the $\z$-function in a short interval  obtained  in \cite{MR4348686}, as well as those of the random model $\z_{\mathrm{rand}}$ studied in \cite{AT19}, exhibit the scaling behavior and the so-called \emph{freezing transition}  conjectured in \cite{PhysRevLett.108.170601,MR3151088}:  the event,
\be\int_{[-1,1]}\left|\zeta\left (\frac{1}{2}+it+ih\right)\right |^{\b} dh= 
\begin{cases}(\log T)^{\b^2/4+\mathrm{o}(1)},&\quad \quad \mathrm{for}\,\,\, \b\leq 2,\\
                     (\log T)^{\b -1+\mathrm{o}(1)},&\quad \quad \mathrm{for}\,\,\, \b > 2,
                      \end{cases}
\label{FHK}\ee
 with the  random shift parameter $t$ chosen uniformly in $[T,2T]$, has probability $1-\mathrm{o}(1)$ as $T\to \infty$. Naturally, one can not stay unfazed by the similarity of the scaling exponents in \eqv(FHK) to those of the Kraetzer conjecture seen above. Another similarity of those exponents and of the freezing-transition at $\b=2$, already mentioned in \cite{PhysRevLett.108.170601,MR3151088,LPA17,AT19,MR4348686}, is that with the free energy and the  freezing transition of the \emph{Random Energy Model}, invented by B. Derrida in 1980 as the simplest possible model of a spin-glass \cite{De1} and rigorously studied in \cite{OP84}. 
 
 \subsection{Statement of main results} \TH(S1.0)
 Let $\mathcal P$ be the set of prime numbers, and $(\theta_p$, $\text{$p \in \mathcal P$})$ a collection of i.i.d.~random variables uniformly distributed on $[0,1]$, so that $(U_p=e^{2\pi i \theta_p}$,  $\text{$p \in \mathcal P$})$ is uniformly distributed on the unit circle. The random zeta-function is defined in the half-plane, $s=\s+ih, \s>1/2, h\in \R$, as,
\be
\z_{\mathrm{rand}}(s):=\prod_{p\in \mathcal P}\left(\frac{1}{1-p^{-s}U_p}\right).
 \Eq(0.0bis)
\ee
The variables $(U_p$, $\text{$p \in \mathcal P$})$ are defined on a common probability space, $(\O,\FF,\P)$. Expectation with respect to $\P$ is denoted by $\E$. We also define the primitive $F$ of the randomized zeta function by setting
\be\Eq(eq:F)
F(s):=\int_1^s \z_{\mathrm{rand}}(z)dz,
\ee
for $\sigma>1/2.$

For $\b\in \C$, we are interested in the complex moments of its derivative, $F'(s)=\z_{\mathrm{rand}}(s)$, in a finite interval. 
The first result of this section is the following. 
\begin{theorem}
    \label{th:zeta}
$ \P$-almost surely, it holds for all $\b \in \mathbb C$ that
\be
\lim_{\s\rightarrow{\frac{1}{2}}^+}\frac{\log\int_0^1\left|\big(\z_{\mathrm{rand}}(\s+ih)\big)^\b\right|dh}{\log\left[(\s-\frac{1}{2})^{-1}\right]}=f(\b),
\Eq(1.theo1.2bis)
\ee
 where 
\be
f(\b) = 
\begin{cases}
\frac{1}{4}|\b|^2& \text{if $|\b|\leq 2$,} \\
|\b|-1 &  \text{if $|\b|\geq 2$.}
\end{cases}
\Eq(1.theo1.3bis)
\ee
The convergence also takes place  in $L^q(\O,\FF,\P)$ for any $0\leq q<\infty$.
\end{theorem}

We also consider a unit disc analogue of the random model above in the right half-plane, which is related to a canonical model of holomorphic multiplicative chaos \cite{MR4597318}. Thus, consider the random analytic function $F$ on $\D:=\{ z\in\C\, :\,|z|<1\}$, defined via the derivative
\be\Eq(defF)
F'(z):=\exp\left(\sum_{n=1}^\infty \frac{G_nz^n}{\sqrt{n}}\right)\qquad \textrm{and}\;\; F(0)=0.
\ee
Here the variables $G_n=(V_n-iW_n)/\sqrt{2}
$ are independent  standard complex Gaussians. Writing $z=re^{i\theta}$ we may write
$$
|F'(z)|=\exp(U(z)),
$$
with
$$
U(z)=U(re^{i\theta}):=\log |F'(re^{i\theta})|=\sum_{n=1}^\infty\frac{r^n}{\sqrt{2n}}\big(V_n\cos(n\theta)+W_n\sin(n\theta)\big).
$$
Thus $U$ is a random harmonic function on $\D$, and its restriction on the circle of radius $r<1$ is the natural approximation of the  log-correlated random distribution on the unit circle,
\be\Eq(eq:chaos)
X(\theta):= \sum_{n=1}^\infty\frac{1}{\sqrt{2n}}\big(V_n\cos(n\theta)+W_n\sin(n\theta)\big).
\ee
We then have the following result for the integral moments of $F'$.

\begin{theorem}
\label{th:prob}
$\mathbb P$-almost surely, it holds for all $\b \in \mathbb C$ that
\be
\lim_{r\to 1^-} \frac{\log\Big( \int_0^{2\pi}\big| \left(F'(re^{i\theta})\right)^{\b}\big| d\theta\Big)}{\log \big[(1-r)^{-1}\big]}=f(\b).
\Eq(1.theo2.1)
\ee
\end{theorem}
The proof of our first main result,  Theorem \ref{th:zeta}, draws on the work of  \cite{ABH} and \cite{AT19}. It relies on the \emph{multiscale decomposition}  of the random process, using the prime number theorem and Kistler's multiscale second moment method  \cite{K15}, in a practically self-contained way. The proof of the second result, Theorem \ref{th:prob}, rests on the basic properties of natural approximations of the associated  Gaussian multiplicative chaos  \eqv(eq:chaos). 
It is then suggestive of a path  to an alternative proof of Theorem \eqref{th:zeta} in the half-plane setting, which rests on the use of the Gaussian approximation of the randomized $\z$-function  established in \cite{SW}. 

Given the appearance of Kraetzer's integral means spectrum both in Theorem \ref{th:zeta} and Theorem \ref{th:prob}, a crucial question concerns the \emph{injectivity} properties of the primitive of $\z_{\mathrm{rand}}(\s+ih)$ \eqv(0.0bis)  in the critical strip, $1/2<\s\leq 1$, as well as those in the unit disc of the primitive $F$ of \eqv(defF). Using the Becker-Pommerenke injectivity criterion, we (unfortunately) answer to the negative.
\begin{proposition}\label{le:1}
 Let $a >0$ and let $Q:=Q_a:=(1/2, 1/2+2a)\times (-a,a)$ be a small square with the right side on the critical line $\s=1/2$. Then, almost surely the primitive $F$ \eqv(eq:F) of $\z_{\mathrm{rand}}$ \eqv(0.0bis) is not injective on $Q$. A fortiori, it is not injective in any neighborhood (to the right) of the critical line.
\end{proposition}
A similar result naturally holds in the disc setting.
\begin{proposition}\label{le:2}
Almost surely, the primitive $F$ of the random holomorphic function \eqv(defF) is not injective in the unit disc $\mathbb D$.
\end{proposition}

The function appearing in \eqv(1.theo1.3bis) is well known in statistical mechanics as the free energy of the Random Energy Model (REM) \cite{De1}.  Often called the simplest model of a spin glass, the REM is a collection, $(H_i, 1\leq i\leq N)$, of $N$ independent and identical Gaussian random variables with mean zero and variance $(\ln N)/2$. Given $\b>0$ and  defining the partition function as  $Z_{\b, N}=N^{-1}\sum_{i=1}^{N}e^{\b H_i}$, we have\footnote{We chose the parameters so that the limit in  \eqv(REM.1) matches the function $f(\b)$ in  \eqv(1.theo1.3bis). In the classical REM, a collection of $N:=2^M$ centred Gaussians of variance $M$ is considered, and $Z_{\b, N}=\sum_{i=1}^{N}e^{\b H_i}$.}
\be
\lim_{N\rightarrow \infty}\frac{1}{\ln N}\ln Z_{\b, N}
= \max_{x\in[0,1]}\left\{\b x-x^2\right\}
= 
\begin{cases}
\frac{1}{4}\b^2& \text{if $\b\leq 2$,} \\
\b-1 &  \text{if $\b\geq 2$.}
\end{cases}
\Eq(REM.1)
\ee
Convergence holds both almost surely and in mean \cite{OP84}.
The non-analyticity of the limiting function at $\b=2$ signals a so-called phase transition between a high-temperature phase ($\b\leq 2$) and a low temperature spin-glass phase ($\b\geq 2$). The form of the limit \eqv(REM.1) arises from the the properties of the extreme order statistics of the family $(H_i, 1\leq i\leq N)$: when $\b\leq 2$, the sum $Z_{\b, N}$ is dominated by the large number of its typical increments, whereas when $\b\geq 2$, it is dominated by the small number of its extreme increments.

In the last decades, the REM behaviour has proved to have some degree of universality, in that all models with extreme order statistics that are, to leading order, the same as those of i.i.d.~Gaussian r.v.'s, are expected to have the same log-partition function as the REM. This encompasses the class of log-correlated processes.
Classical examples of these processes include Branching Brownian Motion (BBM) and its close relative, Branching Random Walk (which models random polymers on trees), as well as the two-dimensional Gaussian Free Field \cite{LPA17}, \cite{MR3151088}.

The REM has also been studied in the case of complex $\b$ \cite{De91}, \cite{KaKli}. There, a richer picture emerges. The limit of the log-modulus of the partition function of the complex REM takes on three different forms in three distinct regions of the plane: one shaped like an eye with an inscribed circle, and two vertical bands. Again, a similar pattern emerges from complex log-correlated processes, as established in  \cite{HaKli} for the complex BBM, and in \cite{arXiv:1905.12027,JSW19} for more general fields.

This article is organized as follows. In Section \ref{S2} we gather a number of key probability estimates related to the moments of the random zeta function \eqref{0.0bis}. At their root lies the prime number theorem and its use to control short and large distance correlations. Section \ref{S3} is devoted to the probabilistic proof of Theorem \ref{th:zeta}. Section \ref{S4} provides the proofs of Propositions \ref{le:1} and \ref{le:2} concerning the non-injectivity of the primitives of the random functions \eqref{0.0bis} and \eqref{defF}. The relation to complex Gaussian multiplicative chaos is developed in Section \ref{S5}, first with the proof of Theorem \ref{th:prob} for the holomorphic chaos \eqref{defF} in the disc. An alternative GMC approach to Theorem \ref{th:zeta} is finally presented. Section \ref{App} serves as an Appendix offering a  historical perspective on the multifractal analysis of harmonic measure in the plane, and the various and celebrated conjectures pertaining to it.



\section{Preliminaries}
    \TH(S2)

\subsection{Notations and preliminary observations}
We denote by $\mathcal P$ the set of prime numbers. Expectation with respect to $\P$ is denoted by $\E$.

Recall the randomized Riemann zeta function from \eqref{0.0bis},
$$
\zeta_\rand(s)=  \prod_{p\in\mathcal P}^\infty\left(\frac{1}{1-p^{-s}U_p}\right) \quad\textrm{in}\quad \sigma>1/2,
$$
 {\it i.e.,} the derivative of $F$ \eqv(eq:F).  
As we are interested in the behaviour of integral means of this quantity near the critical line $\s=1/2$, we would like to control $\log \zeta_\rand(s)$ up to an almost surely   bounded factor. In order to simplify our further considerations it is thus useful to write (if needed, first for $\sigma>1$, and then using analytic continuation)
$$
\log \zeta_\rand(s) = -\sum_{p \in \mathcal P} \log (1-p^{-s}U_p) = \sum_{p \in \mathcal P} U_pp^{-s} + B(s)
$$
with $B=B_1+B_2$, where
$$
B_1(s):= \frac{1}{2}\sum_{p \in \mathcal P}(U_p)^2p^{-2s}\qquad\textrm{and}\quad B_2(s):= \sum_{k\geq 3}\frac{1}{k}\sum_{p \in \mathcal P}(U_p)^kp^{-ks}.
$$

The analytic continuation alluded to above, and over all the fact that $\zeta_\rand(s)$ is well-defined, both follow from the following observation.
\begin{lemma}\label{le:monday}
The series defining $\log \zeta_\rand(s)$ converges almost surely uniformly in any compact subset of $\{\sigma>1/2\}$, and hence defines an analytic function there. In turn, almost surely both the series defining $B_1$ and $B_2$ converge locally uniformly in any compact subset of $\{\sigma>1/3\}$. Especially, $B$ is almost surely analytic and bounded in $\{\sigma\geq 1/2\}.$
\end{lemma}
\begin{proof}We shall make use of Jensen's lemma from the classical theory of Dirichlet series (see \cite[Lemma 4.1.1]{QQ}), which implies that if an ordinary Dirichlet series $\sum_{n=1}^\infty b_nn^{-s}$ converges at the single point $s_0=\sigma_0+ih_0$, then it converges uniformly in $K$, where $K\subset\{\sigma>\sigma_0\}$ is an arbitrary compact subset.  To deduce the first statement concerning $\zeta_\rand(s)$ we simply note that the defining series converges almost surely at every point $\sigma_k:=1/2+1/k$, $k\geq 1$, by \cite[Theorem 5.17]{Kal}. In turn, we see that
$B$ is almost surely analytic in the domain $\sigma >1/3$. First of all, the series defining $B_1$ converges almost surely at $\sigma=1/3$ , and we may again invoke Jensen's lemma. Moreover,  the  double sum defining $B_2$ converges absolutely and uniformly in any half-plane $\{ \sigma\geq \sigma_0\}$ with $\sigma_0>1/4.$ 
\end{proof}

A useful conclusion from the above lemma is that in order to understand the behaviour of the integral means, we may restrict ourselves to considering the integral means of the quantity $\sum_{p \in \mathcal P}U_pp^{-s}$.

Another reduction we will make here is to note that instead of $\b\in \C$, {\it i.e.}, ~general  complex moments, it will be enough to consider only positive real moments. Namely, if $\b=|\b|e^{iu}$ we have
\begin{eqnarray*}
\Big|\exp (\b\sum_{p \in \mathcal P} U_pp^{-s})\Big|=\Big|\exp (|\b|\sum_{p \in \mathcal P} \widetilde U_pp^{-s})\Big|,
\end{eqnarray*}
where $\widetilde U_p:=e^{iu}U_p.$ It remains to use the fact that the infinite sequence $(\widetilde U_p)_{p\in \mathcal P}$ has the same distribution as $(U_p)_{p\in \mathcal P}$.

Given $\s>1/2$,  the process we are interested in is defined by
\be
\left(X^{\s}_h, h\in [0,1]\right)\quad\text{where}\quad X^{\s}_h=\sum_{p \in \mathcal P}\frac{\Re\left(U_pp^{-ih}\right)}{p^{\s}}.
\Eq(1.1)
\ee

The proof of our main result draws on the work of  \cite{ABH} and \cite{AT19}, where a random model of the zeta function \cite{Bag81,HLS,H13,Harper19} is investigated.
A key idea behind the strategy of the proof consists in introducing a so-called \emph{multiscale decomposition} of the process, namely, we rewrite $X^{\s}_h$ as 
\be
X^{\s}_h=\sum_{k=0}^{\infty}Y^{\s}_h(k)\quad\text{where}\quad Y^{\s}_h(k)=\sum_{2^{k-1}<\log p\leq 2^k}W^{\s}_h(p),
\quad k\in\N.
\Eq(1.5)
\ee
where we write
\be
W^{\s}_h(p):=\frac{\Re\left(U_pp^{-ih}\right)}{p^{\s}}, \quad h\in\R.
\Eq(1.6)
\ee
The law of the process $(W^{\s}_h(p), h\in\R)$ is both translation invariant and invariant under the reflexion $h\mapsto -h$. Using these properties, it easily follows from \eqv(1.6)  and the law of the random variables $U_p$ that
\be
\E\left[W^{\s}_h(p)W^{\s}_{h'}(p)\right]=\frac{1}{2p^{2\s}}\cos(|h-h'|\log p)\quad \text{for all}\,\, h,h'.
\Eq(1.7)
\ee
In the notation \eqv(1.6), the increments $Y^{\s}_h(k)$ defined in \eqv(1.5) read
\be
Y^{\s}_h(k)=\sum_{2^{k-1}<\log p\leq 2^k}W^{\s}_h(p), \quad h\in\R,
\Eq(1.8)
\ee
Thus, by \eqv(1.7) and the independence of the variables $U_p$,  $Y^{\s}_h(k)$ has variance 
\be
\varsigma_k^2 = \var(Y^{\s}_h(k))=\E\left[Y^{\s}_h(k)^2\right]=\sum_{2^{k-1}<\log p\leq 2^k}\frac{1}{2p^{2\s}},
\Eq(1.9)
\ee
whereas the covariance of $Y^{\s}_h(k)$ and $Y^{\s}_{h'}(k)$  is 
\be
\rho_{k}(h,h') =\E\left[Y^{\s}_h(k)Y^{\s}_{h'}(k)\right]=\sum_{2^{k-1}<\log p\leq 2^k}\frac{1}{2p^{2\s}}\cos(|h-h'|\log p).
\Eq(1.10)
\ee
A useful collective random variable will be 
\be
X^{\s}_{h}(k_1,k_2)
=\sum_{k=k_1+1}^{k_2}Y^{\s}_h(k),
\quad h\in\R,
\Eq(1.11)
\ee  
whose variance reads
\be
\varsigma_{k_1,k_2}^2 = \var(X^{\s}_{h}(k_1,k_2))=\sum_{k=k_1+1}^{k_2} \var(Y^{\s}_h(k))=\sum_{k=k_1+1}^{k_2}\varsigma_k^2=\sum_{2^{k_1}<\log p\leq 2^{k_2}}\frac{1}{2p^{2\s}},
\Eq(1.12)
\ee
where, with a slight redundancy of notation, $\varsigma_{k-1,k}\equiv\varsigma_k$. Of course, all moments of the process  $Y^{\s}_h(k)$ depend on $\s$. To avoid overloading the notation, we keep this dependence implicit. This should not cause any confusion.


\subsection{Moments of the random Riemann zeta function}
    \TH(S2.1)
Recall that we are interested in taking the limit of the ``free energy'' (see \eqv(1.theo1.2bis)) as $\s$ tends to $1/2$ from above.
Due to the multiscale decomposition \eqref{1.5} of the model, a key parameter is the integer $n$, defined by
\be
n\equiv n(\s):=\left\lceil\frac{\log \left[(2\s-1)^{-1}\right]}{\log 2}\right\rceil.
\Eq(1.4bis)
\ee
Note that 
\be\label{2ns} 
2^{n-1}< (2\s-1)^{-1}\leq 2^n,
\ee
and that without loss in generality we can take  $1/2<\s\leq 1$ and $n\in \mathbb N$. 
A key quantity for $\sigma > 1/2$ will be 
\be
v^2_k=
\frac{1}{2}\left\{{\rm E}_1\left[(2\s-1)2^{k-1}\right]-{\rm E}_1\left[(2\s-1)2^{k}\right]\right\},
\Eq(2.1.0')
\ee
where ${\rm E}_1$ is the Exponential Integral function
\be
{\rm E}_1(x)=\int_{x}^{\infty}\frac{e^{-t}}{t}dt.
\Eq(2.1.1')
\ee
The latter can be written for $x>0$ in terms of a converging series as,
\be
{\rm E}_1(x)=-\g -\ln x+\sum _{n=1}^{\infty}\frac{(-1)^{n+1}x^n}{n\, n!}, 
\Eq(2.1.2')
\ee
One also has the asymptotic expansion as $x\to +\infty$, 
\be
{\rm E}_1(x)=\frac{e^{-x}}{x}\sum_{n=0}^N \frac{n!}{(-x)^n} + \frac{N!}{(-x)^N}o(e^{-x}).
\Eq(2.1.3")
\ee
(See Section 5.1 in \cite{AS}.)

For $h,h'\in\R$, define
\be
h\curlyvee h':=\lceil\log_2|h-h'|^{-1}\rceil ,
\Eq(2.1.4')
\ee
so that 
\be \label{2.1.4'bis}
2^{h\curlyvee h' -1} < |h-h'|^{-1} \leq 2^{h\curlyvee h'}.
\ee
Recall definitions \eqv(1.9) and \eqv(1.10).

\begin{lemma}
 \TH(2.1.lem1)
 For $\s>1/2$ and $k\geq 1$, we have
\bea
\varsigma_k^2 \hspace{-6pt}&=&\hspace{-6pt} v^2_k+\OO\left(\sqrt{2^{k-1}}e^{-c\sqrt{2^{k-1}}}\right). 
\Eq(2.1.lem1.0)
\\
\rho_{k}(h,h')  \hspace{-6pt}&=&\hspace{-6pt}
\begin{cases}
\OO\left(2^{h\curlyvee h'-k}\right)+\OO\left(2^{h\curlyvee h'-n}\right)+\OO\left(\sqrt{2^{k-1}}e^{-c\sqrt{2^{k-1}}}\right)& \hspace{-2pt}\text{if $k> h\curlyvee h'$,} \quad\quad\,\,\\
v^2_k+\OO\left(2^{-2(h\curlyvee h'-k)}\right)+\OO\left(\sqrt{2^{k-1}}e^{-c\sqrt{2^{k-1}}}\right)&  \hspace{-2pt}\text{if $k\leq h\curlyvee h'$.}
\end{cases}
\Eq(2.1.lem1.1)
\eea
\end{lemma} 

\begin{lemma}
  \TH(2.1.lem2)
For $0\leq k_1<k_2$, we have
\bea
&&\varsigma_{k_1,k_2}^2=\sum_{k=k_1+1}^{k_2}\varsigma_k^2=\sum_{k=k_1+1}^{k_2}v^2_k+\OO\left(\sqrt{2^{k_1}}e^{-c\sqrt{2^{k_1}}}\right), 
\Eq(2.1.lem1.1bis)
\\
&&\sum_{k=k_1+1}^{k_2}v^2_k=\frac{1}{2}\left\{{\rm E}_1\left[(2\s-1)2^{k_1}\right]-{\rm E}_1\left[(2\s-1)2^{k_2}\right]\right\} \Eq(2.1.lem1.3bis)
\\
&&\quad\quad\quad\quad =\frac{1}{2}\left(k_2\wedge n-k_1\wedge n\right)\log 2+\OO\left(2^{-(n-k_2\wedge n)}\right). 
\Eq(2.1.lem1.4bis)
\eea
\end{lemma}
\begin{lemma}
\TH(2.1.lem3)
For $0\leq k_1<k_2$ and $k_1 <n$, we have
\be
\begin{split}
&
\sum_{k=k_1+1}^{k_2}\rho_{k}(h,h')  
\\
& = \begin{cases}
\OO\left(2^{h\curlyvee h'-k_1}\right)+\OO\left((k_2\wedge n-k_1)2^{h\curlyvee h'-n}\right)+\OO\left(\sqrt{2^{k_1}}e^{-c\sqrt{2^{k_1}}}\right)& \text{if $k_1> h\curlyvee h'$,} \\
\sum_{k=k_1+1}^{k_2}v^2_k+\OO\left(2^{-2(h\curlyvee h'-k_2)}\right)+\OO\left(\sqrt{2^{k_1}}e^{-c\sqrt{2^{k_1}}}\right)&  \text{if $k_2\leq h\curlyvee h'$.}
\end{cases}
\end{split}
\Eq(2.1.lem1.2bis)
\ee
\end{lemma}
\begin{proof}[Proof of Lemma \thv(2.1.lem1)]
Lemma \thv(2.1.lem1) is a particular case of Lemmata \thv(2.1.lem2) and \thv(2.1.lem3), obtained by setting $k_1=k-1$ and $k_2=k$.
\end{proof}

\begin{proof}[Proof of Lemma \thv(2.1.lem2)] To simplify the notation, set 
\be \label{PQ}
\log P=2^{k_1},\,\,\,\log Q=2^{k_2},
\ee
and consider the sum
\be
\varsigma^2_{k_1,k_2}=\sum_{k=k_1+1}^{k_2}\varsigma_k^2
=\frac{1}{2}S(P,Q)=\sum_{P< p\leq Q}\frac{1}{2p^{2\s}}.
\Eq(2.1.lem1.4)
\ee
Sums over primes can be estimated using the prime number theorem, which gives the density of primes up to very good error terms,
\bea
\pi(x)&=&\#\{p\leq x : p\, \mathrm{prime}\}=\mathrm{Li}(x)+{\mathcal E}(x),\\
\mathrm{Li}(x)&=&\int_2^x\frac{1}{\log y} dy,\\
{\mathcal E}(x)&=&\OO\left(xe^{-c\sqrt{\log x}}\right).
\eea
A sum over primes between $2< P\leq Q$ can thus be written as,
$$
\sum_{P<p\, \mathrm{prime} \leq Q}f(p)=\int_{(P,Q]}f(x)\pi(dx)=\int_{(P,Q]}\frac{f(x)}{\log x} dx+\int_{(P,Q]}f(x)d{\mathcal E}(x),
$$
with an error term which can be recast as
\be \label{Ef'}
\int_P^Qf(x)d{\mathcal E}(x)={\mathcal E}(Q)f(Q)-{\mathcal E}(P)f(P)-\int_P^Q{\mathcal E}(x)f'(x)dx.
\ee
The \emph{variance} sum \eqref{2.1.lem1.4} corresponds to $f(x)=x^{-2\sigma}$ with  ${\sigma}>1/2$,  and by using  the change of variable  $u=\log x$ we get 
$$
\int_P^Q\frac{x^{-2\sigma}}{\log x} dx=\int_{\log P}^{\log Q}e^{-(2\sigma -1)u}\frac{du}{u}=E_1[(2{\sigma}-1)\log P]-E_1[(2{\sigma}-1)\log Q],
$$
in terms of the exponential integral \eqref{2.1.1'}.
A detailed calculation of the remainder term \eqref{Ef'} shows that 
$$
\int_P^Qx^{-2\sigma}{\mathcal E}(dx)= \OO\left(\frac{\sqrt{\log P}}{P^{2{\sigma}-1}}e^{-c\sqrt{\log P}}\right),
$$ where the  leading term actually comes from the third one in Eq. \eqref{Ef'}, thus correcting Eq.~(16) in \cite{AT19}.   We finally  have
 \be
 S(P,Q)=\sum_{P<p\leq Q}\frac{1}{p^{2{\sigma}}}=E_1[(2{\sigma}-1)\log P]-E_1[(2{\sigma}-1)\log Q] +\OO\left(\frac{\sqrt{\log P} }{P^{2{\sigma}-1}}e^{-c\sqrt{\log P}}\right). 
 \Eq(2.1.F1)
 \ee 
 If $(2\sigma -1)\log Q \leq 1$, using \eqref{2.1.2'} then gives
 \be
 S(P,Q)=\log\log Q-\log\log P+ \OO\left((2\sigma -1)\log Q\right)+\OO\left(\frac{\sqrt{\log P} }{P^{2{\sigma}-1}}e^{-c\sqrt{\log P}}\right).
 \Eq(2.1.F1bis)
\ee
 If $(2\sigma -1)\log P \leq 1$ and $(2\sigma -1)\log Q>1$, using respectively \eqref{2.1.2'}  and \eqref{2.1.3"} yields
\be
 S(P,Q)=-\log[(2\sigma-1)\log P]+\OO(1)+ \OO\left(\frac{\sqrt{\log P} }{P^{2{\sigma}-1}}e^{-c\sqrt{\log P}}\right).
\Eq(2.1.F1ter)
\ee 
  If both $(2\sigma -1)\log P > 1$ and $(2\sigma -1)\log Q>1$, Eq. \eqref{2.1.3"} yields
\bea\nonumber
  S(P,Q)&=&\OO\left(\frac{1}{(2\sigma -1)\log P\,P^{2{\sigma}-1}}\right)+ \OO\left(\frac{\sqrt{\log P} }{P^{2{\sigma}-1}}e^{-c\sqrt{\log P}}\right)\\
  &=&\OO\left(1\right)+ \OO\left(\frac{\sqrt{\log P} }{P^{2{\sigma}-1}}e^{-c\sqrt{\log P}}\right).
\Eq(2.1.F1quater)
\eea 
Notice also that the remainder term can be simply bounded above for $\sigma> 1/2$ by
  \be\label{1/2}
  \OO\left(\frac{\sqrt{\log P} }{P^{2{\sigma}-1}}e^{-c\sqrt{\log P}}\right)< \OO\left(\sqrt{\log P}e^{-c\sqrt{\log P}}\right).
  \ee
  For $\log P=2^{k_1}, \log Q=2^{k_2}$, recalling \eqref{2ns}, results \eqref{2.1.F1} to \eqref{1/2} together yield results \eqref{2.1.lem1.1bis}, \eqref{2.1.lem1.3bis} and \eqref{2.1.lem1.4bis}. Note that \eqref{2.1.lem1.4bis} is obtained by using \eqv(1.4bis) to express $\log (2\sigma -1)^{-1}$  in \eqref{2.1.F1bis}, \eqref{2.1.F1ter} and \eqref{2.1.F1quater} in terms of $n \log2$. The $\OO(1)$ difference between these two quantities appears
only in the latter two cases, {\it i.e.},  when $k_2>n$ , and is taken into account in \eqref{2.1.lem1.4bis} by the $\OO\left(2^{-(n-k_2\wedge n)}\right)$ term there.  
 \end{proof}
 \begin{proof}[Proof of Lemma \thv(2.1.lem3)]
Consider now the \emph{covariance} sum with $\Delta:=|h-h'|$,
\be
\sum_{k=k_1+1}^{k_2}\rho_{k}(h,h') = 
\frac{1}{2}S(P,Q,\Delta):=\sum_{P< p\leq Q}\frac{1}{2p^{2\s}}\cos(\Delta\log p).
\Eq(2.1.lem1.5)
\ee
Let us write
\be
S(P,Q,\Delta)=I(P,Q,\Delta)+R(P,Q,\Delta),
\Eq(2.1.F2)
\ee
where the prime number theorem now yields for $f(x)=\cos(\Delta \log x){x^{-2\sigma}}$
\bea
I(P,Q,\Delta)&=&\int_P^Q \frac{\cos(\Delta \log x)}{x^{2\sigma}\log x}dx \Eq(2.1.F3)\\
R(P,Q,\Delta)&=&\int_P^Q \frac{\cos(\Delta \log x)}{x^{2\sigma}}{\mathcal E}(dx). \Eq(2.1.F4)
\eea
$\bullet$ For $\Delta \log Q\leq 1$, expanding the cosine term in \eqref{2.1.F3}  immediately yields
\be 
I(P,Q,\Delta)=E_1[(2{\sigma}-1)\log P]-E_1[(2{\sigma}-1)\log Q]+\OO\left(\Delta^2\log^2 Q\right).
\Eq(2.1.F5)
\ee
$\bullet$ For $\Delta \log  P\geq 1$, one can integrate  \eqref{2.1.F3} by parts with $u=\log x$ to get
$$
I(P,Q,\Delta)=\frac{\sin \Delta u}{\Delta u}e^{-(2\sigma -1)u}\Big{|}_{\log P}^{\log Q}+\frac{1}{\Delta}\int_{\log P}^{\log Q}\sin \Delta u \left[\frac{1}{u^2}+\frac{2\sigma -1}{u}\right]e^{-(2\sigma -1)u}d{u}.$$ This yields the orders of magnitude
\be
|I(P,Q,\Delta)|
\lesssim
\frac{1}{\Delta \log P}\frac{1}{P^{2{\sigma}-1}}+\frac{2\sigma -1}{\Delta}\left\{\mathrm{E}_1((2\sigma -1)\log P)-\mathrm{E}_1((2\sigma -1)\log Q)\right\}.
\Eq(2.1.F7)
\ee
Notice that the estimates of the exponential integral terms appearing in \eqref{2.1.F5} and \eqref{2.1.F7} can be performed exactly as the estimates of \eqref{2.1.F1}  in \eqref{2.1.F1bis}, and \eqref{2.1.F1ter} and  \eqref{2.1.F1quater} in the proof of Lemma  \thv(2.1.lem2).

\noindent $\bullet $ Finally, using $|f(x)|\leq x^{-2\sigma}$, the remainder $R$-term can be estimated exactly as above, independently of $\Delta$ and for all $Q$, as
\be \Eq(2.1.F9)
R(P,Q,\Delta)=\int_P^Q f(x){\mathcal E}(dx)= \OO\left(\frac{\sqrt{\log P} }{P^{2{\sigma}-1}}e^{-c\sqrt{\log P}}\right).
\ee

For $P,Q$ as in \eqref{PQ}, and recalling the notation \eqref{2.1.4'}, the estimates \eqref{2.1.F5} and \eqref{2.1.F7}, together with \eqref{2.1.lem1.3bis}, \eqref{2.1.lem1.4bis}, \eqref{2.1.F9} and \eqref{1/2}, respectively yield the $k_2 \leq h\curlyvee h'$ and $k_1 > h\curlyvee h'$ cases of \eqref{2.1.lem1.2bis} in Lemma \thv(2.1.lem3).
\end{proof}
We recall that $\sigma$ and $n$ in the lemma below are related via \eqref{1.4bis}.
\begin{lemma}\label{rk1}For all $m> 0$, and for $k_2\leq n$, 
\be
\sum_{2^{k_1}<\log p\leq 2^{k_2}}\frac{(\log{p})^m}{p^{2\s}}\leq \frac{1}{m}2^{mk_2} \left[1+\OO(2^{k_2-n})\right]+\OO\left(2^{k_1(m+1/2)}e^{-c\sqrt{2^{k_1}}}\right),
\Eq(2.1.F10bis)
\ee
whereas for $k_2>n$
\bea
\sum_{2^{k_1}<\log p <2^{k_2}}\frac{(\log{p})^m}{p^{2\s}}
\leq 2^{mn}\Gamma(m)+\OO\left(2^{k_1(m+1/2)}e^{-c\sqrt{2^{k_1}}}\right).
\Eq(2.1.F11bis)
\eea
\end{lemma}
 
\begin{proof}
Using the prime number theorem for $f(x)=(\log x)^m x^{-2\sigma}$ yields
\bea\nonumber
\sum_{P<p\leq Q}\frac{(\log{p})^m}{p^{2\s}}&=&I_m + R_m\\ \nonumber
I_m&=&\int_P^Q dx \frac{(\log x)^{m-1}}{x^{2\sigma}}\\ \nonumber
&=&(2\sigma-1)^{-m}\left[\gamma(m, (2\sigma-1) \log Q)-\gamma(m, (2\sigma-1) \log P)\right] \\ \nonumber
R_m
&=&\OO\left(\frac{(\log P)^{m+1/2}}{P^{2\s-1}}e^{-c\sqrt{\log P}}\right).
\eea
 Here $\gamma(s,x)$ is the incomplete gamma-function
\be \nonumber
\gamma(s,x)=\int_0^x e^{-t} t^{s-1} dt=x^s\sum_{k=0}^{\infty} \frac{(-x)^k}{(s+k)k!},
\ee
such that $\gamma(s,\infty)=\Gamma(s)$.

When $(2\sigma -1)\log Q \leq 1$, we have
 \be
\sum_{P<p\leq Q}\frac{(\log{p})^m}{p^{2\s}}= \frac{1}{m}(\log Q)^m \left\{1+\OO[(2\sigma -1)\log Q]\right\}+\OO\left((\log P)^{m+1/2}e^{-c\sqrt{\log P}}\right),
\Eq(2.1.F10)
\ee 
 in agreement with Remark 1 in \cite{ABH}.  
 When $(2\sigma -1)\log Q > 1$, we simply observe that 
\bea
\sum_{P<p \leq Q}\frac{(\log{p})^m}{p^{2\s}}
\leq (2\sigma-1)^{-m}\Gamma(m)+\OO\left((\log P)^{m+1/2}e^{-c\sqrt{\log P}}\right).
\Eq(2.1.F11ter)
\eea
Specifying the above for $\log P=2^{k_1}, \log Q=2^{k_2}$, \eqref{2.1.F10} thus yields the $k_2\leq n$ case \eqref{2.1.F10bis}, 
whereas \eqref{2.1.F11ter} yields the $k_2>n$ case \eqref{2.1.F11bis}. 
\end{proof}
\begin{lemma}\label{rk1bis}For all $m> 0$, and all $q>1$
\be
\sum_{2^{k_1}<\log p\leq 2^{k_2}}\frac{(\log{p})^m}{p^{2q\s}} < (q-1)^{-m}\Gamma(m)+\OO\left(e^{-(q-1)2^{k_1}}2^{k_1(m+1/2)}e^{-c\sqrt{2^{k_1}}}\right).
\Eq(2.1.F10quater)
\ee
\end{lemma}
 
\begin{proof}
Using the prime number theorem for $f(x)=(\log x)^m x^{-2 q\s}$ yields
\bea\nonumber
\sum_{P<p\leq Q}\frac{(\log{p})^m}{p^{2q\s}}&=&I_{m,q} + R_{m,q}\\ \nonumber
I_{m,q}&=&\int_P^Q dx \frac{(\log x)^{m-1}}{x^{2q\s}}\\ \nonumber
&=&(2q\sigma-1)^{-m}\left[\gamma(m, (2q\sigma-1) \log Q)-\gamma(m, (2q\sigma-1) \log P)\right] \\ \nonumber
R_{m,q}
&=&\OO\left(\frac{(\log P)^{m+1/2}}{P^{2q\s-1}}e^{-c\sqrt{\log P}}\right).
\eea
Using $\gamma(s,x)<\gamma(s,\infty)= \Gamma(s)$, and $\s>1/2$, we obviously have 
\bea\nonumber
I_{m,q}&<& (q-1)^{-m}\Gamma(m)\\ \nonumber
R_{m,q}
&<&\OO\left(\frac{(\log P)^{m+1/2}}{P^{q-1}}e^{-c\sqrt{\log P}}\right).
\eea
Specifying the above for $\log P=2^{k_1}, \log Q=2^{k_2}$ yields \eqref{2.1.F10quater}, 
\end{proof}
\begin{lemma}\label{rk2}
For all  $m > 1$ 
and $\sigma \geq \frac{1}{2}$, 
\be 
\sum_{2^{k_1}< \log p\leq 2^{k_2}}\frac{1}{(p^{2\s})^m}\leq \frac{1}{(m-1)2^{k_1}}e^{-(m-1)2^{k_1}} (1+o(1)),
\Eq(2.1.5'bis)
\ee
where the small o is with respect to $k_1\to\infty$.
\end{lemma}
\begin{proof}
Using the prime number theorem gives
\be \nonumber
\sum_{P< p\leq Q}\frac{1}{(p^{2\s})^m}={\rm E}_1((2m\s-1)\log P)-{\rm E}_1((2m\s-1)\log Q) + \OO\left(\frac{1}{P^{2m\s-1}}e^{-c\sqrt{\log P}}\right).
\ee
We thus have 
\bea \nonumber
\sum_{P< p\leq Q}\frac{1}{(p^{2\s})^m} &\leq & \sum_{P< p < \infty}\frac{1}{(p^{2\s})^m} \Eq(2.1.F12)\\ \nonumber
&=&{\rm E}_1((2m\s-1)\log P) + \OO\left(\frac{1}{P^{2m\s-1}}e^{-c\sqrt{\log P}}\right)\\ \nonumber
&\leq& {\rm E}_1((m-1)\log P) +\OO\left(\frac{1}{P^{m-1}}e^{-c\sqrt{\log P}}\right)\\ \nonumber
&=&\frac{1}{(m-1)\log P}\frac{1}{P^{m-1}}(1+o(1))+\OO\left(\frac{1}{P^{m-1}}e^{-c\sqrt{\log P}}\right), 
\Eq(2.1.F13)
\eea
where we used \eqv(2.1.3"). 
Specifying this result for $\log P=2^{k_1}, \log Q=2^{k_2}$ gives  \eqref{2.1.5'bis}.
\end{proof} 


\subsection{Key probability estimates}
    \TH(S2.2)
The properties of the fine dyadic decomposition $Y^{\s}_h$  of the process $X^{\s}_{h}$ obtained in Section \thv(S2.1)  are now used to obtain estimates on the law of $X^{\s}_h$ on mesoscopic scales. More precisely, given $0\leq \a_1<\a_2\leq \infty$ (independent of $n$), we set 
\be
k_1\equiv k_1(\a_1):=\lceil\a_1 n \rceil, \quad k_2\equiv k_2(\a_2):=\lceil\a_2 n \rceil,
\Eq(2.2.1)
\ee
and consider
\be
X^{\s}_{h}(k_1,k_2)
=\sum_{i=k_1+1}^{k_2}Y^{\s}_h(i),
\quad h\in\R.
\Eq(2.2.2)
\ee      
A key parameter governing tail probabilities at mesoscopic scales is
\be
\bar\a_2:=\a_2\wedge 1.
\Eq(2.2.1')
\ee

To formulate our results, we use the following notations and definitions. 
Let $(\s_n, n\in\N)$ be the sequence with the general term
\be
\s_n=\frac{1}{2}+2^{-(n+1)},\quad n\geq 0.
\Eq(3.lem1.01)
\ee
This sequence is decreasing and, in view of \eqv(1.4bis) and  \eqv(2ns),  satisfies  for all $n\geq 0$
\be
n\left(\s_n\right)=n,
\Eq(3.lem1.02)
\ee
and for all $n\geq 1$
\be
\left\{\s \mid n\left(\s\right)=n\right\}=[\s_{n},\s_{n-1}).
\Eq(3.lem1.03)
\ee
\smallskip

We first state all the results we need in a series of lemmas, and the proofs are given afterwards.
The first result expresses some key features of the process $X^{\s}_{h}(k_1,k_2)$, namely, Gaussianity and branching.
Recall definition \eqv(2.2.1).

\begin{lemma}
    \TH(2.2.lem1)
Let  $0< \a_1<1$, $\a_1<\a_2\leq \infty$, and $\g>0$.
For all $n\in\N$ large enough, the following holds for all $\s\in[\s_n, \s_{n-1})$

\item{(i)} For all $h\in[0,1]$ 
\be
c_1
\leq \left[\sfrac{1}{\sqrt{\g^2 n\log 2}}2^{-\frac{n\g^2}{\bar\a_2-\a_1}}\right]^{-1}\P\left(X^{\s}_{h}(k_1,k_2)\geq \g n\log 2\right)
\leq 
c_2
\Eq(2.2.lem1.1)
\ee
where $0<c_1<c_2<\infty$ are constants that depend only on $\a_1$, $\bar\a_2$ and $\g$.

\item{(ii)} For all $h,h'\in[0,1]$ such that $h\curlyvee h'<k_1$, 
\be
\begin{split}
&
\P\left(
X^{\s}_{h}(k_1,k_2)\geq \g n\log 2,
X^{\s}_{h'}(k_1,k_2)\geq \g n\log 2
\right)
\\
= 
& 
(1+\OO(1))
\left[\P\left(X^{\s}_{h}(k_1,k_2)\geq \g n\log 2\right)\right]^2.
\end{split}
\Eq(2.2.lem1.2)
\ee
Moreover, if there exists $\e>0$ such that $h\curlyvee h'-k_1\leq -\e n$, then the $\OO(1)$ term in \eqv(2.2.lem1.2) 
can be refined to $\OO\left(n2^{-\e n}\right)$.

\item{(iii)} For all $h,h'\in[0,1]$ such that $h\curlyvee h'\geq k_2$
\be
\begin{split}
&
\P\left(
X^{\s}_{h}(k_1,k_2)\geq \g n\log 2,
X^{\s}_{h'}(k_1,k_2)\geq \g n\log 2
\right)
\\
\leq 
& 
\P\left(X^{\s}_{h}(k_1,k_2)\geq \g n\log 2\right).
\end{split}
\Eq(2.2.lem1.3)
\ee
\end{lemma}

\begin{remark} The dichotomy between part (ii) and part (iii) of Lemma \thv(2.2.lem1) reflects the fact that the event under consideration  occurs, respectively, after and before the \emph{branching point} of the random walks associated with the processes $X^{\s}_h(k)$ and $X^{\s}_{h'}(k)$ (obtained by viewing the partial sums $X^{\s}_h(k)=\sum_{i=0}^kY^{\s}_h(i)$ and $X^{\s}_{h'}(k)=\sum_{i=0}^kY^{\s}_{h'}(i)$ as random walks). 
Before branching, the two walks are almost the same, whereas after branching they become almost independent. We refer to Section 1.4 in \cite{ABH} for a detailed discussion of this construction and phenomenology.
\end{remark}

In addition to Lemma \thv(2.2.lem1), which only holds  for $\a_1>0$, we need the following estimates which hold for $\a_1\geq 0$ and provide uniform control over suitable intervals of the parameters $h$ and (or) $\s$.
Again, recall \eqv(2.2.1).

\newpage

\begin{lemma}
    \TH(2.2.lem2)
\item{(i)}
For all $0\leq \a_1<1$, $\a_1<\a_2\leq \infty$, $\g>0$ and $h\in[0,1]$
\be
\P\left(
\max_{\s:\in[\s_n, \s_{n-1})}\max_{h':|h-h'|\leq 2^{-\lceil\bar\a_2 n \rceil}}X^{\s}_{h'}(k_1,k_2)\geq 
\g n\log 2
\right)
\leq C 2^{-\frac{n\g^2}{\bar\a_2-\a_1}},
\Eq(2.2.lem2.1)
\ee
for some constant $0<C<\infty$ depending on $\a_1$, $\bar\a_2$ and $\g$.
\item{(ii)}
For all $\g>0$,  $\a_2>0$ and $h\in[0,1]$ 
\be
\P\left(
\max_{\s:\in[\s_n, \s_{n-1})}X^{\s}_{h}(0,k_2)\geq \g n\log 2
\right)
\leq C' 2^{-\frac{n\g^2}{\bar\a_2}},
\Eq(2.2.lem2.1')
\ee
for some constant $0<C'<\infty$ depending only on $\g$.
\end{lemma}

The rest of this section is devoted to proving Lemma \thv(2.2.lem1) and Lemma \thv(2.2.lem2). The proof of Lemma \thv(2.2.lem1) is based on the following classical and multivariate versions of the Berry-Esseen theorem.

\begin{theorem}[Theorem 2, Section XVI.5 of \cite{Fe2}]
    \TH(2.2.theo1)
Let $\{Z_i, i\geq 1\}$ be independent random variables on $(\R, \BB(\R), {\mathbb P})$ such that ${\mathbb E} Z_i=0$.
Put 
\be
S_m=\sum_{i=1}^mZ_i,\quad
\S^2_m=\sum_{i=1}^m{\mathbb E} Z_i^2,\quad
\vartheta_m=\sum_{i=1}^m \E\left|Z_i^3\right|.
\Eq(2.2.theo1.1)
\ee
Then, for all  $m$
\be
\sup_{-\infty<a<\infty}\left|
{\mathbb P}\bigl(
{\textstyle S_m}\geq  a
\bigr)
-
\frac{1}{\sqrt{2\pi}}\int_{a/\S_m}^{\infty}e^{-\frac{x^2}{2}}dx
\right|
\leq 6\frac{\vartheta_m}{\S^3_m}.
\Eq(2.2.theo1.2)
\ee
\end{theorem}

\begin{theorem}
[Corollary 17.2 in \cite{BR10}] 
    \TH(2.2.theo2)
Let $\{\boldsymbol{Z_i} , i\geq 1\}$ be a sequence of independent random vectors on $(\R^d, \BB(\R^d), {\P})$ with mean zero
and covariance matrix $\Cov(\boldsymbol{Z_i})$. Put
\be
\boldsymbol{S_m}=\sum_{i=1}^m\boldsymbol{Z_i},\quad
\boldsymbol{\S_m}=\sum_{i=1}^m\Cov(\boldsymbol{Z_i}),\quad
\boldsymbol{\vartheta}_m=\sum_{i=1}^m{\E}\| \boldsymbol{Z_i}\|^3,
\Eq(2.2.theo2.1)
\ee
where $\|\cdot\|$ is the Euclidean norm in $\R^d$. Denote by $\l_m$ be the smallest eigenvalue of $\boldsymbol{\S_m}$. There exists an absolute constant $c>0$ depending only on $d$ such that
\be
\sup_{A\in\AA}\left|
{\P}\bigl(\boldsymbol{S_m}\in A\bigr)
-\frac{1}{\sqrt{(2\pi)^d\det \boldsymbol{\S_m}}}
\int_Ad^d\mathbf{x}
e^{-\frac{1}{2}\left(\mathbf{x},{\boldsymbol{\S^{-1}_m}}\mathbf{x}\right)}
\right|
\leq c\frac{\boldsymbol{\vartheta_m}}{\l_m^{3/2}},
\Eq(2.2.theo2.2)
\ee
where $\AA$ is the collection of Borel measurable convex subsets of $\R^d$.
\end{theorem}

\begin{proof}[Proof of Lemma \thv(2.2.lem1)] 
By \eqv(1.5), \eqv(1.6) and \eqv(2.2.2).
\be
X^{\s}_{h}(k_1,k_2)=\sum_{i=k_1+1}^{k_2}Y^{\s}_{h}(i)=\sum_{2^{k_1}<\log p\leq 2^{k_2}}W^{\s}_{h}(p),\quad h\in\R.
\Eq(2.2.lem1.30)
\ee
Part (i) will follow from Theorem \thv(2.2.theo1) applied to  the rightmost  sum in \eqv(2.2.lem1.30), {\it i.e.},  the sum over the  $W^{\s}_{h}(p)$'s.
To determine the two associated quantities that appear in  \eqv(2.2.theo1.1), we note, on the one hand, that by the independence of the $W^{\s}_{h}(p)$'s and \eqv(2.2.lem1.30), the variance term is given by \eqv(1.12) as
\be
\Var{X(k_1,k_2)}=\sum_{2^{k_1}<\log p\leq 2^{k_2}}\E\left[W^{\s}_{h}(p)^2\right]
=\varsigma^2_{k_1,k_2}
=\sum_{i=k_1+1}^{k_2}\varsigma^2_i.
\Eq(2.2.lem1.31)
\ee
Thus, by \eqv(2.1.lem1.1bis)
\be
\varsigma^2_{k_1,k_2}
=\sum_{i=k_1+1}^{k_2}v^2_i+\OO\left(\sqrt{2^{k_1}}e^{-c\sqrt{2^{k_1}}}\right),
\Eq(2.2.lem1.4)
\ee
where by the identity \eqv(2.1.lem1.4bis)
\be
\begin{split}
\sum_{i=k_1+1}^{k_2}v^2_i
& =
\frac{1}{2}\left(k_2\wedge n-k_1\right)\log 2+\OO\left(2^{-(n-k_2\wedge n)}\right).
\end{split}
\Eq(2.2.lem1.5)
\ee
On the other hand, we simply use that, by definition, $|W^{\s}_{h}(p)|\leq {1}/{p^{\s}}$, so that
\be
\begin{split}
\vartheta_{k_1,k_2}
=&
\sum_{2^{k_1}<\log p\leq 2^{k_2}}|W^{\s}_{h}(p)|^3
\\
\leq
&
\sum_{2^{k_1}<\log p\leq 2^{k_2}}\left(\frac{1}{p^{2\s}}\right)^{3/2}
\leq 
\frac{1}{2^{k_1-1}}e^{-2^{k_1-1}} (1+o(1)),
\end{split}
\Eq(2.2.lem1.4')
\ee
where the last inequality is  \eqv(2.1.5'bis) for $m=3/2$.
Combining \eqv(2.2.lem1.4) and \eqv(2.2.lem1.4'), the error term of \eqv(2.2.theo1.2) satisfies
\bea
\frac{\vartheta_{k_1,k_2}}{\varsigma^3_{k_1,k_2}}
=\OO\left(\left(k_2\wedge n-k_1\right)^{-3/2}2^{-k_1+1}e^{-2^{k_1-1}}\right).
\Eq(2.2.lem1.7)
\eea

Now, consider the Gaussian integral in \eqv(2.2.theo1.2) and set 
\be
\II=\frac{1}{\sqrt{2\pi}}\int_{a/\varsigma_{k_1,k_2}}^{\infty}e^{-\frac{x^2}{2}}dx.
\Eq(2.2.lem1.6'')
\ee
From well-known bounds on Gaussian tails (see, {\it e.g.}, 7.1.23-7.1.24 in \cite{AS}) we have,  for all $a>0$
\be
\frac{\varsigma_{k_1,k_2}}{a}\left[1-\left(\frac{\varsigma_{k_1,k_2}}{a}\right)^2\right]
\leq 
(2\pi)^{1/2}e^{+\frac{1}{2} {a^2}/{\varsigma^2_{k_1,k_2}}}
\II
\leq 
\frac{\varsigma_{k_1,k_2}}{a}.
\Eq(2.2.lem1.6)
\ee
Choosing
$
a
=\g\log\left(2^n\right)
$
for $\g>0$ and using \eqv(2.2.lem1.4)-\eqv(2.2.lem1.5), this yields
\be
\II
=
\left(1+\OO\left(\sfrac{1}{n}\right)\right)
\sqrt{\sfrac{k_2\wedge n-k_1+\OO\left(1\right)}{4\pi\g^2n^2\log 2}}
e^{-\frac{\g^2n^2\log 2}{k_2\wedge n-k_1+\OO\left(1\right)}}.
\Eq(2.2.lem1.6')
\ee
Thus, under the assumptions of the lemma and for all sufficiently large $n$, the error term \eqv(2.2.lem1.7) is negligible compared to the Gaussian integral \eqv(2.2.lem1.6'). In particular,
\be
\frac{\vartheta_{k_1,k_2}}{\varsigma^3_{k_1,k_2}}\frac{1}{\II}
=
o\left(e^{-2^{k_1-2}}\right),
\Eq(2.2.lem1.6''')
\ee
so that by  Theorem \thv(2.2.theo1),
\be
\P\left(X^{\s}_{h}(k_1,k_2)\geq \g n\log2\right)=\left(1+o\left(e^{-2^{k_1-2}}\right)\right)\II.
\Eq(2.2.lem1.6*)
\ee
Finally, using \eqv(2.2.1) and  \eqv(2.2.1') to express  $k_1$ and $k_2$ in terms of $\a_1$, $\bar\a_2$ and $n$, it follows from \eqv(2.2.lem1.6') and \eqv(2.2.lem1.6*) that for all $n$ large enough, 
$
\P\left(
X^{\s}_{h}(k_1,k_2)\geq \g n\log2
\right)
$
is equal to
\be
\left(1+\OO\left(\tfrac{1}{n}\right)\right)\sqrt{\tfrac{\bar\a_2-\a_1}{4\pi\g^2n \log2}}e^{-\frac{\g^2n\log 2}{\bar\a_2-\a_1+\OO\left(1/n\right)}}
\Eq(2.2.lem1.8)
\ee
where the error terms depend only on $\a_1$, $\bar\a_2$ and $\g$.
This yields \eqv(2.2.lem1.1), proving part (i).

We now turn to the proof of part (ii). Recall the notation \eqv(1.5), \eqv(1.6) and \eqv(2.2.2).
Given $h,h'\in \R$, we write
\be
\begin{split}
\boldsymbol{X}^{\s}(k_1,k_2)& =(X^{\s}_{h}(k_1,k_2), X^{\s}_{h'}(k_1,k_2)),
\\
\boldsymbol{Y}^{\s}(k)&=(Y^{\s}_{h}(k), X^{\s}_{h'}(k)),
\\
\boldsymbol{W}^{\s}(p)&=(W^{\s}_{h}(p), W^{\s}_{h'}(p)).
\end{split}
\Eq(2.2.lem1.20)
\ee
Then
\be
\boldsymbol{X}^{\s}(k_1,k_2)=\sum_{i=k_1+1}^{k_2}\boldsymbol{Y}^{\s}(i)=\sum_{2^{k_1}<\log p\leq 2^{k_2}}\boldsymbol{W}^{\s}(p).
\Eq(2.2.lem1.21)
\ee
The proof is similar in spirit to that of part (i), namely, we now apply Theorem \thv(2.2.theo2) to the rightmost sum in \eqv(2.2.lem1.21) over vectors $\boldsymbol{W}^{\s}(p)\in\R^2$. By independence and \eqv(2.2.lem1.21), the matrix $\boldsymbol{\S}_{k_1,k_2}$ from \eqv(2.2.theo2.1) satisfies
\be
\boldsymbol{\S}_{k_1,k_2}
=\sum_{2^{k_1}<\log p\leq 2^{k_2}}\Cov(\boldsymbol{W}^{\s}(p))
=\sum_{i=k_1+1}^{k_2}\Cov(\boldsymbol{Y}^{\s}(i)),
\Eq(2.2.lem1.22)
\ee
and so, writing  $\boldsymbol{\S}_{k_1,k_2}=(s_{kl})_{1\leq k,l\leq 2}$, we have
$
s_{11}=s_{22}=\varsigma^2_{k_1,k_2}
$ 
where $\varsigma^2_{k_1,k_2}$ is given by \eqv(2.2.lem1.4)-\eqv(2.2.lem1.5), while by \eqv(2.1.lem1.2bis), under the assumption that $h\curlyvee h'<k_1<n$
\be
\begin{split}
s_{12}=s_{21}
&=\sum_{i=k_1+1}^{k_2} \rho_{i}(h,h') 
\\
& =\OO\left(2^{h\curlyvee h'-k_1}\right)+\OO\left((k_2\wedge n-k_1)2^{h\curlyvee h'-n}\right)+\OO\left(\sqrt{2^{k_1}}e^{-c\sqrt{2^{k_1}}}\right).
\end{split}
\Eq(2.2.lem1.11)
\ee
One checks that $\boldsymbol{\S}_{k_1,k_2}$ has eigenvalues $s_{11}\pm s_{12}$, so that
$
\det\boldsymbol{\S}_{k_1,k_2}=s_{11}^2-s_{12}^2
$,
and that 
\be
\boldsymbol{\S}_{k_1,k_2}^{-1}
=
\frac{1}{s_{11}^2-s_{12}^2}
\begin{pmatrix}
s_{11} & -s_{12} \\
-s_{12} & s_{11} 
\end{pmatrix}.
\Eq(2.2.lem1.12)
\ee
Thus, choosing $A=[a,\infty)\times [a,\infty)$ in \eqv(2.2.theo2.2) where as before $a=\g\log\left(2^n\right)$ for $\g>0$
\be
\begin{split}
\JJ 
=& 
\frac{1}{\sqrt{(2\pi)^2\det\boldsymbol{\S}_{k_2-k_1}}}\int_Ad^2\mathbf{x}e^{-\frac{1}{2}\left(\mathbf{x},{\boldsymbol{\S}_{k_1,k_2}^{-1}}\mathbf{x}\right)}
\\
=&
\left[1-\bigl(\sfrac{s_{12}}{s_{11}}\bigr)^2\right]^{-1/2}(2\pi)^{-1}
\int_{a/\varsigma_{k_1,k_2}}^{\infty}dx_1\int_{a/\varsigma_{k_1,k_2}}^{\infty}dx_2
e^{-\left[1-\left(\frac{s_{12}}{s_{11}}\right)^2\right]^{-1}\left(\frac{x_1^2+x_2^2}{2}-\frac{s_{12}}{s_{11}}x_1x_2\right)}.
\end{split}
\Eq(2.2.lem1.13)
\ee
Since $2x_1x_2\leq x_1^2+x_2^2$, we have 
\be
 \JJ\leq \left[1-\bigl(\sfrac{s_{12}}{s_{11}}\bigr)^2\right]^{-1/2} \overline\II^2,\quad 
\overline\II=\frac{1}{\sqrt{2\pi}}\int_{a/\varsigma_{k_1,k_2}}^{\infty}dx e^{-\left(1+\frac{s_{12}}{s_{11}}\right)^{-1}\frac{x^2}{2}}.
\Eq(2.2.lem1.17)
\ee
In view of  \eqv(2.2.lem1.6''), $\overline\II$  differs from $\II$ only by a $(1+b)^{-1}$ term in the exponential with $b={s_{12}}/{s_{11}}$. 
To evaluate it, we introduce a threshold value $K>a/\varsigma_{k_1,k_2}$ satisfying $bK^2<C$ for some $0<C<\infty$ and decompose $\overline\II$ into $\overline\II=\overline\II_1+\overline\II_2$,
\be
\overline\II_1=\frac{1}{\sqrt{2\pi}}\int_{a/\varsigma_{k_1,k_2}}^{K}dx e^{-\frac{x^2}{2(1+b)}}, \quad
\overline\II_2=\frac{1}{\sqrt{2\pi}}\int_{K}^{\infty}dx e^{-\frac{x^2}{2(1+b)}}.
\ee
It is easy to see that
\be
\overline\II_1
=\frac{1}{\sqrt{2\pi}}\int_{a/\varsigma_{k_1,k_2}}^{K}dx e^{-\frac{x^2}{2}+\frac{b}{1+b}\frac{x^2}{2}}
\leq \II e^{\frac{1}{2}bK^2}
\leq \left(1+\OO\left(bK^2\right)\right)\II,
\ee
where we used in the last inequality that $e^x\leq 1+xe^C$ for all $0\leq x\leq C$, while proceeding as in \eqv(2.2.lem1.6)-\eqv(2.2.lem1.6') to bound $\overline\II_2$,
\be
\overline\II_2
\leq 
\left(1+\OO(\sfrac{1}{n})\right)\frac{a}{K\varsigma_{k_1,k_2}} 
e^{-\frac{1}{2}\left(K^2/(1+b)- {a^2}/{\varsigma^2_{k_1,k_2}}\right)}\II .
\ee
The choice of $K$ now depends on the behaviour of $b$.  According to \eqv(2.2.lem1.11), the identity $s_{11}=\varsigma^2_{k_1,k_2}$ and \eqv(2.2.lem1.4)-\eqv(2.2.lem1.5), we have
\be
b=\frac{s_{12}}{s_{11}}
=
\frac{1}{\varsigma^2_{k_1,k_2}}\OO\left(2^{h\curlyvee h'-k_1}\right)
=
\OO\left(\frac{1}{n}2^{h\curlyvee h'-k_1}\right).
\Eq(2.2.lem1.18)
\ee
If $h\curlyvee h'<k_1$ then at best ${s_{12}}/{s_{11}}\leq\OO(1/n)$ and since ${a}/{\varsigma_{k_1,k_2}}=\OO(\sqrt{n})$, choosing  $K=3{a}/{\varsigma_{k_1,k_2}}$ yields  $\overline\II=\left(1+\OO(1)\right)\II$. By contrast, if  $h\curlyvee h'-k_1\leq -\e n$ for some $\e>0$ then, choosing $K=n$ yields $\overline\II=\left(1+\OO\left(n 2^{-\e n}\right)\right)\II$. Using \eqv(2.2.lem1.6*) to express $\II$ as a function of $\P\left(X^{\s}_{h}(k_1,k_2)\geq \g n\log2\right)$, and plugging the resulting bound on $\overline\II$ into the bound on $\JJ$ (see \eqv(2.2.lem1.17)), we finally get
\be
\JJ\leq (1+r)\left[\P\left(X^{\s}_{h}(k_1,k_2)\geq \g n\log2\right)\right]^2,
\Eq(2.2.lem1.14)
\ee
where $r=\OO(1)$ if $h\curlyvee h'<k_1$ and $r=\OO\left(n2^{-\e n}\right)$ if $h\curlyvee h'-k_1\leq -\e n$ for some $\e>0$.

It remains to bound the error term of \eqv(2.2.theo2.2), {\it i.e.}, to bound the last two quantities in \eqv(2.2.theo2.1). 
Firstly, proceeding just as in \eqv(2.2.lem1.4'), 
\be
\begin{split}
\boldsymbol{\vartheta}_{k_1,k_2}
&=\sum_{2^{k_1}<\log p\leq 2^{k_2}}{\E}\|\boldsymbol{W}^{\s}(p)\|^3
=\sum_{2^{k_1}<\log p\leq 2^{k_2}}{\E}\left[\left(W^{\s}_{h}(p)^2+W^{\s}_{h'}(p)^2\right)^{3/2}\right]
\\
&\leq
2^{3/2}\frac{1}{2^{k_1-1}}e^{-2^{k_1-1}} (1+o(1)).
\end{split}
\Eq(2.2.lem1.15)
\ee
Secondly, $\l_{k_1,k_2}\geq s_{11}-|s_{12}|$ (see the line above \eqv(2.2.lem1.12)), so that
\be
\l_{k_1,k_2}
\geq \varsigma^2_{k_1,k_2}+\OO\left(2^{h\curlyvee h'-k_1}\right).
\Eq(2.2.lem1.16)
\ee
Note that the bounds on $\vartheta_{k_1,k_2}$ and $\boldsymbol{\vartheta}_{k_1,k_2}$ in \eqv(2.2.lem1.4') and \eqv(2.2.lem1.15)  differ only from a constant prefactor. Note also that the second  term on the right-hand side of \eqv(2.2.lem1.16) is negligible compared to $\varsigma^2_{k_1,k_2}$.
It thus follows from  \eqv(2.2.lem1.4) and \eqv(2.2.lem1.5) that the term  ${\boldsymbol{\vartheta}_{k_1,k_2}}/{\l_{k_1,k_2}^{3/2}}$ is exactly of the same order as  ${\vartheta_{k_1,k_2}}/{\varsigma^3_{k_2,k_2}}$ in \eqv(2.2.lem1.7). From there, the proof is a re-run of the proof of item (i) of Lemma \thv(2.2.lem1). The proof of item (ii) is now complete.

The proof of part (iii) is immediate since the intersection of two events is contained in either of these two events. This is of course true for all $h,h'\in[0,1]$. Clearly, however, it can only be sharp if $|h-h'|\leq 2^{-\a_2n}$, namely before the \emph{branching point} of the underlying random walks associated to the processes $X^{\s}_h(k)$ and $X^{\s}_{h'}(k)$ (see the remark below Lemma \thv(2.2.lem1)). Note also that the strategy of proof based on the Berry-Esseen theorem fails here, since the covariance matrix is close to a projector and its smallest eigenvalue decays to zero exponentially fast as $n$ diverges.
With this, the proof of Lemma \thv(2.2.lem1) is complete.
\end{proof}

The proof of Lemma \thv(2.2.lem2) closely follows the strategy of  \cite{ABH} to prove an analogous result in the case of the truncated $\s=1/2$ process. 
It is based on two general a priori bounds, given as Lemma \thv(2.2.lem5') and Lemma \thv(2.2.prop2)(i), the latter lemma in turn being based on the preliminary Lemma \thv(2.2.lem3)(i).

In the rest of this section the notation $\varsigma^2_{k_1,k_2}$ introduced in \eqv(1.12) applies.

\begin{lemma}
    \TH(2.2.lem5')
For all $0\leq k_1<n$, $k_2\geq k_1+1$, $h\in[0,1]$, $\s>1/2$ and all $x>0$
\be
\P\left[X^{\s}_{h}(k_1,k_2)>x\right]
\leq \exp\biggl(-\frac{x^2}{2\varsigma^2_{k_1,k_2}}\biggr),
\Eq(2.2.lem5'.1)
\ee
where 
\be
2\varsigma^2_{k_1,k_2}=
\left(k_2\wedge n-k_1\right)\log 2+\OO\left(2^{-(n-k_2\wedge n)}\right)+\OO\left(\sqrt{2^{k_1}}e^{-c\sqrt{2^{k_1}}}\right).
\Eq(2.2.lem5'.2)
\ee
\end{lemma}

\begin{proof}[Proof of Lemma \thv(2.2.lem5')]
By \eqv(2.2.lem1.30) and independence, for all $\l>0$
\be
\E e^{\l X^{\s}_{h}(k_1,k_2)}
=
\prod_{2^{k_1}<\log p\leq 2^{k_2}}
\E \exp\left[\l W^{\s}_h(p)\right].
\Eq(2.2.lem5'.3)
\ee
Expanding the exponential in series, a direct calculation of the moments using the Wallis integral $$\frac{1}{2\pi}\int_0^{2\pi} d\theta \cos^{2m}\theta =\frac{(2m)!}{2^{2m}(m!)^2}$$ yields
\be
\E \exp\left[\l W^{\s}_h(p)\right]
=
\sum_{m=0}^{\infty}\frac{1}{(m!)^2}\left(\frac{\l^2}{4p^{2\s}}\right)^m.
\Eq(2.2.lem5'.4)
\ee
Since the Laplace transform of $X^{\s}_{h}(k_1,k_2)$ is independent of $h$, so is its distribution, as anticipated. 

Next, by \eqv(2.2.lem5'.3), \eqv(2.2.lem5'.4), the bound
\be
\sum_{m=0}^{\infty}\frac{1}{(m!)^2}\left(\frac{\l^2}{4p^{2\s}}\right)^m
\leq
\exp\left(\frac{\l^2}{4p^{2\s}}\right),
\Eq(2.2.lem5'.6)
\ee
and the identities \eqv(1.9) and \eqv(1.12), we obtain that for all $\l>0$
\be
\E \exp\left(\l X^{\s}_{h}(k_1,k_2)\right)
\leq 
\exp\left(\frac{\l^2}{2}\varsigma^2_{k_1,k_2}\right).
\Eq(2.2.lem5'.7)
\ee
Since by  Chebyshev's exponential inequality, for all $\l>0$
\be
\P\left[X^{\s}_{h}(k_1,k_2)>x\right]\leq\exp\left(\frac{\l^2}{2}\varsigma^2_{k_1,k_2} -\l x\right),
\Eq(2.2.lem5'.8)
\ee
taking
$
\l=x/\varsigma^2_{k_1,k_2}
$
in \eqv(2.2.lem5'.8) yields \eqv(2.2.lem5'.1).
Finally, it follows from \eqv(2.1.lem1.1bis) and \eqv(2.1.lem1.4bis) that $2\varsigma^2_{k_1,k_2}$ obeys \eqv(2.2.lem5'.2) under the assumptions on $k_1,k_2$ of the lemma.
The proof of Lemma \thv(2.2.lem5') is done.
\end{proof}

\begin{lemma}
    \TH(2.2.lem3)
Given $C>0$ 
we have, for all $0\leq k_1<n$, $k_2\geq k_1+1$,
$0<x<C\varsigma^2_{k_1,k_2}$, $0\leq y\leq 2^{2(k_2\wedge n)}$, 
all $h,h'$  such that $|h'-h|\leq 2^{-(k_2\wedge n)}$ and all $\s,\s'\in[\s_n,\s_{n-1})$,
\be
\begin{split}
&
\P\left[X^{\s}_{0}(k_1,k_2)\geq x, X^{\s'}_{h'}(k_1,k_2)-X^{\s}_{h}(k_1,k_2)\geq y\right]
\\
\leq & 
c\exp\biggl(
-\frac{x^2}{2\varsigma^2_{k_1,k_2}}
-\frac{c'y^{3/2}}{2^{k_2\wedge n}(|h'-h|+|\sigma'-\sigma|)}
\biggr),
\end{split}
\Eq(2.2.lem3.1)
\ee
for some positive constants $c,c'>0$ that depend on $C$, 
and  where $2\varsigma^2_{k_1,k_2}$ satisfies \eqv(2.2.lem5'.2).
\end{lemma}

\begin{proof}[Proof of Lemma \thv(2.2.lem3)] If $y$ in \eqv(2.2.lem3.1) is smaller than $C_02^{k_2\wedge n}(|h'-h|+|\s'-\s|)$ for some large constant $C_0>0$, then the $y$-term in \eqv(2.2.lem3.1) is at most $\OO(1)$ and the claim of the lemma reduces to that of Lemma \thv(2.2.lem5'). We can therefore assume that $y$ is larger than such a threshold. 
For any $\l_1, \l_2>0$, the left-hand side of \eqv(2.2.lem3.1)
is bounded from above by
\be
\E\left[\exp\left(\l_1X^{\s}_{0}(k_1,k_2)+\l_2\left(X^{\s'}_{h'}(k_1,k_2)-X^{\s}_{h}(k_1,k_2\right)\right)\right]
\exp(-\l_1x-\l_2y).
\Eq(2.2.lem3.2)
\ee
From this point onwards, the proof of Lemma \thv(2.2.lem3) is an extension of the proof of Lemma 2.7 in  \cite{ABH}, with the process $X^{\s}_{h}(k_1,k_2)$ defined in \eqv(1.11) for $\s>1/2$ substituted for the process with $\s=1/2$ treated there.
We need to evaluate 
\be
\log\E\left[\exp\left(\l_1X^{\s}_{0}(k_1,k_2)+\l_2\left(X^{\s'}_{h'}(k_1,k_2)-X^{\s}_{h}(k_1,k_2\right)\right)\right].
\Eq(2.2.lem3.3)
\ee
First, we note that, similarly to results (39)-(40) and (50)-(51) in \cite{ABH} 
\be
\begin{split}
&
\E\left[\exp\left(\l_1W^{\s}_{0}(p)+\l_2\left(W^{\s'}_{h'}(p)-W^{\s}_{h}(p)\right)\right)\right]
\\
=&
\frac{1}{2\pi}\int_0^{2\pi}d\theta\exp\left(\frac{\l_1}{p^{\s}}\cos \theta+\frac{\l_2}{p^{\s'}}\cos(\theta+h'\log p)-\frac{\l_2}{p^{\s}}\cos(\theta+h\log p)\right).
\end{split}
\Eq(2.2.lem3.4)
\ee
Using the identity
\be
\frac{1}{2\pi}\int_0^{2\pi} d\theta \exp\left(a\cos \theta +b\sin\theta\right)=I_0\left(\sqrt{a_p^2+b_p^2}\right),
\Eq(2.2.lem3.5)
\ee
where $I_0$ denotes the modified Bessel function of the first kind, as well as the identity
$
\cos(\theta+\eta)=\cos\theta\cos\eta-\sin\theta\sin\eta
$, 
we get for \eqv(2.2.lem3.4)
\be
a_p=\frac{\l_1}{p^{\s}}+\frac{\l_2}{p^{\s'}}\cos(h'\log p)-\frac{\l_2}{p^{\s}}\cos(h\log p),
\Eq(2.2.lem3.6)
\ee
and 
\be
b_p=-\frac{\l_2}{p^{\s'}}\sin(h'\log p)+\frac{\l_2}{p^{\s}}\sin(h\log p).
\Eq(2.2.lem3.7)
\ee
At this point we recall that $\log I_0(\sqrt{x})=x/4+\OO(x^2)$ and note that $\cos(h'\log p)-\cos(h\log p)=\OO (|h'-h|\log p)$ and $\sin(h'\log p)-\sin(h\log p)=\OO (|h'-h|\log p)$. Also note that
\be
\left|\frac{1}{p^{\s}}-\frac{1}{p^{\s'}}\right|\leq \frac{1}{p^{\s\wedge\s'}}\left|\s'-\s\right|\log p.
\Eq(2.2.lem3.8)
\ee
Assume that $\s\wedge\s'=\s$. Further assume that the parameters $\l_1$ and $\l_2$ satisfy
$\l _1\leq C_1$ and $\l_2(|h'-h|+|\s'-\s|)\leq C_2$ for some constants $C_1, C_2>0$ 
(observe that this will introduce a condition on $x$ and $y$) 
and let $0<c,c',c''<\infty$ be numerical constants whose values may change from line to line. Then,
 \be
\begin{split}
a_p^2+b_p^2
\leq 
A_p\equiv &\frac{1}{p^{2\s}}\left(\l_1+c\l_2(|h'-h| +|\s'-\s|)\log p\right)^2
\\
+
&\frac{1}{p^{2\s}}\left(c'\l_2(|h'-h| +|\s'-\s|)\log p\right)^2
\\
\leq 
&
\frac{\l_1^2}{p^{2\s}}+\frac{c}{p^{2\s}}\l_2(|h'-h| +|\s'-\s|)\log p
\\
+&
 \frac{c'}{p^{2\s}}\left(\l_2(|h'-h| +|\s'-\s|)\log p\right)^2,
\end{split}
 \ee
where we only used the boundedness assumption on $\l_1$, not on $\l_2$. The logarithm of the first line in \eqv(2.2.lem3.4) is then at most 
\be
\begin{split}
\frac{1}{4}A_p+\OO\bigl(A_p^2\bigr)
\leq 
&
\frac{\l_1^2}{4 p^{2\s}}+\frac{c}{p^{2\s}}\l_2(|h'-h| +|\s'-\s|)\log p
\\
+&
 \frac{c'}{4p^{2\s}}\left[\l_2(|h'-h| +|\s'-\s|)\log p\right]^2
\\
+&
\frac{c''}{p^{4\s}}\left\{1+\sum_{i=1}^4\left[\l_2(|h'-h| +|\s'-\s|)\log p\right]^i\right\}.
\end{split}
\Eq(2.2.lem3.9)
\ee
Summing over prime numbers $p$ such that $2^{k_1}< \log p\leq 2^{k_2}$ thus yields
\be
\begin{split}
&
\log\E\left[\exp\left(\l_1X^{\s}_{0}(k_1,k_2)+\l_2\left(X^{\s'}_{h'}(k_1,k_2)-X^{\s}_{h}(k_1,k_2\right)\right)\right]
\\
\leq &
\,c
+\sum_{2^{k_1}<\log p\leq 2^{k_2}}\frac{\l_1^2}{4p^{2\s}}
+c'\sum_{2^{k_1}<\log p\leq 2^{k_2}}\frac{\log p}{p^{2\s}}\l_2(|h'-h|+|\s'-\s|)
\\
&
+c'\sum_{2^{k_1}<\log p\leq 2^{k_2}}\frac{(\log p)^2}{p^{2\s}}\left(\l_2(|h'-h|+|\s'-\s|)\right)^2,
\end{split}
\Eq(2.2.lem3.10)
\ee
where we used that by  Lemma \thv(rk1bis) and under our boundedness assumption on $\l_2$
\be
\sum_{2^{k_1}<\log p\leq 2^{k_2}}\frac{c''}{p^{4\s}}\left\{1+\sum_{i=1}^4\left[\l_2(|h'-h| +|\s'-\s|)\log p\right]^i\right\}\leq c.
\Eq(2.2.lem3.10')
\ee
Therefore, by \eqv(1.12), \eqv(2.1.F10bis) and \eqv(2.1.F11bis),  \eqv(2.2.lem3.10) is bounded from above by
\be
c
+\frac{\l_1^2}{2}\varsigma^2_{k_1,k_2}
+c'2^{k_2\wedge n}\l_2(|h'-h|+|\s'-\s|)
+c'\left(2^{k_2\wedge n}\l_2(|h'-h|+|\s'-\s|)\right)^2.
\Eq(2.2.lem3.4bis)
\ee
Combining \eqv(2.2.lem3.2) and  \eqv(2.2.lem3.4bis), \eqv(2.2.lem3.1) is obtained by choosing
$
\l_1=x/\varsigma^2_{k_1,k_2}
$
and 
$
\l_2=c''y^{1/2}(2^{k_2\wedge n}(|h'-h|+|\s'-\s|))^{-1}
$
for some $c''>0$. The latter choice is meaningful since we assumed that $y$ is  larger than $C_02^{k_2\wedge n}(|h'-h|+|\s'-\s|)$. 
Note that, under the assumptions on $x$ and $y$ of the lemma, our choices of $\l_1$ and $\l_2$ guarantee that the boundedness assumptions $\l_1\leq C_1$ and $\l_2(|h'-h|+|\s'-\s|)\leq C_2$ are satisfied for any $C_1\geq C $ and $C_2= c''$.
The proof of Lemma \thv(2.2.lem3) is done.
\end{proof}

The next auxiliary result is the key to proving Lemma \thv(2.2.lem2).
\begin{lemma}
	\TH(2.2.prop2)
\item{(i)} Given $C>0$ we have, 
for all $0\leq k_1<n$, $k_2\geq k_1+1$,
$0<x<C\varsigma^2_{k_1,k_2}$, $2\leq a\leq 2^{2(k_2\wedge n)}-x$,
all $h\in[0,1]$ and all $\s\in[\s_n,\s_{n-1})$,
\be
\begin{split}
& 
\P\left(
\max_{\s'\in[\s_n,\s_{n-1})}\max_{h':|h-h'|\leq 2^{-(k_2\wedge n)}}X^{\s'}_{h'}(k_1,k_2)\geq x+a, X^{\s}_{h}(k_1,k_2)\leq x
\right)
\\
\leq 
& \, c\exp\biggl(
-\frac{x^2}{2\varsigma^2_{k_1,k_2}}-c'a^{3/2}
\biggr)
\end{split}
\Eq(2.2.prop2.1)
\ee
for some positive constants $c,c'>0$ that depend on $C$, and where $2\varsigma^2_{k_1,k_2}$ satisfies \eqv(2.2.lem5'.2).

\item{(ii)} There exist constants $0<C_0\leq C<\infty$ such that the following holds. 
For all $C_0\leq y\leq C 2^{2n}$ and all  intervals $I\subset[0,1]$, 
\be
\P\left(
\max_{\s,\s'\in[\s_n,\s_{n-1})}\max_{h,h'\in I}\left|X^{\s}_{h}-X^{\s'}_{h'}\right|\geq y
\right)
\leq 
c\exp\left\{-c'y^{3/2}\right\},
\Eq(2.2.lem6.1)
\ee
for some constants $c,c'>0$ that depend on $C$ and $C_0$.
\end{lemma}
\begin{proof}[Proof of  Lemma \thv(2.2.prop2)]
(i) As in the proof of Proposition 2.5 in  \cite{ABH}, from which we have been inspired, the lemma is proved by means of a chaining argument. First, since the process is translation invariant in $h$, we can take $h=0$. Since we can decrease $x$ and $a$, we can also assume without loss of generality that $x$ is an integer ({\it e.g.}, $\lfloor x\rfloor$) and $a\geq 1$. Define the events
\be
\BB_x=\left\{X^{\s}_{0}(k_1,k_2)\leq 0\right\},
\Eq(2.2.prop2.2)
\ee
for $m=0,1,\dots, x-1$
\be
\BB_m=\left\{X^{\s}_{0}(k_1,k_2)\in[x-m-1, x-m]\right\},
\Eq(2.2.prop2.3)
\ee
and for $m=0,1,\dots, x$
\be
\AA_m=\left\{
\max_{\s'\in[\s_n,\s_{n-1})}
\max_{h':0\leq h'\leq 2^{-(k_2\wedge n)}}\left\{X^{\s'}_{h'}(k_1,k_2)-X^{\s}_{0}(k_1,k_2)\right\}\geq m+a
\right\}.
\Eq(2.2.prop2.11)
\ee
Then, the probability appearing on the left-hand side of \eqv(2.2.prop2.1) is bounded above by
\be 
\sum_{m=0}^{x}
\P\left(\BB_m\cap\AA_m
\right)
\Eq(2.2.prop2.4).
\ee
Given $h'\in[0, 2^{-(k_2\wedge n)}]$ and $\s'\in[\s_{n},\s_{n-1})$, consider the sequences of dyadic rationals\footnote{The sequence $\s'_i$ and  $\s_n$ defined in \eqv(2.2.prop2.6) and \eqv(3.lem1.01) respectively should not be confused.}
\be
\begin{split}
h'_i\equiv h'_i(h')&={\textstyle \lceil 2^{i+k_2\wedge n}h'\rceil /2^{i+k_2\wedge n}},\\
\s'_i\equiv \s'_i(\s')&={\textstyle \lceil 2^{i+n+1}\s'\rceil /2^{i+n+1}},\\
\end{split}
\quad i\geq 1,
\Eq(2.2.prop2.5)
\ee
and set $h'_0=0$ and $\s'_0=\s$.
Since $y\leq \lceil y\rceil <y+1$, we have $h'_i\geq 0$, $\s'_j\geq \s$ and
\be
\begin{split}
&\lim_{i\rightarrow\infty}h'_i=h', \quad |h'_{i+1}-h'_{i}|\leq 2^{-(i+k_2\wedge n)} \textrm{ for all $i\geq 0$},\\
&\lim_{i\rightarrow\infty}\s'_i=\s',\quad |\s'_{i+1}-\s'_{i}|\leq 2^{-(i+n+1)} \textrm{ for all $i\geq 0$}.
\end{split}
\Eq(2.2.prop2.6)
\ee
Because the mapping $(\s',h')\mapsto X^{\s'}_{h'}$ is $\P$-almost surely continuous, 
\be
\lim_{i\rightarrow\infty}X^{\s'_{i+1}}_{h'_{i+1}}(k_1,k_2)=X^{\s'}_{h'}(k_1,k_2)\quad \P-a.s.,
\Eq(2.2.prop2.7)
\ee
and this enables us to write the difference  $X^{\s'}_{h'}(k_1,k_2)-X^{\s}_{0}(k_1,k_2)$ as the telescopic series
\be
X^{\s'}_{h'}(k_1,k_2)-X^{\s}_{0}(k_1,k_2)
=
\sum_{i=0}^{\infty}\left(X^{\s'_{i+1}}_{h'_{i+1}}(k_1,k_2)-X^{\s'_{i}}_{h'_{i}}(k_1,k_2)\right)\quad \P-a.s.
\Eq(2.2.prop2.8)
\ee
In what follows we denote by $\O'\subseteq \O$ the set of full measure on which \eqv(2.2.prop2.8) holds.
We next observe as in \cite{ABH} that since
$
\sum_{i=0}^{\infty}\frac{1}{2(i+1)^2}=\frac{\pi^2}{12}<1
$,
we have the inclusion 
\be
\begin{split}
&\left\{X^{\s'}_{h'}(k_1,k_2)-X^{\s}_{0}(k_1,k_2)\geq m+a\right\}
\\
\subseteq 
&
\,\bigcup_{i=0}^{\infty}\left\{
X^{\s'_{i+1}}_{h'_{i+1}}(k_1,k_2)-X^{\s'_{i}}_{h'_{i}}(k_1,k_2)\geq \frac{m+a}{2(i+1)^2}
\right\}.
\end{split}
\Eq(2.2.prop2.9)
\ee
Introducing the discrete sets (see \eqv(2.2.prop2.5))
\be
\begin{split}
\HH_{i+k_2\wedge n} & =\left\{ h'_i(h') : h'\in[0, 2^{-(k_2\wedge n)}]\right\},\\
\SS_{i+n+1} &=\left\{ \s'_i(\s') : \s'\in[\s_{n},\s_{n-1})\right\},
\end{split}
\Eq(2.2.prop2.10)
\ee
and writing $\s_2\sim \s_1$ (respectively, $h_2\sim h_1$)  if and only if $\s_2$ and $\s_1$ (respectively, $h_2$ and $h_1$) are two consecutive points in the dyadic grid, equation \eqv(2.2.prop2.9) implies that on $\O'$,
\be
\AA_m\subseteq
\bigcup_{i=0}^{\infty}
\bigcup_{{h_1,h_2\in\HH_{i+k_2\wedge n}}\atop{ h_2\sim h_1}}
\bigcup_{{\s_1,\s_2\in\SS_{i+n+1}}\atop{\s_2\sim \s_1}}
\left\{
X^{\s_2}_{h_2}(k_1,k_2)-X^{\s_1}_{h_1}(k_1,k_2)\geq \frac{m+a}{2(i+1)^2}
\right\}.
\Eq(2.2.prop2.12)
\ee
Since
\be
\left|\HH_{i+k_2\wedge n}\right|\leq 2^i,\quad 
\left|\SS_{i+n+1}\right|\leq 2^i,
\ee
for each $m=0,1\dots, x$, the summand $\P\left(\BB_m\cap\AA_m\right)$ in \eqv(2.2.prop2.4) is bounded above by
\be 
\sum_{i=0}^{\infty}2^{4i}
\sup_{{h_1\in\HH_{i+k_2\wedge n}}\atop{ h_2\sim h_1}}
\sup_{{\s_1\in\SS_{i+n+1}}\atop{\s_2\sim \s_1}}
\P\left(\BB_m\cap
\left\{
X^{\s_2}_{h_2}(k_1,k_2)-X^{\s_1}_{h_1}(k_1,k_2)\geq \frac{m+a}{2(i+1)^2}
\right\}
\right).
\Eq(2.2.prop2.13)
\ee
By assumption,
$
m+a\leq x+a\leq 2^{2(k_2\wedge n)}
$.
Thus, Lemma \thv(2.2.lem3) can be applied to get
\be
\begin{split}
\P\left(\BB_m\cap\AA_m\right)
\leq
&
\,c\sum_{i=0}^{\infty}2^{4i}\exp\biggl\{
-\frac{(x-m-1)^2}{2\varsigma^2_{k_1,k_2}}
-c'2^i\frac{(m+a)^{3/2}}{(i+1)^3}
\biggr\}
\\
\leq 
&
\,\tilde c\exp\biggl\{
-\frac{(x-m-1)^2}{2\varsigma^2_{k_1,k_2}}-c'(m+a)^{3/2}
\biggr\},
\end{split}
\Eq(2.2.prop2.14)
\ee
where we used in the first line that 
$
0<2^{k_2\wedge n}(|h_2-h_1|+|\s_2-\s_1|)\leq 2^{-i+1}
$
by \eqv(2.2.prop2.6), for any two consecutive elements $\s_2\sim \s_1$ and  $h_2\sim h_1$ of the dyadic grid,
with a halved constant $c'$. The second line corresponds to the $i=0$ term 
in the first line, with a constant $\tilde c>c$. Since
$
(m+a)^{3/2}\geq m^{3/2}+a^{3/2}
$,
substituting \eqv(2.2.prop2.14) in  \eqv(2.2.prop2.4), we finally arrive at
\be
\begin{split}
\sum_{m=0}^{x}\P\left(\BB_m\cap\AA_m\right)
&\leq 
\tilde c\, e^{-c' a^{3/2}}\sum_{m=0}^{x}
\exp\biggl\{
-\frac{(x-m-1)^2}{2\varsigma^2_{k_1,k_2}}-c'm^{3/2}
\biggr\}
\\
&\leq 
\tilde c\, e^{-c'a^{3/2}}\exp\biggl(-\frac{x^2}{2\varsigma^2_{k_1,k_2}}\biggr)
\sum_{m=0}^{x}e^{C(m+1)-c'm^{3/2}}
\\
&\leq 
\hat c\exp
\biggl(-\frac{x^2}{2\varsigma^2_{k_1,k_2}}-c'a^{3/2}\biggr),
\end{split}
\Eq(2.2.prop2.15)
\ee
where we used the assumption that $0<x<C\varsigma^2_{k_1,k_2}$, and where $\hat c>\tilde c$ depends on $C$. The proof of part (i) is now complete.

(ii) We may obviously apply $x=0$ in \eqv(2.2.lem3.1). By symmetry of the statement and of the field itself, \eqv(2.2.lem3.1) implies that
\be
\begin{split}
&
\P\left(\left| X^{\s'}_{h'}(0,k_2)-X^{\s}_{h}(0,k_2)\right|\geq y\right)
\\
\leq & 
c\exp\biggl(
-\frac{c'y^{3/2}}{2^{k_2\wedge n}(|h'-h|+|\sigma'-\sigma|)}
\biggr).
\end{split}
\Eq(2.2.lem3.1new)
\ee
Set
\be
\AA\equiv\left\{
\max_{h,h'\in I}\max_{\s'\in[\s_n,\s_{n-1})}\left[X^{\s}_{h}(0,k_2)-X^{\s'}_{h'}(0,k_2)\right]\geq y
\right\}, \quad y>0.
\ee
This event is nothing else than a symmetrised version of the event $\AA_m$ defined in \eqv(2.2.prop2.11) in the proof of Lemma \thv(2.2.prop2) with $a=y$, $m=0$. Proceeding exactly as in the proof of  \eqv(2.2.prop2.13),
\be
\P\left(\AA\right)
\leq 
\sum_{i=0}^{\infty}2^{4i}
\sup_{{h_1\in\HH_{i+k_2\wedge n}}\atop{ h_2\sim h_1}}
\sup_{{\s_1,\s_2\in\SS_{i+n+1}}\atop{\s_2\sim \s_1}}
\P\left(
\left|X^{\s_2}_{h_2}(0,k_2)-X^{\s_1}_{h_1}(0,k_2)\right|\geq \frac{y}{2(i+1)^2}
\right).
\nonumber
\ee
It then follows from  \eqv(2.2.lem3.1new) that for all $C_0\leq y\leq C2^{2(k_2\wedge n)}$
\be
\P\left(\AA\right)
\leq
c\sum_{i=0}^{\infty}2^{4i}\exp\left\{
-c'2^i
\frac{y^{3/2}}{4(1+i)^3}
\right\}
\leq
\tilde c\exp\left\{-c''y^{3/2}\right\},
\Eq(2.2.lem6.4)
\ee
and finally, the statement follows by letting $k_2$ go to infinity.
The proof of Lemma \thv(2.2.prop2) is  complete.
\end{proof}

We are now ready to give the proof of Lemma \thv(2.2.lem2).

\begin{proof}[Proof of  Lemma \thv(2.2.lem2)]
(i) Let us establish first that given $n\in\N$ and $C>0$, for any $0\leq k_1\leq  k_2-1$ with $k_1<n$, $0<y<C\varsigma^2_{k_1,k_2}$, 
$\s=\s_n$ and $h\in[0,1]$, there exists a constant $c>0$ depending on $C$ such that
\be
\P\left(
\max_{\s'\in[\s_n,\s_{n-1})}\max_{h':|h-h'|\leq 2^{-(k_2\wedge n)}}X^{\s'}_{h'}(k_1,k_2)
>y\right)
\leq 
c\exp\biggl(
-\frac{y^2}{2\varsigma^2_{k_1,k_2}}
\biggr).
\Eq(2.2.lem2.2)
\ee
Clearly, the left-hand side of \eqv(2.2.lem2.2) is bounded above by
\be
\begin{split}
&\P\left(
\max_{\s'\in[\s_n,\s_{n-1})}\max_{h':|h-h'|\leq 2^{-(k_2\wedge n)}}X^{\s'}_{h'}(k_1,k_2)>(y-2)+2, 
X^{\s}_{h}(k_1,k_2)\leq y-2
\right)
\\
&+
\P\left(
X^{\s}_{h}(k_1,k_2)> y-2
\right).
\end{split}
\Eq(2.2.lem2.3)
\ee
Applying Lemma \thv(2.2.prop2) with $x=y-2$ and $a=2$ to bound the first probability in \eqv(2.2.lem2.3), and Lemma \thv(2.2.lem5') to bound the second, we obtain \eqv(2.2.lem2.2).
Recalling the notations \eqv(2.2.1) and \eqv(2.2.1') as well as the expression \eqv(2.2.lem5'.2) of $2\varsigma^2_{k_1,k_2}$, and choosing 
$
C
>\g \frac{2\log 2}{\bar\a_2-\a_1}
$ 
and
$
y=\g n\log 2
$
in \eqv(2.2.lem2.2), we obtain the claim of \eqv(2.2.lem2.1). 

(ii) Similarly, \eqv(2.2.lem2.1') follows from the choice $h=h'$, $k_1=0$ ({\it i.e.}, $\a_1=0$), taking 
$
C
>\g 2\log 2
$ 
and
$y=\g n\log 2
$
in \eqv(2.2.lem2.2).
The proof of Lemma \thv(2.2.lem2) is done.
\end{proof}

\section{Proof of Theorem \thv(th:zeta)}
    \TH(S3)

Recall from the discussion following Lemma \ref{le:monday} that it is enough to consider only positive moments and instead of the full derivative just $\exp(\sum_{p\in\PP}U_pp^{-s})$ and the process \eqv(1.1). Thus, given $\b>0$, define
\be
{\FF}^{\s}(\b)
:=\frac{1}{\log\left[(\s-\frac{1}{2})^{-1}\right]}\log\int_0^1\exp\left(\b X^{\s}_{h}\right)d h, \quad \s>1/2.
\Eq(3.theo1.0)
\ee
Then, Theorem \thv(th:zeta) will be proven if we can prove the following:
\begin{theorem}
    \TH(3.prop1)
$\P$-almost surely, it holds for all $\b>0$ that
\be
\lim_{\s\rightarrow{\frac{1}{2}}^+}{\FF}^{\s}(\b)=f(\b),
\Eq(3.prop1.1)
\ee
where $f(\b)$ is defined in \eqv(1.theo1.3bis). The convergence also takes place  in $L^q(\O,\FF,\P)$ for any $0\leq q<\infty$.
\end{theorem}  
The proof of Theorem \thv(3.prop1) is inspired from  that of Proposition 2 in \cite{AT19}, in which a free energy-like convergence result is obtained for the truncated  $\s=1/2$ process, in expectation with respect to the law $\P$ of the process.

The strategy of the proof is, in fact, adapted from a classical scheme dating back to REM \cite{De1, OP84}. It consists in dividing the range of the process $X^{\s}_{h}$ into intervals, and  treating the contribution of each interval to the integral in \eqv(3.theo1.0) separately. This requires precise estimates of the number of points of  $X^{\s}_{h}$ that lie in a given interval.
Kistler's multiscale second moment method is then used to deal with the correlations of the process \cite{K15}.
These estimates are collected in Section \thv(S3.1). 
The proof of Theorem \thv(3.prop1) is finally given in  Section \thv(S3.2).

\subsection{Preparatory tools}
    \TH(S3.1)
The first result of this section is an \emph{a priori} almost sure bound on the maximum over $h$ of the sequence of processes $X^{\s_n}_{h}$ obtained by taking $\s=\s_n$ in \eqv(1.1).

Given $\varepsilon^{\star}>0$ and $\s>1/2$, define the sequence of events
\be
\O^\star_n(\varepsilon^{\star})=\bigcap_{\s\in[\s_{n},\s_{n-1})}\left\{\max_{h\in[0,1]}X^{\s}_{h}< (1+\varepsilon^{\star}) n\log 2\right\},\quad n\geq 1,
\Eq(3.lem1.1')
\ee
and set
\be
\O^\star(\varepsilon^{\star})=\bigcup_{n_0\geq 1}\bigcap_{n\geq n_0}\O^\star_n(\varepsilon^{\star}),
\Eq(3.lem1.1")
\ee
and
\be
\O^\star=\bigcap_{\varepsilon^{\star}>0}\O^\star(\varepsilon^{\star}).
\Eq(3.lem1.2)
\ee
\begin{lemma}
    \TH(3.lem1)
 $\P\left(\O^\star\right)=1$.
\end{lemma}

\begin{proof}[Proof of Lemma \thv(3.lem1)]  
We start by proving that for all sufficiently large $n$,
\be
\P\left([\O^\star_n(\varepsilon^{\star})]^c\right)\leq  C 2^{-2n\varepsilon^{\star}}
\Eq(3.lem1.4)
\ee
for some constant $C>0$ where, $[\O^\star_n(\varepsilon^{\star})]^c$ denotes the complement of $\O^\star_n(\varepsilon^{\star})$ in $\O$, namely,
\be
[\O^\star_n(\varepsilon^{\star})]^c
= 
\left\{\max_{h\in[0,1]}\max_{\s\in[\s_{n},\s_{n-1})}X^{\s}_{h}\geq (1+\varepsilon^{\star}) n\log 2\right\}.
\Eq(3.lem1.4')
\ee
The claim \eqv(3.lem1.4) will then follow from Lemma \thv(2.2.lem2) with $k_1=0$,  $k_2=\infty$, $\g=1+\varepsilon^{\star}$, 
and a suitable covering of the interval $[0,1]$. Specifically, we take $[0,1]= \cup_{1\leq \ell\leq 2^n}I_\ell$, where $I_\ell= [(\ell-1)2^{-n}, \ell 2^{-n}]$ for all $1\leq \ell\leq 2^n$. Then,
\be
\begin{split}
\P\left([\O^\star_n(\varepsilon^{\star})]^c\right)
&\leq 
\sum_{\ell=1}^{2^n}\P\left(\max_{h\in I_{\ell}}\max_{\s\in[\s_{n},\s_{n-1})}X^{\s}_{h}\geq (1+\varepsilon^{\star}) n\log 2\right)
\\
&\leq 
C 2^{-n[(1+\varepsilon^{\star})^2-1]},
\end{split}
\Eq(3.lem1.5)
\ee
where the first inequality is Boole's inequality and the second is \eqv(2.2.lem2.1) with 
$\g=1+\varepsilon^{\star}$, 
$\a_1=0$ and  $\bar\a_2=1$. This gives \eqv(3.lem1.4). Since
$
\sum_{n=1}^{\infty}\P\left(\O^c_n(\varepsilon^{\star})\right)<\infty
$
for all $\varepsilon^{\star}>0$, the Borel-Cantelli lemma ensures that
\be
\P\left(\O^\star(\varepsilon^{\star})\right)=1.
\ee
The claim of Lemma \thv(3.lem1) then follows in a classical way from the observation that 
$\O^\star(\varepsilon^{\star})$ is a decreasing family as $\varepsilon^{\star}\downarrow 0$.
\end{proof}

For  $\g\geq 0$ and $\s>1/2$ we introduce the sequence of Lebesgue measures of the set of $\g$-high points
\be
\MM^{\s}_n(\g)=\Leb\left\{h\in[0,1] : X^{\s}_{h}\geq \g n\log2\right\}, \quad n\geq 0,
\Eq(3.lem2.2)
\ee
and its normalised log-measure 
\be
\EE^{\s}_n(\g)
=\frac{\log \left(\MM^{\s}_n(\g)\right)}{n\log 2}.
\Eq(3.lem2.1)
\ee
We also set 
\be
\EE(\g)=-\g^2.
\Eq(3.lem2.3)
\ee
Given $\varepsilon>0$, we  next define the sequences of events, for $n\geq 0$,
\be
\begin{split}
\O^{+}_n(\varepsilon,\g)&=\bigcap_{\s\in[\s_{n},\s_{n-1})}\left\{\EE^{\s}_n(\g)-\EE(\g)\leq+\varepsilon \right\},\\
\O^{-}_n(\varepsilon,\g)&=\bigcap_{\s\in[\s_{n},\s_{n-1})}\left\{\EE^{\s}_n(\g)-\EE(\g)\geq-\varepsilon \right\},
\end{split}
\Eq(3.prop2.1)
\ee
and we set
\be
\O^{\pm}(\varepsilon,\g)=\bigcup_{n_0\geq 1}\bigcap_{n\geq n_0}\O^{\pm}_n(\varepsilon,\g),
\Eq(3.prop2.3)
\ee
and
\be
\O^{\pm}(\g)=\bigcap_{\varepsilon>0}\O^{\pm}(\varepsilon,\g).
\Eq(3.prop2.4)
\ee

The next proposition is the central result of this section and the lemma that follows is the key ingredient of its proof.
\begin{proposition}
    \TH(3.prop2)
    
\item (i) For all $\g>0$
\be
\P\left(\O^{+}(\g)\right)=1.
\Eq(3.prop2.5)
\ee
\item (ii) For all  $0<\g<1$ 
\be
\P\left(\O^{-}(\g)\right)=1.
\Eq(3.prop2.6)
\ee
\end{proposition}

\begin{lemma}
    \TH(3.lem2')
The following holds for all $n$ sufficiently large. 
\item (i) For all $\g>0$ and  all $\varepsilon>0$, there exists a constant $0<C<\infty$ such that
\be
\P\left([\O^{+}_n(\varepsilon,\g)]^c\right)\leq C  2^{-n\varepsilon}.
\Eq(3.lem2'.4)
\ee
\item (ii) For all  $0<\g<1$ and all $\varepsilon>0$, there exists a constant $0<C<\infty$ and a constant $c(\g,\varepsilon)>0$ that depends only on $\g$ and $\varepsilon$, such that
\be
\P\left[
\EE^{\s_n}_n(\g)-\EE(\g)\leq-\varepsilon 
\right]
\leq  C 2^{-nc(\g,\varepsilon)}.
\Eq(3.lem2'.4bis)
\ee
\item (iii)  There exist constants $0<C_1(\eta),C_2(\eta)>0$ such that for all $\eta>0$
\be
\P\left(\max_{h\in [0,1]}\max_{\s\in[\s_{n},\s_{n-1})}\left|X^{\s_n}_{h}-X^{\s}_{h}\right|\geq\eta n\right)
\leq 
C_1(\eta)e^{-C_2(\eta)n}.
\Eq(3.lem2'.01)
\ee
\end{lemma}

We begin by proving  Proposition \thv(3.prop2) assuming Lemma \thv(3.lem2') and then prove Lemma \thv(3.lem2').

\begin{proof}[Proof of Proposition \thv(3.prop2)]  
On the one hand, by part (i) of Lemma \thv(3.lem2'), for all $\varepsilon>0$ and all $\g>0$,
\be
\textstyle
\sum_{n=1}^{\infty}\P\left(\left[\O^{+}_n(\varepsilon,\g)\right]^c\right)<\infty.
\ee
Thus, the Borel-Cantelli lemma ensures that
\be
\P\left(\O^{+}(\varepsilon,\g)\right)=1.
\ee
On the other hand, by Borel-Cantelli lemma and part (ii) of Lemma \thv(3.lem2'), for all $\varepsilon>0$ and all $0<\g<1$, with $\P$-probability one, there exists a random $n_0<\infty$, such that for all $n\geq n_0$, we have 
\be
\EE^{\s_n}_n(\g)\geq \EE(\g)-\varepsilon.
\ee
Hence, by part (iii) of Lemma \thv(3.lem2') with $\eta=\varepsilon$, for all $\s\in[\s_{n},\s_{n-1})$
\be
\EE^{\s}_n(\g)\geq \EE(\g)-\varepsilon-\frac{\g\varepsilon}{\log 2}.
\ee
Thus, for all $\varepsilon>0$
\be
\P\left(\O^{-}_n\left(\varepsilon\bigl[1+\sfrac{\g}{\log 2}\bigr], \g\right)\right)=1.
\ee

The claim of Proposition \thv(3.prop2) now follows  from the observation that given $\g$,
both $\O^{+}(\varepsilon,\g)$ and $\O^{-}(\varepsilon,\g)$ are decreasing families of sets as $\varepsilon\downarrow 0$.
\end{proof}

\begin{proof}[Proof of Lemma \thv(3.lem2')] Let $n\geq 0$ be given.
We first prove item (i), namely that for all $\g>0$ and all $\varepsilon>0$
\be
\P\left(\max_{\s\in[\s_{n},\s_{n-1})}\MM^{\s}_n(\g)>2^{n[\EE(\g)+\varepsilon]}\right)\leq C  2^{-n\varepsilon},
\Eq(3.lem2.5)
\ee
for some consrant $C>0$. To do this, note first that for all $h\in[0,1]$ and $\s\in[\s_{n},\s_{n-1})$,
\be
\left\{X^{\s}_{h}\geq \g n\log2\right\}\subseteq\left\{\max_{\s\in[\s_{n},\s_{n-1})}X^{\s}_{h}\geq \g n\log2\right\}.
\ee
By this and \eqv(3.lem2.2),
\be
\begin{split}
\max_{\s\in[\s_{n},\s_{n-1})}\MM^{\s}_n(\g)
\leq 
\Leb\left\{h\in[0,1] : \max_{\s\in[\s_{n},\s_{n-1})}X^{\s}_{h}\geq \g n\log2\right\}.
\end{split}
\Eq(3.lem4'.3)
\ee
Therefore,
\be
\begin{split}
&\P\left(\max_{\s\in[\s_{n},\s_{n-1})}\MM^{\s}_n(\g)>2^{n[\EE(\g)+\varepsilon]}\right)
\\
\leq 
&
\P\left(\Leb\left\{h\in[0,1] : \max_{\s\in[\s_{n},\s_{n-1})}X^{\s}_{h}\geq \g n\log2\right\}>2^{n[\EE(\g)+\varepsilon]}\right).
\end{split}
\Eq(3.lem4'.4)
\ee
Using in turn Markov's inequality and Fubini's theorem, the second line of \eqv(3.lem4'.4) is at most
\be
\begin{split}
&
\leq
2^{n(-\EE(\g)-\varepsilon)}\int_{0}^1\P\left(\max_{\s\in[\s_{n},\s_{n-1})} X^{\s}_{h}\geq \g n\log2\right)dh
\\
&=2^{n(-\EE(\g)-\varepsilon)}\P\left(\max_{\s\in[\s_{n},\s_{n-1})} X^{\s}_{h}\geq \g n\log2\right)
\\
&\leq C'  2^{-n\varepsilon} ,
\end{split}
\Eq(3.lem2.6')
\ee
where the last line follows from \eqv(2.2.lem2.1') of Lemma \thv(2.2.lem2), together with \eqv(3.lem2.3), while in the second line we used that the distribution of $X^{\s}_{h}$ does not depend on $h$, 
as seen from the proof of Lemma \thv(2.2.lem5') (i). 
The claim of \eqv(3.lem2'.4) now follows.

Turning to item (ii), let us establish that under the more restrictive condition that $0<\g<1$, for all $\varepsilon>0$
\be
\P\left(\MM^{\s_n}_n(\g)<2^{n[\EE(\g)-\varepsilon]}\right)\leq C2^{-nc(\g,\varepsilon)},
\Eq(3.lem2.6)
\ee
where $C$ and $c(\g,\varepsilon)$ are as in \eqv(3.lem2'.4bis).
The proof relies on the Paley-Zygmund inequality, which states that if $Z\geq 0$ is a random variable with finite variance, 
and if $0\leq \eta \leq 1$, then 
\be
\P(Z>\eta \E Z)\geq (1-\eta)^{2}{\frac{(\E Z)^{2}}{\E(Z^{2})}}.
\Eq(3.lem2.7)
\ee
Due to the correlation structure of the process, we will not apply \eqv(3.lem2.7) directly to the quantity 
$\MM^{\s_n}_n(\g)$. Instead, we will follow the strategy developed in \cite{K15}, {\it i.e.}, we introduce a coarse-grained version of the process that allows us to gain enough independence to make efficient use of \eqv(3.lem2.7).

Given an integer $K$ and a real number $\d>0$, both of which will eventually depend on $\varepsilon$, define the sequence of $K$-level coarse-grained events
\be
\JJ^{\s_n}_h(m)=
\begin{cases}
\left\{
X^{\s_n}_{h}\left(\left\lceil\sfrac{m-1}{K} n\right\rceil, \left\lceil\sfrac{m}{K} n \right\rceil\right)
\geq \sfrac{(1+\d)}{K}\g n\log 2
\right\} 
& \text{ if $ m=1,\dots, K-1$}, \\
\left\{
X^{\s_n}_{h}\left(\left\lceil\sfrac{m-1}{K} n\right\rceil, \infty\right)
\geq \sfrac{(1+\d)}{K}\g n\log 2
\right\}
 &
\text{ if $ m=K$}.
\end{cases}
\Eq(3.lem2.9)
\ee
Using these events, we define the (random) set of all $h\in[0,1]$ for which all the events \eqv(3.lem2.9) occur,
\be 
\AA^{\s_n}(\g)=
\left\{h\in[0,1] : \prod_{2\leq m\leq K}\1_{\left\{\JJ^{\s}_h(\g,m)\right\}}=1
\right\}.
\Eq(3.lem2.10)
\ee
Note that for all $h\in\AA^{\s_n}$
\be
X^{\s_n}_{h}-X^{\s_n}_{h}\left(0, \left\lceil\sfrac{1}{K} n \right\rceil\right)\geq (1+\d)(1-\sfrac{1}{K})\g n\log 2.
\Eq(3.lem2.11)
\ee
Therefore, setting 
$
\d_K=\d(1-\sfrac{1}{K})-\sfrac{1}{K}
$
with $K>1+1/\d$ so that $\d_K>0$, and defining the events
\be
\begin{split}
\BB^{\s}_{\d_K}
=
&
\left\{h\in[0,1] : X^{\s}_{h}\left(0, \left\lceil\sfrac{1}{K} n  \right\rceil\right)\leq -\d_K\g n\log 2
\right\},
\\
\CC= & 
\left\{h\in[0,1] : X^{\s}_{h}\geq \g n\log 2
\right\},
\end{split}
\Eq(3.lem2.13')
\ee
we have
$
\AA^{\s_n}\subseteq\BB^{\s_n}_{\d_K}\bigcup\CC
$,
and so, observing that $\Leb(\CC)=\MM^{\s_n}_n(\g)$, as follows from \eqv(3.lem2.2),
\be
\MM^{\s_n}_n(\g)\geq \Leb(\AA^{\s_n})-\Leb(\BB^{\s_n}_{\d_K}).
\Eq(3.lem2.13)
\ee
To deal with $\Leb(\BB^{\s_n}_{\d_K})$, note that by Fubini's theorem  and part (ii) of Lemma \thv(2.2.lem2) together with the symmetry of 
$X^{\s}_{h}\left(0, \left\lceil\sfrac{1}{K} n  \right\rceil\right)$ 
\bea
\E\left[\Leb(\BB^{\s_n}_{\d_K})\right]
&\leq&
C 2^{-n K(\d_K\g)^2},
\Eq(3.lem2.14)
\eea
while by Markov's inequality
\be
\P\left(\Leb(\BB^{\s_n}_{\d_K})< 2^{n[\EE(\g)-\varepsilon]}\right)\geq 1-C 2^{-n[\g^2(K\d^2_K-1)-\varepsilon]}.
\Eq(3.lem2.15)
\ee
Since $\d_K\rightarrow \d$ as $K\rightarrow \infty$, this probability approaches $1$ exponentially fast as $n\rightarrow\infty$ for all $\g,\d>0$ provided that $K$ is sufficiently large.

Therefore, to prove \eqv(3.lem2.6), it remains to establish a lower bound on
\be
\P\left(
Z>2\, 2^{n[\EE(\g)-\varepsilon]}\right)
\Eq(3.lem2.16)
\ee 
where
\be
Z=\Leb(\AA^{\s_n}),
\Eq(3.lem2.16')
\ee
or, equivalently, setting 
\be
\eta_n=2\, 2^{n[\EE(\g)-\varepsilon]}/\E\left[Z\right],
\Eq(3.lem2.16'')
\ee
a lower bound on
\be
\P\left(Z>\eta_n\E\left[Z\right]\right).
\Eq(3.lem2.17')
\ee
This is where the Paley-Zygmund inequality \eqv(3.lem2.7) comes into play.

Let us first prove  that with the choice made in \eqv(3.lem2.16''), $\eta_n\rightarrow 0$ as $n\rightarrow\infty$. We thus need a lower bound on $\E\left[Z\right]$. 
Again, by Fubini's theorem, independence, and the fact that the distribution of $X^{\s_n}_{h}(k_1,k_2)$ does not depend on $h$,
\be
\E\left[\Leb(\AA^{\s_n})\right]=\int_{0}^1\prod_{m=2}^K\P\left[\JJ^{\s_n}_h(m)\right]dh=\prod_{m=2}^K\P\left[\JJ^{\s_n}_h(m)\right].
\Eq(3.lem2.17)
\ee
By the lower bound of part (i) of Lemma \thv(2.2.lem1), for all $2\leq m\leq K$
\be
\P\left[\JJ^{\s_n}_h(m)\right]\geq \sfrac{C_1}{\sqrt{n}}2^{-\frac{n\g^2(1+\d)^2}{K}}
\Eq(3.lem2.18)
\ee
for some constant $C_1>0$ that depends only on $K$, $\g$ and $a$. Plugging \eqv(3.lem2.18) in \eqv(3.lem2.17)
\be
\E\left[\Leb(\AA^{\s_n})\right]\geq
\bigl(\sfrac{C_1}{\sqrt{n}}\bigr)^{K-1}2^{-n\g^2(1+\d)^2(1-\frac{1}{K})}.
\Eq(3.lem2.19)
\ee
Thus, if 
$\g^2(1+\d)^2(1-\frac{1}{K})\leq \g^2+\frac{\varepsilon}{2}$, then
\be
\E\left[\Leb(\AA^{\s_n})\right]
\geq
\bigl(\sfrac{C_1}{\sqrt{n}}\bigr)^{K-1}2^{n[\EE(\g)-\frac{\varepsilon}{2}]}.
\Eq(3.lem2.20)
\ee
Substituting \eqv(3.lem2.20) in \eqv(3.lem2.16'')  implies that 
\be
\eta_n\leq 2\, 2^{-n\frac{\varepsilon}{2}}\OO({n}^{(K-1)/2})\rightarrow 0 \,\,\,\text{as} \,\,\, n\rightarrow\infty,
\Eq(3.lem2.20')
\ee 
as desired.

We next establish an upper bound on $\E\left[Z^2\right]$. 
Once again, by Fubini's theorem and independence
\be
\E\left[\Leb(\AA^{\s_n})^2\right]=\int_{0}^1\int_{0}^1\prod_{m=2}^K\P\left[\JJ^{\s_n}_h(m)\cap \JJ^{\s_n}_{h'}(m)\right]dhdh'
=\sum_{i=1}^{4}\II_i
\Eq(3.lem2.23)
\ee
where we decomposed the integration domain into four domains, namely, for  $1\leq i\leq 4$, $\II_i$ are the integrals over the domains
\be
\begin{split}
\DD_1&=\left\{(h,h') : 2^{-n\frac{1}{2K}}<|h-h'|\leq 1\right\},
\\
\DD_2&=\left\{(h,h') : 2^{-n\frac{1}{K}}< |h-h'|\leq 2^{-n\frac{1}{2K}}\right\},
\\
\DD_3&=\bigcup_{2\leq r\leq K}\DD_{3,r}, \quad \DD_{3,r}=\left\{(h,h') :  2^{-n\frac{r}{K}}< |h-h'|\leq 2^{-n\frac{(r-1)}{K}}\right\},
\\
\DD_4&=\left\{(h,h') :  |h-h'|\leq 2^{-n}\right\}.
\end{split}
\Eq(3.lem2.24)
\ee

To estimate $\II_1$, note that for each for $2\leq m\leq K$, 
$
h\curlyvee h'-\left\lceil\sfrac{m-1}{K} n\right\rceil
\leq   \left\lceil\sfrac{1}{2K} n\right\rceil
-\left\lceil\sfrac{1}{K} n\right\rceil
\leq -\frac{n}{2K} +1
$
on $\DD_1$.
Hence, for each $2\leq m\leq K$, item (ii) of  Lemma \thv(2.2.lem1) applies with an exponentially small error term 
of the form $\OO\left(n2^{-n/2K}\right)$.
This yields
\begin{eqnarray}\nonumber
\II_1
&\leq& 
\left(1+(K-1)\OO\left(n2^{-\frac{n}{2K}}\right)\right)
\prod_{m=2}^K\left[\P\left(\JJ^{\s_n}_h(m)\right)\right]^2\\
&=&  
\left(1+(K-1)\OO\left(n2^{-\frac{n}{2K}}\right)\right)
\left(\E\left[\Leb(\AA^{\s_n})\right]\right)^2,
\Eq(3.lem2.27)
\end{eqnarray}
where the last equality is the identity 
\be
\E\left[\Leb(\AA^{\s_n})\right]=\int_{0}^1\prod_{m=2}^K\P\left[\JJ^{\s_n}_h(m)\right]dh=\prod_{m=2}^K\P\left[\JJ^{\s_n}_h(m)\right],
\Eq(3.lem2.17'')
\ee
which is proven exactly as \eqv(3.lem2.17).

Turning to $\II_2$, we have that for each $2\leq m\leq K$, 
$
h\curlyvee h'-\left\lceil\sfrac{m-1}{K} n\right\rceil
\leq -n\frac{m-2}{K}+1
$
on $\DD_2$.
Thus item (ii) of Lemma \thv(2.2.lem1) holds again, but this time with an error term of order
$\OO(1)$ for $m=2$ and 
$\OO\left(n2^{-n/K}\right)$
for $3\leq m\leq K$.  This gives 
\be
\II_2
\leq (1+\OO(1))2^{-n\frac{1}{2K}}\prod_{m=2}^K\left[\P\left(\JJ^{\s_n}_h(m)\right)\right]^2
=  (1+\OO(1))2^{-n\frac{1}{2K}}\left(\E\left[\Leb(\AA^{\s_n})\right]\right)^2.
\Eq(3.lem2.25)
\ee

To estimate $\II_4$, we use instead that item (iii) of Lemma \thv(2.2.lem1) applies for each $2\leq m\leq K$, so that
\be
\begin{split}
\II_4 
&\leq 2^{-n}\prod_{m=2}^K\P\left(\JJ^{\s_n}_h(m)\right)
=  2^{-n}\E\left[\Leb(\AA^{\s_n})\right]
\\
&\leq 
\bigl(\sfrac{\sqrt{n}}{C_1}\bigr)^{K-1}2^{-n[1-\g^2(1+\d)^2(1-\frac{1}{K})]}\left(\E\left[\Leb(\AA^{\s_n})\right]\right)^2,
\end{split}
\Eq(3.lem2.26)
\ee
where the first equality is again \eqv(3.lem2.17''), 
while the last follows from 
\eqv(3.lem2.19).

It remains to deal with the third case. Here
$
\II_3=\sum_{2\leq r\leq K}\II_{3,r}
$,
where $\II_{3,r}$ is the restriction of the integral in \eqv(3.lem2.23) to the domain $\DD_{3,r}$. 
For each fixed $2\leq r\leq K$, $\II_{3,r}$ is bounded using either part (ii) or (iii) of  Lemma \thv(2.2.lem1), depending on whether $m>r$ or $m\leq r$.
More precisely, observing that for each $m>r$, 
$
h\curlyvee h'-\left\lceil\sfrac{m-1}{K} n\right\rceil
<-n\frac{r-(m-1)}{K}
$,
part (ii) of Lemma \thv(2.2.lem1) holds with an error term of order $\OO(1)$ for $m=r+1$ and $\OO\left(n2^{-n/K}\right)$ for $m>r+1$. 
Thus,
\be
\II_{3,r}\leq
(1+\OO(1))
2^{-n\frac{(r-1)}{K}}\bigl(1-2^{-n\frac{1}{K}}\bigr)
\prod_{m=2}^r\P\left(\JJ^{\s_n}_h(m)\right)
\prod_{m=r+1}^K\left[\P\left(\JJ^{\s_n}_h(m)\right)\right]^2.
\Eq(3.lem2.28)
\ee
Summing over $r$ and using \eqv(3.lem2.17'') to reconstruct $\left(\E\left[\Leb(\AA^{\s_n})\right]\right)^2$,
\be
\II_3 
\leq
(1+\OO(1))
\left(\E\left[\Leb(\AA^{\s_n})\right]\right)^2\sum_{2\leq r\leq K}2^{-n\frac{(r-1)}{K}}\prod_{m=2}^r\left[\P\left(\JJ^{\s_n}_h(m)\right)\right]^{-1}.
\Eq(3.lem2.29)
\ee
Using \eqv(3.lem2.18) to express the last product, the sum in \eqv(3.lem2.29) is bounded above by
\be
\begin{split}
\sum_{2\leq r\leq K}2^{-n\frac{(r-1)}{K}}\prod_{m=2}^r\left[\P\left(\JJ^{\s_n}_h(m)\right)\right]^{-1}
&\leq
\sum_{2\leq r\leq K}\left(\sfrac{\sqrt{n}}{C_1}\right)^{r-1}2^{-n\frac{(r-1)}{K}[1-\g^2(1+\d)^2]}
\\
&\leq
C\sqrt{n}2^{-n \left[\frac{1}{K}\left(1-\g^2(1+\d)^2\right)\right]},
\end{split}
\Eq(3.lem2.30)
\ee
where the last inequality holds provided that $1-\g^2(1+\d)^2>0$, for all $n$ sufficiently large.
Here and below, $0<C<\infty$ denotes a constant that may depend on the parameters $K,\d$ and $\g$, and whose value may vary from line to line.
Substituting \eqv(3.lem2.30) into  \eqv(3.lem2.29) then gives
\be
\II_3 
\leq
C\sqrt{n}2^{-n \left[\frac{1}{K}\left(1-\g^2(1+\d)^2\right)\right]}
\left(\E\left[\Leb(\AA^{\s_n})\right]\right)^2.
\Eq(3.lem2.31)
\ee

Collecting the bounds \eqv(3.lem2.25), \eqv(3.lem2.26) and \eqv(3.lem2.31) we get that
 if $\g^2(1+\d)^2<1$, for any positive integer $K$ and all sufficiently large $n$
\be
\begin{split}
\II_2 + \II_3 + \II_4 
&\leq 
C\sqrt{n}2^{
-n\min\left\{
\frac{1}{2K}, \frac{1}{K}\left(1-\g^2(1+\d)^2\right)
\right\}
}
\left(\E\left[\Leb(\AA^{\s_n})\right]\right)^2.
\end{split}
\Eq(3.lem2.32)
\ee
Inserting \eqv(3.lem2.32) and  \eqv(3.lem2.27) into \eqv(3.lem2.23), and recalling \eqv(3.lem2.16'),
\be
\E\left[Z^2\right]
\leq 
 (1+\rho_n(\d,\g,K))\left(\E\left[Z\right]\right)^2.
\Eq(3.lem2.33'')
\ee
where 
\be
\rho_n(\d,\g,K)=
C\sqrt{n}2^{
-n\frac{1}{K}\min\left\{
\frac{1}{2},1-\g^2(1+\d)^2
\right\}
}.
\Eq(3.lem2.33')
\ee

For any given $0< \g<1$ and $\varepsilon >0$, one has to choose $\d$ and $K$ so that $\g^2(1+\d)^2 <1$ and $\g^2(1+\d)^2(1-\frac{1}{K})\leq \g^2+\frac{\varepsilon}{2}$, in addition to the conditions on $\d_K=\d-\frac{1+\d}{K}$ seen above, namely 
$\delta_K>0$ and $\g^2(K\d_K^2-1)>\varepsilon$ in \eqv(3.lem2.15).  This can be realized by simply requiring that $\g^2(1+\d)^2\leq \g^2+\frac{\varepsilon}{2} <1$, {\it i.e.},  choosing $0<\d <\min\{\g^{-1},(1+\frac{\varepsilon}{2\g^2})^{1/2}\}-1$. One can then take $K$ large enough so that the constraints on $\d_K$ are satisfied.

Applying \eqv(3.lem2.7) to the variable  $Z$ defined in \eqv(3.lem2.16'),
 choosing the sequence $\eta_n$ as in \eqv(3.lem2.20'), and using the bounds \eqv(3.lem2.33') and \eqv(3.lem2.33''),  we get
\be
\begin{split}
\P\left(Z>2\, 2^{n[\EE(\g)-\varepsilon]}\right)
\geq 
&
(1-\eta_n)^{2} /(1+\rho_n(\d,\g,K))
\\
\geq 
&1
-C\left(
2^{-n\frac{\varepsilon}{2}}+
\sqrt{n}2^{
-n\frac{1}{K}\min\left\{
\frac{1}{2},1-\g^2(1+\d)^2
\right\}}
\right).
\end{split}
\Eq(3.lem2.33")
\ee
The claim of  \eqv(3.lem2.6) for $n$ large enough now follows from \eqv(3.lem2.13), \eqv(3.lem2.15) and \eqv(3.lem2.33").

It remains to prove part (iii) of the lemma. Consider the covering  $[0,1]= \cup_{1\leq j\leq 2^n}I_j$, 
with $I_j= [(j-1)2^{-n}, j2^{-n}]$. By part (ii) of Lemma \thv(2.2.prop2), for each $1\leq j\leq 2^n$ and $\eta>0$,
\be
\P\left(
\max_{h\in I_j}\max_{\s\in[\s_n,\s_{n-1})}\left|X^{\s_n}_{h}-X^{\s}_{h}\right|\geq \eta n
\right)
\leq 
c\exp\left\{-c'(\eta n)^{3/2}\right\}.
\ee
The claim of \eqv(3.lem2'.01) follows by noting that 
\be
2^nc\exp\left\{-c'(\e n)^{3/2}\right\}\leq C_1(\eta)e^{-C_2(\eta)n}
\ee
for some constants $C_1(\eta), C_2(\eta)>0$.
The proof of  Lemma \thv(3.lem2') is complete. 
\end{proof}

\subsection{Proof of Theorem \thv(3.prop1)}
    \TH(S3.2)
We prove the almost sure convergence and the mean value convergence statements  of Theorem \thv(3.prop1) separately.

\begin{proof}[Proof of Theorem \thv(3.prop1): almost sure convergence] 
Our aim is to establish bounds on the integral in \eqv(3.theo1.0).
To this end, we introduce the following decomposition of the real line. 
Given an integer $M\in\N$, set 
$\g_j\equiv\g_j(M)$ where
\be
\g_j(M)=\frac{j}{M}, \quad j=0,\dots, M+1,
\Eq(3.theo1.4)
\ee
and given $n\in\N$, for all $\s\in[\s_n,\s_{n-1})$, define 
\be
\begin{split}
\Delta_{-1} &= \bigl]-\infty, 0\bigr[
\\
\Delta_j & =\bigl[\g_{j}n\log 2, \g_{j+1}n\log 2\bigr[\,\,,\quad j=0,\dots, M,
\\
\Delta_{M+1} & =\bigl[\g_{M+1}n\log 2, \infty\bigr[\,\,.
\end{split}
\Eq(3.theo1.5)
\ee
Using \eqv(3.lem2.2), we define
\be
\LL\left(\Delta_j \right)
\equiv \Leb\left\{h\in[0,1] : X^{\s}_{h}\in\Delta_j \right\}
=\MM^{\s}_n(\g_{j})-\MM^{\s}_n(\g_{j+1}).
\Eq(3.theo1.6)
\ee
In order to control  the Lebesgue measures $\LL\left(\Delta_j \right)$, we apply Proposition \thv(3.prop2) with $\g=\g_j$ for each $0\leq j\leq M$ and  $M\geq 1$. Note however that since item (ii) of the proposition is only  valid for $0<\g<1$, we do not have a lower bound on these Lebesgue measures when $j\geq M$.
This prompts us to set
\be
\begin{split}
\O^+& 
=\bigcap_{M\in\N}\bigcap_{0\leq j\leq M}\O^{+}(\g_j(M)),
\\
\O^-& 
=\bigcap_{M\in\N}\bigcap_{0\leq j\leq M-1}\O^{-}(\g_j(M)),
\end{split}
\Eq(3.theo1.7)
\ee
and to define, recalling the definition \eqv(3.lem1.2) of $\O^\star$
\be
\wh\O=\O^\star\cap \O^+\cap \O^-.
\ee
Then, by Lemma \thv(3.lem1) and Proposition \thv(3.prop2)
\be
\P\bigl(\wh\O\bigr)=1.
\Eq(3.theo1.8)
\ee
Note that by \eqv(3.lem2.1), for $\g=\g_j$, the sets 
$\O^{+}_n(\varepsilon,\g)$ from \eqv(3.prop2.1) can be written as
\be
\begin{split}
\O^{+}_n(\varepsilon,\g_j)&=\bigcap_{\s\in[\s_{n},\s_{n-1})}\left\{\MM^{\s}_n(\g_j)\leq 2^{n[\EE(\g_j)+\varepsilon]}\right\}\\
\O^{-}_n(\varepsilon,\g_j)&=\bigcap_{\s\in[\s_{n},\s_{n-1})}\left\{\MM^{\s}_n(\g_j)\geq 2^{n[\EE(\g_j)-\varepsilon]}\right\}
\end{split}
,\quad n\geq 1.
\Eq(3.prop2.1bis)
\ee

We begin by proving upper and lower bounds on ${\FF}^{\s}(\b)$. 

\smallskip
\noindent\emph{Upper bound on ${\FF}^{\s}(\b)$.}  We use the partition \eqv(3.theo1.5) to decompose the integral in \eqv(3.theo1.0).
Firstly, given $n\in\N$ we have that on $\O^{+}_n(\varepsilon,\g_j)$, for all $\s\in[\s_{n},\s_{n-1})$
\bea
\sum_{j=0}^{M}\int_0^1e^{\b X^{\s}_{h}}\1_{\left\{X^{\s}_{h}\in\Delta_j \right\}}dh
&\leq &
\sum_{j=0}^{M}e^{\b\g_{j+1}n\log 2}\MM^{\s}_n(\g_{j})
\Eq(3.theo1.9)
\\
&\leq &
\sum_{j=0}^{M}e^{\left\{\b\g_{j+1}+\EE(\g_{j})+\varepsilon\right\}n\log 2}
\Eq(3.theo1.10)
\\
&\leq &
(M+1)\,e^{\max_{0\leq j\leq M}\left\{\b\g_{j+1}+\EE(\g_{j})+\varepsilon\right\}n\log 2},
\Eq(3.theo1.11)
\eea
where we used that by \eqv(3.theo1.6),
$
\LL\left(\Delta_j \right)\leq \MM^{\s}_n(\g_{j})
$
which, on $\O^{+}_n(\varepsilon,\g_j)$, is bounded above as in \eqv(3.prop2.1bis).
To deal with the interval $\Delta_{M+1}$, we note that by Lemma \thv(3.lem1), taking $\varepsilon^{\star}=1/M$ in
\eqv(3.lem1.1') (and replacing the union over $\varepsilon^{\star}>0$ in \eqv(3.lem1.2) by a union over $M\geq 1$),
\be
\int_0^1e^{\b X^{\s}_{h}}\1_{\left\{X^{\s}_{h}\in\Delta_{M+1} \right\}}dh=\int_0^1e^{\b X^{\s}_{h}}\1_{\left\{X^{\s}_{h}\in\Delta_{M+1} \right\}}
\1_{\left[\O^{\star}_{n}\left(\frac{1}{M}\right)\right]^c}dh=0
\Eq(3.theo1.12)
\ee
with $\P$-probability one, for all but a finite number of indices $n$ and all $\s\in[\s_{n},\s_{n-1})$.
Finally,
\be
\int_0^1e^{\b X^{\s}_{h}}\1_{\left\{X^{\s}_{h}\in\Delta_{-1}\right\}}dh\leq 1.
\Eq(3.theo1.13)
\ee
Gathering our bounds, we conclude that, by virtue of \eqv(3.theo1.8), with $\P$-probability one, for all but a finite number of indices $n$ and all $\s\in[\s_{n},\s_{n-1})$
\be
{\FF}^{\s}(\b)\leq 
\frac{n\log 2}{\log\left[(\s-\frac{1}{2})^{-1}\right]} 
\left(\max_{0\leq j\leq M}\left\{\b\g_{j}+\EE(\g_{j})\right\}+\varepsilon+\frac{\beta}{M}+\frac{\log (M+1)}{n\log 2}+o(1)\right),
\Eq(3.theo1.14)
\ee
where the $o(1)$ term is exponentially small in $n$, and where by \eqv(1.4bis) and \eqv(2ns)
\be
\frac{n\log 2}{\log\left[(\s-\frac{1}{2})^{-1}\right]} =1+\OO\left(n^{-1}\right).
\Eq(3.theo1.14')
\ee

\smallskip
\noindent\emph{Lower bound on ${\FF}^{\s}(\b)$.}  Choose any $0\leq j\leq M-1$.
Then, given $n\in\N$ we have that on $\O^{-}_n(\varepsilon,\g_j)\cap\O^+_n(\varepsilon,\g_{j+1})$, for all $\s\in[\s_{n},\s_{n-1})$
\bea
\int_0^1e^{\b X^{\s}_{h}}dh
&\geq&
\int_0^1e^{\b X^{\s}_{h}}\1_{\Delta_j}dh
\Eq(3.theo1.15)
\\ 
&\geq&
e^{\b\g_{j}n\log 2}\left[\MM^{\s}_n(\g_{j})-\MM^{\s}_n(\g_{j+1})\right]
\nonumber
\\ 
&\geq&
e^{\left\{\b\g_{j}+\EE_n(\g_{j})-\varepsilon\right\}n\log 2}
\left(1-2^{-n\left[\frac{2j+1}{M^2}-2\varepsilon\right]}\right)
\nonumber
\\
&\geq&
e^{\left\{\b\g_{j}+\EE(\g_{j})-\varepsilon\right\}n\log 2}
\left(1-2^{-n\left[(1/M)^2-2\varepsilon\right]}\right),
\Eq(3.theo1.17)
\eea
where we used the lower bound on $\MM^{\s}_n(\g_{j})$ from $\O^{-}_n(\varepsilon,\g_j)$ 
and the upper bound on $\MM^{\s}_n(\g_{j+1})$ from $\O^+_n(\varepsilon,\g_{j+1})$ in \eqv(3.prop2.1bis).  Without loss of generality, we can 
assume that $\varepsilon$ is small enough so that
$
2\varepsilon< (1/M)^2
$. 
Since this holds for any $0\leq j\leq M-1$, 
choosing $j$ in \eqv(3.theo1.15) as the index at which the maximum over $0\leq j\leq M-1$ of the set 
$\left\{\b\g_{j}+\EE(\g_{j})\right\}$ is reached, we obtain, again by virtue of \eqv(3.theo1.8), that 
\be
{\FF}^{\s}(\b)\geq
\frac{n\log 2}{\log\left[(\s-\frac{1}{2})^{-1}\right]} 
\left(
 \max_{0\leq j\leq M-1}\left\{\b\g_{j}+\EE(\g_{j})\right\}-\varepsilon+o(1)
 \right)
\Eq(3.theo1.18)
\ee
with $\P$-probability one, for all but a finite number of indices $n$ and all $\s\in[\s_{n},\s_{n-1})$. Here again, the $o(1)$ term is exponentially small in $n$. 

Because $x\mapsto \b x+\EE(x)$, $x\in [0,1]$, is a second order polynomial, it easily follows from \eqv(3.theo1.14)  and \eqv(3.theo1.18) by first letting $n$ go to infinity and then letting $M$ go to infinity that
\be
\lim_{\s\rightarrow\frac{1}{2}^+}{\FF}^{\s}(\b)=\max_{x\in[0,1]}\left\{\b x+\EE(x)\right\}=f(\b)\quad \P-\text{almost surely},
\Eq(3.theo1.20bis)
\ee
where $f(\b)$ is defined in \eqv(1.theo1.3bis).

In fact we can say more, namely, that \eqv(3.theo1.20bis) holds $\P$-almost surely for all $\b>0$ simultaneously. This comes from the fact that $\wh\O$ in \eqv(3.theo1.8) does not depend on $\beta$. Alternatively, one can use the fact that the integral means spectrum of an analytic function is convex in $\beta$.
The proof of the almost sure convergence claim of Theorem \thv(3.prop1)  is now complete.\end{proof}

\begin{proof}[Proof of Theorem \thv(3.prop1): convergence in mean of order $1\leq q<\infty$.] 
Let us establish that for any $0\leq q<\infty$
\be
\textstyle
\lim_{\s\rightarrow\frac{1}{2}^+}\E \left(|{\FF}^{\s}(\b)-f(\b)|^q\right)=0,
\Eq(3.theo1.25)
\ee
where ${\FF}^{\s}(\b)$ is as defined in \eqv(3.theo1.0). This will be deduced from the almost sure convergence 
\eqv(3.theo1.20bis) and the following criterion.

\begin{lemma}[Mean Convergence Criterion]
    \TH(3.lem5)
If the random variables $\{\left|{\FF}^{\s}(\b)\right|^q, \s> 1/2\}$ are uniformly integrable (UI) for some $q>0$, that is
\be
\lim_{a\rightarrow\infty} \sup_{\s>1/2}\E\left[\left|{\FF}^{\s}(\b)\right|^q\1_{\left\{\left|{\FF}^{\s}(\b)\right|^q>a\right\}}\right]=0,
\Eq(3.lem5.1)
\ee
and 
$
\lim_{\s\rightarrow\frac{1}{2}^+}{\FF}^{\s}(\b)=f(\b)
$ 
in $\P$-probability, then
$
\lim_{\s\rightarrow\frac{1}{2}^+}\E \left(|{\FF}^{\s}(\b)-f(\b)|^q\right)=0.
$ 
\end{lemma}

\begin{proof} This is a direct adaptation of the proof of the mean convergence criterion classically stated for sequences of UI random variables (see, {\it e.g.}, Theorem 3 (i) p.~100 in \cite{CT}).
\end{proof}

As almost sure convergence implies convergence in probability, it remains to prove that:
\begin{lemma}
    \TH(3.lem6)
The random variables  $\{\left|{\FF}^{\s}(\b)\right|^q, 1/2<\s\leq 1\}$ are uniformly integrable for any $0\leq q<\infty$.
\end{lemma}
\begin{proof}[Proof of Lemma \thv(3.lem6)] Throughout the proof $n\equiv n(\s)$ as in \eqv(1.4bis). Let  $1\leq q<\infty$ be given (the case $0<q<1$ will then follow by the H\"older inequality).  It suffices to prove that for any $\varepsilon>0$, there exists $a>0$ large enough so that for all $1/2<\s\leq 1$
\be
\E\left[\left|{\FF}^{\s}(\b)\right|^q\1_{\left\{\left|{\FF}^{\s}(\b)\right|^q>a\right\}}\right]<\varepsilon.
\Eq(3.lem6.1)
\ee
We can rewrite the left-hand side of \eqv(3.lem6.1) as
\be
\begin{split}
\E\left[\left|{\FF}^{\s}(\b)\right|^q\1_{\left\{\left|{\FF}^{\s}(\b)\right|^q>a\right\}}\right]
=
&
\int_{a^{1/q}}^{\infty}qy^{q-1}\P\left({\FF}^{\s}(\b)>y\right)dy+a\P\left( {\FF}^{\s}(\b)>a^{1/q}\right)
\\
+&
\int_{a^{1/q}}^{\infty}qy^{q-1}\P\left( {\FF}^{\s}(\b)<-y\right)dy+a\P\left( {\FF}^{\s}(\b)<-a^{1/q}\right).
\end{split}
\Eq(3.lem6.2)
\ee
Using successively Chebyshev's exponential inequality, \eqv(2ns) and Fubini's theorem
\be
\P\left( {\FF}^{\s}(\b)>y\right)
\leq 
2\,2^{-ny}\E\left(\int_0^1e^{\b X^{\s}_{h}}dh\right)
=
2\,2^{-ny}\int_0^1\E\left(e^{\b X^{\s}_{h}}\right)dh.
\Eq(3.lem6.3)
\ee
By \eqv(2.2.lem5'.7) with $\l=\b$, $k_1=0$ and $k_2=\infty$ ({\it i.e.}, $\a_1=0$ and $\bar\a_2=1$), and  the expression \eqv(2.2.lem5'.2) of $\varsigma^2_{k_1,k_2}$
\be
\E\left(e^{\b X^{\s}_{h}}\right)
\leq e^{\frac{\b^2}{4}[n\log 2 +\OO(1)]}.
\Eq(3.lem6.4)
\ee
Thus,
\be
\P\left( {\FF}^{\s}(\b)>y\right)\leq c(\b)2^{\frac{n\b^2}{4}-ny},
\Eq(3.lem6.5)
\ee
for all $y\in\R$ and  some constant $c(\b)>0$ that depends only on $\b$. 
Consider now the integral
\be
\II_n(a,q)\equiv\int_{a^{1/q}}^{\infty}qy^{q-1}\P\left({\FF}^{\s}(\b)>y\right)dy.
\Eq(3.lem6.5bis)
\ee
If $q=1$
\be
\II \leq c(\b)\sfrac{1}{n\log 2}2^{\frac{n\b^2}{4}-na}.
\ee
If $q>1$, then
$
y^{q-1}2^{-ny}=\exp\bigl\{-n\ln 2\bigl[y-(q-1)\frac{\ln y}{n\ln 2}\bigr]\bigr\}\leq -n\ln 2\bigl(1-\frac{(q-1)}{n\ln 2}\bigr)y
$
for all 
$
y>0
$. 
Thus,
\be
\II_n(a,q)
\leq 
qc(\b)\left(1-\sfrac{(q-1)}{n\ln 2}\right)^{-1}\sfrac{1}{n\log 2}2^{\frac{n\b^2}{4}-na^{1/q}\left(1-\frac{(q-1)}{n\ln 2}\right)}.
\Eq(3.lem6.5ter)
\ee
Combining our bounds, we get that for all $q\geq 1$
\be
\II_n(a,q)+a\P\left( {\FF}^{\s}(\b)>a^{1/q}\right)
\leq
c(\b)2^{\frac{n\b^2}{4}-na^{1/q}}\left(\sfrac{q(1+o(1))}{n\ln 2}2^{(q-1)a^{1/q}}+a\right),
\Eq(3.lem6.6)
\ee
which can be made as small as desired by taking $a^{1/q}>\frac{\b^2}{4}$ sufficiently large.

The left tail is dealt with in a similar way, using successively Chebyshev's exponential  inequality, \eqv(2ns), Jensen's inequality 
and Fubini's theorem to write
\be
\begin{split}
\P\left( {\FF}^{\s}(\b)<-y\right)
&\leq 
2^{-ny}\E\left[\left(\int_0^1e^{\b X^{\s}_{h}}dh\right)^{-1}\right]     
\\
&\leq
2^{-ny}\E\left[\int_0^1e^{-\b X^{\s}_{h}}dh\right]    
\leq
2^{-ny}\int_0^1\E\left(e^{-\b X^{\s}_{h}}\right)dh,                        	
\end{split}
\Eq(3.lem6.7)
\ee
and proceeding as in \eqv(3.lem6.4) to bound the exponential moments. The second line in \eqv(3.lem6.2) is thus bounded exactly as the first, {\it i.e.}, as in \eqv(3.lem6.6), and so, for each $\s$, the sum of the two can be made as small as desired by taking 
$
a^{1/q}>\max\bigl\{\frac{\b^2}{4}, \frac{(q-1)}{n\ln 2}\bigr\}
$
sufficiently large.
\end{proof}

Theorem \thv(3.prop1) -- and therefore of Theorem \thv(th:zeta) --  is now complete.
\end{proof}


\section{Noninjectivity  of the primitive: proof of Propositions \ref{le:1} and \ref{le:2}}
\TH(S4)
\begin{proof}[Proof of Proposition \ref{le:1}.]

A version of Koebe's distortion theorems (Bieberbach's inequality, see, {\it e.g.}, \cite[Lemma 1.3, p. 21]{P}) provides a necessary condition for univalence by noting that for a univalent function on the unit disc $\D$ the quantity $(1-|z|^2)|f''(z)|/|f'(z)|$ remains bounded on $\D$. A somewhat more detailed version of this  is sometimes called the Becker-Pommerenke injectivity criterion. One may note here  that this quantity remains invariant in the standard normalization $f\mapsto (f-f(0))/f'(0)$ which is used in stating the distortion theorems  for univalent functions on the unit disc. If $U\subset \C$ is a simply connected domain and $f:U\to\C$ is univalent, one may consider the map $f\circ h$, where $h:\D\to U$ is a conformal map. By applying the criterion on $f\circ h$ and returning to original coordinates we see that 
$$
\frac{(1-|h^{-1}(z)|^2)}{(h^{-1})'(z)}\Big(\frac{f''(z)}{f'(z)}+((h^{-1})'(z))^2h''\circ h^{-1}(z)\Big)
$$
stays bounded in $U$. Especially, if $\alpha\in\partial U$ and $\partial U$ is smooth in a neighbourhood of $\alpha$, by localizing the classical Kellog-Warchawski theorems (see \cite[Theorems 10.2 and 10.3, p. 298-301]{P})  on the  smoothness of $h$  we obtain that 
$$
\sup_{z\to\alpha}d(z,\partial U)\frac{|f''(z)|}{|f'(z)|} <\infty.
$$

 In order to prove our claim it thus suffices to show that there exists a decreasing sequence of positive reals $(\sigma_k)$ such that $\sigma_k\searrow 1/2$ and 
$$
\sup_{k\geq 1} (\sigma_k-1/2)|F''(\sigma_k)/F'(\sigma_k)|=\infty\qquad\textrm{almost surely}.
$$
According to Lemma \ref{le:monday} we may write   for $\sigma >1/2$ the decomposition 
$
\log (F'(s))= \sum_{p\in\mathcal P}U_pp^{-ih-\sigma}+ B(s)
$,
where $B(s)$ extends almost surely analytically to a domain containing  $\{\sigma \geq 1/2\}$, and especially $B'(s)$ is almost surely bounded in $\{ s=\s+ih\,:\, 1/2\leq \sigma \leq 1,\; -1\leq h\leq 1\}$. It follows that it is enough to construct  a decreasing sequence of positive reals $\sigma_k\searrow 1/2$ such that 
\begin{equation}\label{eq:1}
\sup_{k\geq 1} (\sigma_k-1/2) \left|\sum_{p\in\mathcal P}U_p\log p\, p^{-\sigma_k}\right|=\infty\qquad\textrm{almost surely}.
\end{equation}
For that purpose  denote 
$$
Y_k:= (\sigma_k-1/2)\sum_{p\in\mathcal P}U_p \log p\, p^{-\sigma_k},
$$
 and notice that $Y_k$ satisfies the asymptotics
$$
\E |Y_k|^2 =(\sigma_k-1/2)^2\sum_{p\in \mathcal P}(\log p)^2 \, p^{-2\sigma_k} \stackrel{\s_k\to 1/2}{\sim} 1/4,
$$
where one uses for instance the identity giving  $I_{m}$ for $m=2$ in the proof of Lemma \ref{rk1}. 
Our goal is to prove that almost surely $\sup_k |Y_k| =\infty$.

We now use induction to pick up a strictly increasing sequence of indices $(n_k)_{k\geq 0}$with $n_0=0$, and a decreasing sequence of $\sigma_k$'s (say with $\sigma_1=1$ and $\sigma_k\leq 1/2+1/k$) in such a way that for all $k\geq 1$ we have 
\begin{equation}\label{eq:3}
(\sigma_k-1/2)^2\sum_{p< n_{k-1}\;\textrm{or}\; p\geq n_{k}}(\log p)^2 p^{-2\sigma_k} \leq 2^{-k}.
\end{equation}
The starting of the induction is obvious as one may just pick $n_1$ large enough. Assume then that $\sigma_k,n_k$ are already chosen. We  pick first $\sigma_{k+1}\leq \min \{\sigma_k, 1/2+1/(k+1)\}$  small  enough so that $(\sigma_{k+1}-1/2)^2\sum_{p\leq n_{k}}(\log p)^2 p^{-1} \leq 2^{-k-2}$, and thereafter choose
$n_{k+1}$ so large that $ (\sigma_{k+1}-1/2)^2\sum_{p\geq n_{k+1}}(\log p)^2 p^{-2\sigma_{k+1}} \leq 2^{-k-2}$.

Consider the random variables
$$
\widetilde Y_k:= (\sigma_k-1/2)\sum_{n_{k-1}\leq p<n_k}U_p \log p \,p^{-\sigma_k},
$$
According to \eqref{eq:3} we have $\E |\widetilde Y_k- Y_k|^2\leq 2^{-k}$, so that $\E \big(\sum_{k\geq 1}|\widetilde Y_k- Y_k|^2\big)<\infty.$This implies that $\lim_{k\to\infty} (\widetilde Y_k- Y_k)=0$ almost surely, and hence it is enough  to show that $\sup_k |\widetilde Y_k| =\infty$. However, by construction, the variables $\widetilde Y_k$ are independent and centered, and their variance is bounded from below and above. The desired statement follows as soon as we verify that the limiting distribution of (normalized) $\widetilde Y_k$ is Gaussian. But this follows immediately from Lindeberg's central limit theorem (see \cite[Theorem 6.12]{Kal}) as each $Y_k$ is a linear combination of  i.i.d. variables, and the maximal coefficient in the linear combination tends to zero as $k\to\infty.$
\end{proof}

\begin{proof}[Proof of Proposition \ref{le:2}.]
The proof is exactly the same as the above proof of Proposition \ref{le:1}, the only difference being that there is no need to use the central limit theorem since our variables are already Gaussian.
\end{proof}


\TH(S5)
\section{Approach via Gaussian multiplicative chaos}

In this section we first recall the definition of Gaussian multiplicative chaos (GMC) measures, and then use some basic properties of these measures (and their approximations) to prove an analogue of our main result for a random analytic function in the unit disc that uses no number theoretic elements in the definition. Finally, we outline how GMC-theory can be used to provide another proof of our main result via a coupling result from \cite{SW}.
\subsection{GMC measures}
The theory of Gaussian multiplicative chaos measures was pioneered by Kahane in 1985. In the last two decades the interest to GMC-measures has been steadily increasing. During this period the  natural role of GMC was first  realized in connection with Liouville quantum gravity and SLE. Soon after it appeared also  in connection with random matrices, statistics of the Riemann zeta function, and estimates for random and Dirichlet character sums  \cite{DS,Sh,Webb,SW,Harper3}.\footnote{We list only some of the earliest references herein, a general reference for multiplicative chaos is \cite{RV3}.}

Formally GMC is constructed on (say) the interval $[0,1]$ as an exponential $``\exp (X(x))"$, where $X$ is 
a log-correlated (and centered) Gaussian field on the interval $x\in [0,1]$, with the formal covariance structure,\footnote{Note that   the factor $1/2$ is not usually included in front of the logarithm, but we add it here  to ensure that no extra scaling is required in the random zeta function model.}
$$
\E X(x)X(y)\; =\; \frac{1}{2}\log \frac{1}{|x-y|}+g(x,y),
$$
where $g$ is continuous.  The field actually takes values in generalized functions, and one usually addresses this difficulty by considering suitable approximations. A typical example is a standard mollification $X_\varepsilon$ obtained by the convolution of  $X$ and $\varepsilon^{-1}\psi(\cdot/\varepsilon)$, where $\psi$ is a compactly supported bump function that integrates to 1.Given a constant $\beta>0$, we may then define the measure 
\begin{equation}\label{eq:epsi}
\mu_{\beta,\varepsilon}(dx):=\exp\big(\beta X_\varepsilon(x)-\frac{\beta^2}{2}\E X_\varepsilon(x)^2\big)dx
\end{equation}
on $(0,1)$. From the basic multiplicative chaos theory  one knows that there is the convergence as $\varepsilon \to 0$,
$$
\mu_{\beta,\varepsilon}\stackrel{\P}\longrightarrow \mu_\beta,
$$
where $\mu_\beta$ is the GMC measure related to the field $X$.  The case $\beta\in (0,\sqrt{2})$ is an easy $L^2$-computation \cite{HK71}; in the more difficult general case, see, {\it e.g.},  \cite{DS,B} for a simple  approach. For other  treatments  of this convergence result  we refer to \cite{K,RV,DS}. The limit $\mu_\beta$ is a non-trivial random measure in case $\beta\in (0,2)$. It is supported on the set of $\beta$-thick points of $(0,1)$, {\it i.e.,} on the set
\begin{equation}\label{eq:thick}
\big\{x\in (0,1)\;|\; \lim_{\varepsilon\to 0}\frac{X_\varepsilon (x)}{\log(1/\varepsilon)}=\beta/2\big\}.
\end{equation}

One may note that the renormalization factor $\exp\big(-\frac{\beta^2}{2}\E X_\varepsilon (x)^2\big)$ in \eqv(eq:epsi) comes simply from the requirement
that the density has expectation 1 at every point $x$. In the approach we describe in this section, it is exactly this factor that is the origin of the quadratic term $\beta^2/4$ which appears in the results we are after, see, {\it e.g.}, \eqref{1.theo1.3bis}. 

When $\beta=2$ (the critical case) or $\beta>2$ (the supercritical case), the situation is different, and the definition of the chaos requires a non-trivial further renormalization. However, in this range we are dealing with the well-understood freezing phenomenon for log-correlated fields \cite{PhysRevLett.108.170601}. This can be understood by the simple (and effective) heuristics that for $\beta >2$ the mass for the chaos measure approximation is obtained from the $(\beta =2)$-thick points and by noting that the maximum of $X_\varepsilon$ is very likely to be of the order $\log(1/\varepsilon)$. (For the relevant upper bound see Lemma \ref{le:max} below.) These heuristics are actually built in the proof in Section \ref{subse:unit disc}.

\subsection{A random analytic function in the unit disc}\label{subse:unit disc}
Our aim is to first consider here a unit disc analogue of the model that was discussed before in the right half-plane, which is related in a similar way to a (almost canonical basic) model of holomorphic chaos. 
Thus, consider the holomorphic function $F$ on $\D:=\{ z\in\C\, :\,|z|<1\}$ defined via $F(0)=0$ and with the derivative
\begin{eqnarray}
F'(z):=\exp(G(z)),\quad 
G(z):=\sum_{n=1}^\infty \frac{G_nz^n}{\sqrt{n}}.
\end{eqnarray}
Here the variables $G_n=(V_n-iW_n)/\sqrt{2}
$ are independent  standard complex Gaussians. 
Note that for $\beta=|\beta|e^{iu}\in \mathbb C$, the complex power $(F'(z))^\beta=\exp\left(\beta G(z)\right)$  has same law as $\exp\left(|\beta| G(z)\right)=(F'(z))^{|\beta|}$, 
because $e^{iu} G(z)$ has same law as $G(z)$,  since each $G_n$ is a circularly-symmetric complex normal. So to study the integral means spectrum we only need to consider the  $\beta >0$ range. 

We thus have 
$$
|F'(z)|=\exp(U(z)),\quad U(z)=\Re\,G(z),
$$
and for $z=re^{i\theta}$ we write,
$$
U_r(\theta):=U(re^{i\theta})=\Re\, G(re^{i\theta})=\sum_{n=1}^\infty\frac{r^n}{\sqrt{2n}}\big(V_n\cos(n\theta)+W_n\sin(n\theta)\big).
$$
Thus $U$ is a random harmonic function on $\D$, and its restriction $U_r$ on the circle of radius $r$ is the natural aproximation of the  log-correlated field
$$
X(\theta):= \sum_{n=1}^\infty\frac{1}{\sqrt{2n}}\big(V_n\cos(n\theta)+W_n\sin(n\theta)\big).
$$

The covariance structure of these approximations is very easy to compute, and we obtain for any $z=re^{i\theta}, z'=r'e^{i\theta'}\in \D$
                  \begin{eqnarray}\label{eq:covs}
\E\, U(z)U(z')&=&\frac{1}{2}\sum_{n=1}^\infty \frac{r^nr'^n}{n}(\cos(n\theta)\cos(n\theta') + \sin(n\theta)\sin(n\theta'))=\frac{1}{2}\Re\left(\sum_{n=1}^\infty \frac{(z\overline{z'})^n}{n}\right) \nonumber\\
&=& \frac{1}{2}\log\frac{1}{|1-z\overline{z'}|}.
\end{eqnarray}

\newcommand\torus{{\mathbb T}}

In what follows we identify $\torus:=\partial\D=[0,2\pi)$ in a standard manner. The basic theory of GMC measures now implies that in the so-called subcritical range $\beta\in (0,2)$ there is the convergence in probability for the chaos approximations obtained from $U_r$:s,
$$
\mu_{\beta,r} (\torus):=\frac{1}{2\pi}\int_0^{2\pi}\exp \left(\beta U_r(\theta)-(\beta^2/2)\E\, U_r^2(\theta)\right)d\theta\stackrel{{\Prob}}{\longrightarrow} \mu_\beta (\torus),\quad\textrm{as}\;\;r\to 1^-.
$$
Recall that the multiplicative chaos  $\mu_\beta$ is a (random) Borel measure on the unit circle $\torus$ such that $\mu_\beta(\torus)\not=0$ almost surely.  As mentioned before, the construction of the chaos measure and the convergence in probability for convolution approximations as we need here can be found in \cite[Section 3-4, pp.~5-10]{B}. The assumptions in \cite{B} actually deal with compactly supported convolution kernels in $\R^d$, but the only properties the proof in \cite{B} uses for the field and the approximation are the following: 

\begin{itemize}
\item[\it 1.] First of all, for   $r,r'\in(0,1)$ and $\theta,\theta'\in [0,2\pi)$
$$
\E\, U(re^{i\theta})U(r'e^{i\theta'})=\frac{1}{2}\log\frac{1}{|\theta-\theta'|\vee (1-r)\vee(1-r')} + \mathcal O(1),
$$
where the distance $|\theta-\theta'|$ is understood mod $2\pi.$

\item[\it 2.] Secondly, for any fixed $\delta>0$ there is the uniform convergence  in the set $\{ \theta,\theta'\in [0,2\pi),\; |\theta-\theta'|\geq\delta\}$
$$
\E \, U(re^{i\theta})U(r'e^{i\theta'})\to a (\theta,\theta'),\quad\textrm{as}\;\;r,r'\to 1^-,
$$
where $a$ is continuous on $[0,2\pi)^2\setminus\{\theta=\theta'\}.$

\end{itemize}

Both of these conditions follow immediately from formula \eqref{eq:covs}. For our purposes, we further need some basic facts for the moments of the chaos approximations:
\begin{lemma}\label{le:gmc}

The approximations $\mu_{r,\beta} (\torus)$ satisfy
\begin{eqnarray}\label{eq:moments}
\E \, \mu_{\beta,r} (\torus)=1\qquad\textrm{and}\quad \E\, (\mu_{\beta,r} (\torus))^{-1}\leq C,
\end{eqnarray}
where the constant $C=C(\beta)$ is independent of $r\in (0,1).$ Especially, we have $\mu_\beta (\torus)\not=0$ almost surely.
\end{lemma}
\begin{proof}
 That $\E\, \mu_{\beta,r} (\torus)=1$ follows immediately from the definition, and the uniform finiteness of the negative moments is well known. To indicate the proof, we note that,  {\it e.g.}, in \cite[Appendix B]{AJKS} it is proven for a particular log-correlated field, and it then immediately follows by  easy comparision by using Kahane's convexity inequality both for $\mu_{\beta} (\torus)$ and its approximations $\mu_{\beta,r} (\torus)$. Here one recalls that the function $x\to 1/x$ is convex, and one may additionally make use of the trick of adding an independent constant Gaussian field to the reference field if needed in order to facilitate the comparison.
\end{proof}
\begin{corollary}\label{cor:helppo}
Denote $r_n=1-e^{-n}$ for integers $n\geq 1$. For any $\beta\in (0,2)$ and any $\varepsilon >0$ there is a random index $n_0$, finite almost surely, so that
$$
\exp\big((\beta^2/4 -\varepsilon)n\big) \leq \frac{1}{2\pi} \int_0^{2\pi}\exp \big(\beta U_{r_n}(\theta)\big) d\theta \leq \exp\big((\beta^2/4 +\varepsilon)n\big)\qquad\textrm{for}\quad n\geq n_0.
$$
\end{corollary}
\begin{proof}
We have from \ref{eq:covs}
$$\E\, U^2_{r_n}(\theta)=\frac{1}{2}(n- \log 2)+\mathcal O(e^{-n}).$$
By the first moment estimate in \eqref{eq:moments} and Markov's inequality we therefore have 
\begin{eqnarray*}
\Prob\left(\frac{1}{2\pi} \int_0^{2\pi}\exp \big(\beta U_r(\theta)\big)d\theta \geq \exp\big((\beta^2/4 +\varepsilon)n\big)\right)&&= 
\Prob\left(C_n(\beta) \mu_{\beta,r} (\torus) \geq \exp(\varepsilon n)\right)\\
&&\leq C_n(\beta)  e^{-n\varepsilon},
\end{eqnarray*}
where $C_n(\beta)=2^{-\beta^2/4}+ \mathcal O(e^{-n})$. One finishes by an application of Borel-Cantelli.
An analogous lower estimate follows by invoking the negative moment bound  in \eqref{eq:moments}. \end{proof}

Another basic result for approximations of log-correlated fields we need is an easy result from the (extensively studied) behaviour of the maxima of a log-correlated field. For us the following crude result is enough.
\begin{lemma}\label{le:max}
Assume that $\varepsilon >0$. There is a random $r_0$ such that $r_0<1$ almost surely, and for all $r\in [r_0,1)$ and $\theta\in [0,2\pi)$ it holds that
$$
U(re^{i\theta})\leq (1+\varepsilon)\log(1/(1-r)).
$$
\end{lemma}
\begin{proof} This is well-known and follows easily, {\it e.g.}, from \cite[Theorem 1.1]{A}  by applying the estimate there [with in our case the identification $\varepsilon\equiv 1-r$, and $m_{1-r}=\log (1/(1-r))-3/4 \log \log (1/(1-r))$], to the sequence of radii $r_n=1-e^{- n}$ to obtain
$$
\P\left(\max_{\theta\in [0,2\pi]}U(r_ne^{i\theta}) >\log (1/(1-r_n))+\lambda_n^2\right)\leq \exp(-c \lambda_n^2),
$$
with the choice $\lambda_n=\log \log (1/(1-r_n))=\log n$. Since $\sum_{n=1}^\infty \exp(-c\lambda_n^2) <\infty,$ we obtain the claim by using the Borel-Cantelli lemma and the fact that the maximum of the harmonic function $U(re^{i\theta})$ in the annulus $r_n \leq r <r_{n+1}$ is obtained at the boundary.
\end{proof}
We are now ready to prove Theorem \ref{th:prob}.
\begin{proof}
As seen above, by circular invariance we need only to study the $\beta>0$ range. Let us consider, for $r\in (0,1)$, the function $\phi_r: (0,\infty)\to(0,\infty),$ where
$$
\phi_r(\beta)=  \frac{1}{\log \big(1/(1-r)\big)}\log\Big( \int_0^{2\pi}\big| F'(re^{i\theta})\big|^\beta d\theta\Big).
$$
We first consider a fixed $\beta\in (0,2)$. Corollary \ref{cor:helppo} implies that along the sequence $r_n$ of radii, $\phi_{r_n}(\beta)$ has almost surely the stated limit $\beta^2/4$, up to an additive term $\pm\varepsilon$. Since this holds for all $\varepsilon >0$, we obtain the desired limit statement among this sequence of radii. Finally, we observe that the function $r\mapsto   \int_0^{2\pi}\big| F'(re^{i\theta})\big|^\beta d\theta$ is increasing in $r$ (see \cite[p. 38 or Lemma 1.1]{G}), which immediately implies the statement for all radii.

By the log-convexity of the H\"older norm, for each $r\in (0,1)$ our $\phi_{r}$ is a convex function of $\beta\in (0,\infty)$. According to Corollary \ref{cor:helppo},  we know that it almost surely converges to the function $\beta^2/4$ for all $\beta\in (0,2)\cap\mathbb Q$ as $r\to 1$ via the sequence $(r_n)_n$. This then implies convergence through all $r\nearrow 1$ again by the previously mentioned increasing nature of the $\beta$-means. On the other hand, we have by Lemma \ref{le:max} that $\limsup \phi_r(\beta)\leq \beta -1$ for all $\beta>0$. These two facts and a standard convexity argument then imply our theorem for all $\beta >0$ (note that the derivative of the map $\beta\to \beta^2/4$ equals 1 at $\beta=2$).
\end{proof}

\noindent {\bf Remarks.}\begin{itemize}
\item
Actually the only non-trivial ingredient  we use from GMC theory above is  the uniform bound for negative moments.\
 \item Weaker  results in the case $\beta\geq 2$ naturally could be deduced also from the theory of critical chaos and supercritical chaos, but this would be  technically much more complicated.
 
\end{itemize}

\subsection{Alternative treatment of the randomized Riemann zeta model via GMC}
Here we sketch without all details how one may also use the approach of Section \ref{subse:unit disc} to treat the random analytic function $\zeta_\rand$ we considered in Theorem \ref{th:zeta} in the first part of the present paper.
According to \cite[Theorem 1.7]{SW}, we may write the logarithm of the randomized zeta function as
\be
\label{Es}
\log (\zeta_\rand(s))= G_0(s)+E(s),
\ee
where $E(s)$ is almost surely bounded in any compact subset of $\{s=\s+it, \sigma\geq 1/2, t\in \R\}$ and $G_0$ is a  \emph{Gaussian} analogue of the part $ \sum_{p \in \mathcal P}U_pp^{-s}$ of the randomized zeta function:
$$
G_0(s) =\sum_{p\in \mathcal P} p^{-s}W_p,
$$
where the $W_p$:s are complex i.i.d. Gaussians. Naturally $G_0$ is almost surely analytic in $\{\sigma>1/2\}$, and on the boundary $\{\sigma=1/2\}$ its real part (as its  imaginary part) defines a log-correlated Gaussian field. We are thus interested in the real part 
$$
G(s) :={\Re}G_0(s)={\Re}\Big(\sum_{p\in \mathcal P} p^{-s}W_p\Big),
$$
and on  the integral means of $\exp (G(s))$. The real part is the harmonic extension to  $\{\sigma>1/2\}$ of the log-correlated field $G|_{\{\sigma=1/2\}}$, and it has the covariance  
$$
C(s,s'):=\E\, \Re G_0(s)\Re G_0(s')= \frac{1}{2}{\Re}\Big(\sum_{p\in \mathcal P}p^{-s-\overline {s'}}\Big),\qquad \sigma,\sigma'>1/2.
$$
We collect all the properties of the covariance we need in the following statement, where $\zeta$ denotes the \emph{classical} zeta function.
\begin{lemma}\label{le:cov} {\rm (i)}\quad We have
\begin{eqnarray*}
C(s,s')&=&\frac{1}{2}{\Re}\big(\log(\zeta(s+\overline {s'}))\big) + b_0(s,s')\\
&=& \frac{1}{2} \log\frac{1}{|s+\overline {s'}-1|} +b(s,s')
\end{eqnarray*}
where $b_0,b\in C^\infty(\{\sigma,\sigma'\in [1/2,1],\;\; t,t'\in [-1,1]\})$.

\smallskip

\noindent {\rm (ii)}\quad We have for all $\sigma, \sigma' >1/2$ and $t,t'\in [-1,1]$
$$
C(s,s')=\frac{1}{2}\log\frac{1}{|t-t'|\vee (\sigma-1/2)\vee(\sigma'-1/2)} + \mathcal O(1).
$$
\smallskip

\noindent {\rm (iii)} For any fixed $\delta>0$ there is uniform convergence  in the set $\{\sigma,\sigma'\in (1/2,1],\;\; t,t'\in [-1,1],\; |t-t'|\geq\delta\}$
$$
C(s,s')\to a (t,t'),\quad\textrm{as}\;\;\sigma,\sigma'\to (1/2)^+,
$$
where $a$ is continuous on $[-1,1]^2\setminus\{t=t'\}.$
\end{lemma}
\begin{proof}
We note towards (i) that 
\begin{eqnarray*}
C(s,s') 
= \frac{1}{2}{\Re}\big(\log(\zeta(s+\overline {s'}))\big)- \frac{1}{2}{\Re}\big(\widetilde{B}(s+\overline{s})\big),
\end{eqnarray*}
where the Dirichlet series $\widetilde{B}(s):=\sum_{k\geq 2}\frac{1}{k}\sum_{p\in \mathcal P}p^{-ks}$ converges at any $\s\in (1/2,\infty)$, and hence defines an analytic function on $\{\sigma >1/2\}$ by Jensen's lemma. This proves the first statement in (i), and the second one follows by observing that $(s-1)\zeta(s)$ is analytic and non-zero in a neighborhood of the rectangle $ [1,2]\times [-2,2]$.

In order to check (ii) we need to consider only the logarithmic term and we assume first that $\sigma>\sigma'$. Then 
$$
\big| \log |\sigma+\sigma'-1+i(t-t')|-\log |2\sigma-1+i(t-t')|\big| \leq \log \left(\frac{\sigma+\sigma'-1}{2\sigma-1}\right)\leq \log 2.
$$
Finally, it is obvious that  $|\log |2\sigma-1+i(t-t')|-\log\big((\sigma-1/2)\vee |t-t'|\big)|\leq {3\log 2},$ and (ii) follows. In turn, (iii) is a direct consequence of part (i).
\end{proof} 

We are thus now in position to define the renormalized GMC measure for $\b\in (0,2)$ and $\s>1/2$,
\be\label{GMC}
\mu_{\beta,\s} (I):=\int_I\exp \left(\beta \Re G_0(s)-(\beta^2/2)C(s,s)\right)|ds|,
\ee
for any vertical line segment $I=I(h,h')=\{s=\s+it, t\in [h,h']\}$, and where, owing to (ii) in Lemma  \ref{le:cov},
\be\label{Css}
C(s,s)=C(\s,\s)=\log(1/(\s-1/2))+\mathcal O(1)\;\; \textrm{as}\;\; \s\to (1/2)^+.
\ee
Because of Lemma \ref{le:cov} and \cite{B}, we have the convergence in probability,
$$
\mu_{\beta,\s} (I)\stackrel{{\Prob}}{\longrightarrow} \mu_\beta (I),\;\;\textrm{as}\;\;\s\to (1/2)^+,
$$
where $ \mu_\beta(I)\neq 0$  almost surely.

After this we are essentially back in a situation where we need to consider the divergence of the integral means defined via Gaussian log-correlated fields:  indeed, from this point onwards the proof is almost completely analogous to what we did in Section \ref{subse:unit disc}. There is only one item of the argument that will not carry over immediately: the integral means over the segment $\{\sigma\}\times [0,1]$ need not be strictly increasing in decreasing $\sigma>1/2$. Hence we need to do some `juggling' to overcome this phenomenon. In order to  tackle this technical detail we first write down a somewhat crude estimate for the growth of $\zeta_\rand$ near the boundary:

\begin{lemma}\label{le:kasvu} For $\varepsilon >0$, there is $r>0$ and a random constant $C>0$ such that almost surely
$$
|\Re G_0(s)|\leq C+ (1+\varepsilon) \log(1/(\sigma-1/2))+r\log(1+|t|)\quad \textrm{for all}\quad t\in\R \;\; \textrm{and}\;\; \sigma \in (1/2,1].
$$
\end{lemma}
\begin{proof} Let us  set  $Q_{n,j}=(1/2+2^{-n-1},1/2+2^{-n}]\times (j, j+1]$ for $n\geq 1$ and $j\in\Z$. We may apply  (say)
\cite[Theorem 1.1]{A} separately on  each vertical boundary segment. On a horizontal boundary segment $[1/2+2^{-n-1}+ij,1/2+2^{-n}+ij]$ we may consider for $x\in[0,1]$ the random field $x\mapsto  Y_j(x):=\Re G_0(1/2+2^{-n-1}(1+x)+ij)- \Re G_0(1/2+2^{-n-1}+ij)$. By \ref{le:cov} (i), the Lipschitz constant of the covariance of $Y_j$ is uniformly (in $j$) bounded on $[0,1]$, and since $Y_j(0)=0$, standard results (see, {\it e.g.}, \cite[Theorem 4.1.2 with $\alpha=1$]{AT}) yield  that $\sup_{x\in [0,1]} Y_j(x)$ has a  uniform Gaussian tail. We also know exactly the Gaussian tail of $ \Re G_0(1/2+2^{-n-1}+ij)$. Putting all together, we deduce that
$$
\Prob (\Re G_0(s)\geq \log(1/(\sigma-1/2)) +\lambda \;\;\textrm{for some}\;\; s\in \overline{Q}_{n,j})\leq C'e^{-c\lambda},
$$
with $c,C'$ two finite constants. By choosing $\lambda =\lambda_{j,n}: =(2/c)\log (|j|+1)+(\log n)^2$, the Borel Cantelli lemma implies that almost surely $|\Re G_0(s)|\leq \log(1/(\sigma-1/2)) +\lambda_{j,n}$ in all rectangles $ \overline{Q}_{n,j}$, apart from finitely many which induce a finite random bound $C$. The Lemma follows, with $r = 2/c$.
\end{proof}

Let $H^p(\sigma>a)$ denote  the space of analytic functions on the half-plane $\{\sigma>a\}$ such that $\sup_{\sigma>a}\int_{-\infty}^\infty |f(\sigma+it)|^p dt <\infty$. 
Let us define the auxiliary random function  \begin{equation}\label{eq:kasvu}
\widetilde \zeta_\rand(s):=\exp (-\sqrt{s})\exp (G_0(s)).
  \end{equation}
The above lemma implies  that 
  \begin{equation*}
\widetilde \zeta_\rand(s)\in H^p(\sigma>a)\cap H^\infty (\sigma>a)\; \textrm{ for all}\;\; p >0,\;a>1/2,  
  \end{equation*}
together with
    \begin{equation}\label{eq:kasvu2}
\log |\widetilde \zeta_\rand(s)|\leq (1+\varepsilon)\log (1/(\sigma-1/2))\;\; \textrm{on}\;\; \{\sigma>1/2, |t|\leq 1\}.  
  \end{equation}
  We note that since $|\exp (E(s))|$ in \eqref{Es} is almost surely bounded from below and above on $(\sigma,t)\in (1/2,2)\times (-2,2)$, as well as $|\exp (-\sqrt{s})|$, the random functions  $\widetilde \zeta_\rand$ and   $\zeta_\rand$ yield the same integral means spectra.

We will make use of a $H^p$-norm analogue of  Hadamard's three lines theorem: 
\begin{lemma}\label{le:Zygmund} Let $S$ be an infinite strip in the complex plane with boundary lines $L_1, L_2$. Let $p>0$ and assume that $f$ is a holomorphic function on $S$ with $\sup_L\int_L|f(z)|^p|dz|<\infty$, where the supremum is over all lines $L\subset \overline{S}$  $($which naturally are parallel to  $L_1$ and $L_2$$)$. Then for any line $L\in S$ it holds that
$$
\int_L|f(z)|^p|dz|\;\leq \; \Big(\int_{L_1}|f(z)|^p|dz|\Big)^{1-\theta}\Big(\int_{L_2}|f(z)|^p|dz|\Big)^{\theta},
$$
where $\theta:={\rm dist\,}(L,L_1)/{\rm dist\,}(L_1,L_2)$.
\end{lemma}
\begin{proof}
See \cite[Theorem 2.3]{BakanKaijser07} or \cite[Chapter XII, 1.3]{Zygmund}.
\end{proof} 
To apply Lemma \ref{le:Zygmund} to finite segments instead of lines, we use the following auxiliary function: 
\begin{lemma}[Localization]\label{le:auxf}   Denote $z=x+iy\in\C$.
For each $M\geq 2$ there is a  function $g_M$ that is analytic in the strip $S=\{ -1/2<x< 1/2\}$, and satisfies
\begin{eqnarray*}
|g_M(z)|&\leq& 1\;\;\textrm{in}\;\;S,\\
|g_M(z)|&\geq& e ^{-2}\textrm\;\;{for}\;\;  \{|y|\leq M\}\cap S,\quad \textrm{and}
\\
 |g_M(z)|&\leq&\exp(e ^{-M}/2)\;\;\textrm{for}\;\; \{|y|\geq  2M\}\cap S.
\end{eqnarray*}
\end{lemma}
\begin{proof}
Simply choose $$g_M(z)=\exp\big(-e^{-M-iz}-e^{-M+iz}\big),$$ for which 
$|g_M(z)|=\exp\big(-e^{-M}2\cos x\cosh y\big)$.
\end{proof}

We are now ready to complete the needed estimates for the integral means. Consider for any $a\geq 0$ the vertical line segment
$I_a=\{\sigma=1/2+2^{-a}\;, \; 0<t<1\}$,  its double 
$I'_a=\{\sigma=1/2+2^{-a}\;, \; -1/2<t<3/2\}$, and its half  $I''_a=\{\sigma=1/2+2^{-a}\;, \; 1/4<t<3/4\}$. Denote also the line  $L_a=\{\sigma=1/2+2^{-a}\}$. 

Exactly as in Section   \ref{subse:unit disc}, Lemma \ref{le:gmc} and Corollary \ref{cor:helppo}, we may use the positive and negative moments of the Gaussian chaos measure \eqref{GMC} \eqref{Css} corresponding to the  boundary values of the Gaussian field $\Re G_0$, and analogous to \eqref{eq:moments},  to show that  there exists a random index $n_0$, finite almost surely, so that for any $\beta\in(0,2)$ and $\varepsilon >0$ it holds that
\begin{equation}\label{eq:bounds}
2^{(\beta^2/4 -\varepsilon)n} \leq \int_{I''_n}|\widetilde\zeta_\rand(s)|^\beta|ds| \leq \int_{I'_n}|\widetilde\zeta_\rand(s)|^\beta|ds| \leq 2^{(\beta^2/4 +\varepsilon)n}\;\;\textrm{for}\;\; n\geq n_0.
\end{equation}

We now consider any interval $I_a$, with $a=n+\vartheta, \vartheta \in [0,1)$, and apply the convexity inequality of Lemma \ref{le:Zygmund}  to line $L_a$  in the strip $S_n$ that lies between the lines $L_n$ and $L_{n+1}$, with $p=\b$ and $\theta=2^{1-\vartheta}-1$. To this end, we use the function 
$
f(z)=g_{2^{n}}(z)\widetilde \zeta_\rand(s),
$ 
with $z=2^{n+1}(s-(1/2+3\cdot 2^{-n-{2}}+i/2))$. By the properties of $g_{2^{n}}$ we have that $|f(z)|\geq e^{-2}|\widetilde \zeta_\rand(s)|$ on $S_n\cap\{ 0<t<1\}$ together with $|f(z)|\leq \exp(-e^{2^n}/2)|\widetilde\zeta_\rand(s)|$ on $S_n\cap \{t<-1/2\;\textrm{or}\;\,t>3/2\}.$ Using our upper bounds \eqref{eq:bounds} for $\widetilde \zeta_\rand(s)$ it  follows that,
\begin{eqnarray*}
 e^{-{2}\beta}\int_{I_a}|\widetilde\zeta_\rand(s)|^\beta|ds|
 &\leq &\big( \int_{I'_n}|\widetilde\zeta_\rand(s)|^\beta|ds| + o(1)\big)^{1-\theta}
\big( \int_{I'_{n+1}}|\widetilde\zeta_\rand(s)|^\beta|ds| + o(1)\big)^{\theta}\\ & \leq&  2^{(\beta^2/4 +\varepsilon)(n+\theta)}(1+o(1)),
\end{eqnarray*}
where the $o(1)$ terms obviously come from integrals over complementary intervals $L_n\setminus I'_n$ and $L_{n+1}\setminus I'_{n+1}$. The inequality to the other direction is proven analogously by applying Lemma \ref{le:Zygmund} to estimate the obvious analogue of $f$ over the line {$L_n$ in the strip between $L_a$ and $L_{n-1}$. This time we obtain, for $\theta=(2^{\vartheta}-1)/(2^{\vartheta+1}-1)$,
\begin{eqnarray*}
 e^{-2\beta}\int_{I''_n}|\widetilde\zeta_\rand(s)|^\beta|ds|
 &\leq &\big( \int_{I_a}|\widetilde\zeta_\rand(s)|^\beta|ds| + o(1)\big)^{1-\theta}\\
 &\times&\big( \int_{I_{n-1}}|\widetilde\zeta_\rand(s)|^\beta|ds| + o(1)\big)^{\theta}.
\end{eqnarray*}
This clearly yields a suitable lower bound for $\int_{I_a}|\widetilde\zeta_\rand(s)|^\beta|ds|$
 as soon as we substitute in the above inequalities the lower bound for  $ \int_{I''_n}|\widetilde\zeta_\rand(s)|^\beta|ds|$ and the upper bound for $\int_{I_{n-1}}|\widetilde\zeta_\rand(s)|^\beta|ds|$ obtained from \eqref{eq:bounds}. This completes the GMC-based approach to Theorem \ref{th:zeta} for $\beta\in(0,2)$.

For the $\b \geq 2$ case of Theorem \ref{th:zeta}, we use the same convexity argument as at the end of the proof of Theorem \ref{th:prob} in Section \ref{subse:unit disc}. Owing to the upper bound \eqref{eq:kasvu2}, we indeed have for any $\beta>0$, that $\limsup \phi_\s(\beta)\leq \beta -1$, here for the function $\phi_\s(\b):=\log \int_{I_a}|\widetilde\zeta_\rand(s)|^\beta|ds|/\log((\s-1/2)^{-1})$. Its limit as $\s\to (1/2)^+$  yields the convex integral means spectrum $f(\b)=\inf\{\beta^2/4, \b-1\}$ sought for.


\section{A brief history of multifractal analysis for harmonic measure}\label{App}
{\sl Caveat: In this Appendix, we return to notations for the integral means spectrum that are classical in harmonic measure studies \cite{MR1629379,MR2450237}, but differ from those of the previous sections.} 

Let $h$ be a conformal mapping from the unit disc, $h: \D\to \Omega$, where $\Omega$ is a bounded planar domain.  Let $t$ be a complex number, and consider the integral means of the growth of the modulus of the $t$th power of the derivative, $|h'(z)^t|$. The \emph{integral means spectrum} associated with $h$ is then defined as \cite{MR1629379,Bin97}
\be \Eq(-1.0A)
\b_h(t) := \underset{r\to 1^{-}}{\limsup} 
\frac{\log \int_{r\partial \D} |h'(z)^t| |dz|}{\log\left(\frac{1}{1-r}\right)}.
\ee
When the limit exists, one has the asymptotic behavior,
\be \Eq(-1.1A)
\int_{r\partial \D} |h'(z)^t| |dz|\asymp \left(1-r\right)^{-\b_h(t)}, \quad r\to 1^{-},
\ee
in the sense of the equivalence of the logarithms.

When the conformal mapping $h$ is \emph{random}, one is usually lead to first define an \emph{average} integral means spectrum, where one takes the expectation of the l.h.s. of \eqv(-1.1A), and the question naturally arises of the comparison of the average spectrum and of the almost sure one. For instance, let us recall the \emph{average} integral means 
spectrum of the  Schramm-Loewner evolution \cite{OS}, given, for  the bounded version of whole-plane SLE$_\kappa$, by  the convex function for real $t$ \cite{Duplantier00,BS,MR3638311}
\bea \Eq(-1.2A)
\b(t,\kappa)-t+1&=& 1+2\tau -2\sqrt{b\tau}, \quad \tau=d-t,
\\ \label{bkA}
d=d(\kappa) &=&\frac{(4+\kappa)^2}{8\kappa},\quad \quad \,\,\,\quad t\in [t_1,t_2],
\eea
 with $t_1=-1-\frac{3}{8}\kappa, t_2=\frac{3}{4}d(\kappa)$. Outside of that interval, an average spectrum associated with the SLE \emph{tip}  exists  for $t\leq t_1$, 
whereas the  average spectrum becomes \emph{linear}  for $t\geq t_2$ \cite{Duplantier00,PhysRevLett.88.055506,BS,MR3638311}. The  \emph{almost sure tip spectrum} was obtained in \cite{zbMATH06126650}, as well as an \emph{a.s. boundary} spectrum \cite{zbMATH06571699}, extended in \cite{zbMATH07250114}. The \emph{a.s. SLE bulk spectrum}, finally obtained in \cite{gwynne2018} via the so-called \emph{imaginary geometry}, is  identical to the average one   \eqv(-1.2A), except that its transition to a linear spectrum happens before $t_2$, exactly at the point where the intersection of the tangent with the vertical axis leaves the $[0,-1]$ interval, {\it i.e.}, Makarov's criterion for a $\b$-spectrum to be that of an actual conformal map \cite{MR1629379}.

 There exist equivalent formulations in terms of the \emph{multifractal spectrum} of measures, initiated by B.B. Mandelbrot in 1974 \cite{zbMATH03454953}, and generalized in 1985 by Frisch and Parisi \cite{FP} and Halsey  {\it et al.}
  \cite{1986PhRvA..33.1141H,PhysRevA.34.1601}, when trying to describe 
 scaling laws of measures on various physical fractals (see also \cite{zbMATH03856061}). If $\o$ is a (locally finite) Borel measure, its $f(\a)$ spectrum is defined by
 \be \Eq(-1.3)
 f(\a)=\dim (E_\a),
 \ee
where $\dim$ is here thought of as being the Hausdorff  (or sometimes Minkowski) dimension, and where $E_\a$ is the set of points with a specific scaling power $\a$,
\be
E_\a=\{z \,| \,\o\big(\D (z,r)\big)\overset{r\to 0}{\asymp} r^\a\},
\ee
 where $\mathbb D(z,r)$ is the disc of radius $r$ centered at $z$. 

When $\o$ is the \emph{harmonic measure} on a bounded, simply connected domain, Beurling's theorem states that $$\o\big({\mathbb D}(z,r)\big) \leq c\, r^{1/2},$$
where the constant $c$ depends only on the conformal radius at point $z$, which implies that $f(\a)$ is defined only for $\a\geq 1/2$.  
In the harmonic measure case, the integral means spectrum $\b(t)$ \eqv(-1.1A) and the multifractal spectrum $f(\alpha)$ \eqv(-1.3) are, for (in some sense) regular fractals, related by the 
following Legendre-type transform (for general domains, there are only one-sided inequalities, see Makarov's survey \cite{MR1629379})
\bea \Eq(-1.5)
\frac{1}{\a}f(\a)&=&\underset{t}{\inf}\left\{\b(t)-t+1+\frac{1}{\a} t \right\},\\ \Eq(-1.6)
\b(t)-t+1&=&\underset{\a}{\sup}\left\{\frac{1}{\a}\big(f(\a)-t\big)\right\}. 
\eea
Probably the most famous result in the field is due to Makarov \cite{Makadist}, and concerns the value $\a =1$, which corresponds to $t=0$: For any simply connected domain, the harmonic measure is supported on a set of exactly Hausdorff dimension one, which translates into the universal result 
\be \label{Mak1}
f(1)=1.
\ee
This in turn implies that the $\b$-spectrum is quadratic near the origin $t=0$.
 In the SLE case, when applying \eqv(-1.5) to \eqv(-1.2A), one readily finds the average multifractal spectrum of its harmonic measure \cite{Duplantier00,BS,MR3638311}
\be\Eq(-1.4)
f(\alpha,\kappa)=\alpha-d(\kappa)\frac{(\a-1)^2}{2\a-1}
 \quad \a \geq \frac{1}{2}.
\ee
This \emph{expected} spectrum is non-negative only in the interval $[\a_{-},\a_{+}]$, where $\a_{\pm}=(2d-1\pm\sqrt{1+4d})/2(d-2)$, whereas the \emph{almost sure} spectrum is identical to \eqv(-1.4) in this interval and vanishes outside of it \cite{gwynne2018}. As expected, one finds that $f(\a=1,\kappa)=1$ for all $\kappa>0$, in agreement with Makarov's theorem, but notice that it is also the case for the derivative, namely that $\partial_\a f(\a=1,\kappa)=1$ for all $\kappa>0$. 

   The so-called universal spectra are obtained when one considers the supremum of all conformal maps over bounded (vs. unbounded), simply connected planar domains. In particular, the so-called \emph{pressure spectrum} 
   $\mathrm B(t)$ is defined as the supremum over univalent holomorphic functions $h$ from $\mathbb D$ to \emph{bounded} domains $\Omega$ as 
 \be\nonumber
\mathrm{B}(t)=\underset{\Omega}{\sup} \,\b(t),\quad \Omega\, \,\textrm{bounded},
   \ee
whereas a universal spectrum $\mathrm B_{\bullet}(t)$ is similarly defined  for \emph{unbounded} domains.

A number of  exact results and outstanding conjectures are known for the universal integral means spectra.  
(See, {\it e.g.}, the detailed survey by Hedenmalm and Sola \cite{MR2419488}, and the treatise by Garnett and Marshall \cite{MR2450237}.) For real parameter $t$, a well-known result from Makarov \cite{MR1629379} is the following relation between the bounded vs. unbounded spectra,
\be \label{MakbuA}
\mathrm B_{\bullet}(t)=\max\{\mathrm B(t),3t-1\},\quad t\in \mathbb R.
\ee
while it was already known that $\mathrm B_{\bullet}(t)=3t-1$ for $t\geq 2/5$ \cite{FengMcG}. For large values of $t$, the universal integral means spectrum $\mathrm B$ is also linear and a classical result  \cite{MR1217706} is
\be \label{t-1bis}
\mathrm B(t)=t-1,\quad t\geq 2,
\ee
whereas Jones and Makarov \cite{MR1356778},
    proved that near $t=2$,  one has 
\be
\mathrm B(2-t)=1-t+\mathcal O(t^2), \quad t\to 0,
\ee
thus insuring the existence of the derivative of $\mathrm B(t)$ at the ``phase transition'' point $t=2$.  

Near the origin, Makarov's theorem implies that there exists a constant $c$ such that $B(t)\sim c t^2$ for $t\to 0$, and it is known \cite{MR2299502} that $c \leq(\sqrt{24}-{3})/{5}=0.3798...$, a result  improving on \cite{MR2130416}.  
One also has the lower bound, $B(t) > t^2/5$ for $0<t\leq 2/5$ \cite{MR2237215}. 

In the $f(\a)$ formalism, the  universal multifractal function is defined as the supremum $\mathrm F(\a)=\underset{\Omega}{\sup}\, f(\alpha)$, where the bounded and unbounded cases are this time equivalent. It is naturally related to the universal integral mean spectrum above by the Legendre-type transform \cite{MR1629379},
\bea \Eq(-1.5bis)
\frac{1}{\a}\mathrm F(\a)&=&\underset{0\leq t\leq 2}{\inf}\left\{\mathrm B(t)-t+1+\frac{1}{\a} t \right\},\quad  \a \geq 1,\\ \Eq(-1.6bis)
\mathrm B(t)-t+1&=&\underset{\a\geq 1}{\sup}\left\{\frac{1}{\a}\big(\mathrm F(\a)-t\big)\right\}, \quad \quad \, 0\leq t\leq 2. 
\eea

\subsection{Brennan-Carleson-Jones-Kraetzer Conjectures} 
For large negative values of $t$, Carleson and Makarov \cite{CaMa94} proved that there exists a constant $t_0<0$ such that $\mathrm B_{\bullet}(t)=\mathrm B(t)=-t-1$ for $-\infty <t\leq t_0$. 
Determining the optimal value of $t_0$ is an important question in the theory of conformal mappings. While it is only known that $t_0\leq -2$, the value $t_0=-2$ would  yield $\mathrm B_{\bullet}(-2)=\mathrm B(-2)=1$, as well as  $\mathrm F\left(\a=\frac{1}2{}\right)=0$, as suggested by Beurling's inequality, which corresponds to the well-known Brennan conjecture \cite{MR509942}. 

Another conjecture, this time due to Carleson and Jones \cite{MR1162188} and supported by some numerically evidence, concerns the particular value $t=1$ and  is
\be\label{CJ}
\mathrm B(1)=\frac{1}{4}.
\ee
The best results known to date,  including computer assistance, are  from Refs. \cite{MR1162188} and \cite{MR2130416}  respectively, $0.21\leq B(1)\leq 0.46$, 
improving an earlier result  obtained without that assistance, $B(1)\leq 0.4884$ \cite{MR1469793}.   

Finally, the above conjectures are encompassed by the well-known \emph{Kraetzer conjecture} \cite{MR1427159} that, 
\be\label{KA}
\mathrm B_{\mathcal K}(t)=\frac{t^2}{4}, \quad t\in [-2,2],
\ee
which is supplemented by result \eqv(t-1) for $t\geq 2$, conjecturally extended as $|t|-1$ for $t\leq -2$.

For complex values of $t$, this conjecture was extended by I. Binder \cite{Bin09}, under the simple form,
\be\label{KcA}
{\mathrm B}_{{\mathcal K}}(t)=
\begin{cases}\frac{1}{4}|t|^2,  \,\,\,t\in \mathbb C, \,\,|t| \leq 2,&\\ 
|t|-1, \quad\quad  \quad 2\leq |t|.&
\end{cases}
\ee

In the real case \eqv(K), the Legendre transform \eqv(-1.5bis) and \eqv(-1.6bis), once  extended to the range of parameters $t\in [-2,2]$, yields the corresponding universal multifractal spectrum
\bea
\label{Falpha}
\mathrm F_{\mathcal K}(\a)&=&2-\frac{1}{\a}, \quad  \,\,\a \in \left[\frac{1}{2}, \infty\right),\\ \nonumber
\a&=&\frac{2}{2-t}, \quad \,\, t\in [-2,2].
\eea
Of course, on has $\mathrm F_{\mathcal K}(1)=1$, while the Carleson-Jones conjecture \eqv(CJ) translates into 
\be \label{FK2}
\mathrm F_{\mathcal K}(\a=2)=3/2. 
\ee
\subsection{Comparison to SLE spectra}
Notice that in the SLE integral means spectrum \eqv(-1.2A),  the quantity $d=d(\kappa)$ \eqv(bkA), being invariant under the duality $\kappa\to 16/\kappa$, is bounded below by $d(4)=2$, yielding  $3/2\leq t_2$. Thus, the parameter value $t=1$ is always in the range of validity of  \eqv(-1.2A)  for SLE$_\kappa$. This yields $\b(1,\kappa)=2d-1-2\sqrt{d(d-1)}$, which is decreasing function of $d$, thus $\b(1,\kappa)\leq \b(1,\kappa=4)=3-2\sqrt{2}<1/4$, in agreement with the Carleson-Jones conjecture \eqv(CJ). In the $f(\a)$ formalism, one also has agreement with \eqv(FK2), since
 $$\underset{\a\geq 1/2}{\sup}f(\a,\kappa)\leq 3/2=\mathrm F_{\mathcal K}(2),$$ 
 the upper bound being attained for $\kappa=4$, in the limit $\a\to +\infty$. 

Near the origin, one has $\b(t,\kappa)\sim t^2/4b, t\to 0$, in agreement with \eqv(KA), since $b\geq 2$. As seen above, Makarov's theorem \eqv(Mak1) is obeyed by all the SLE multifractal spectra \eqv(-1.4), 
but their $\a$-derivatives at $\a=1$ are also all equal to 1, namely $f'(1,\kappa)=1$ for any $\k >0$. Since no $f(\a)$ spectrum can cross the trace of the universal spectrum at the point $\mathrm F(\a=1)=1$, the latter, as well as any (differentiable) multifractal $f(\a)$ spectrum,  must be cotangent at $\a=1$ to the SLE spectra, with a universal slope $1$, and this, independently of any conjecture on the precise form the universal spectrum. Interestingly, this is  the case for the Kraetzer 
spectrum \eqv(Falpha), for which one indeed has $\mathrm F'_{\mathcal K}(1)=1$.


\def\cprime{$'$}


\end{document}